\documentclass[10pt]{article}%
\usepackage{marvosym}

\usepackage[hide]{ed}

\usepackage{tikz}
\usepackage{pgf}
\usetikzlibrary{arrows}
\usetikzlibrary{calc}
\usetikzlibrary{decorations.pathmorphing}
\usepackage{xfrac}    
\usepackage[strict]{changepage}

\usepackage{faktor}
\usepackage{soul}
\setstcolor{red}
\setulcolor{red}

\usepackage{rotating} % needed for \begin{sidewaysfigure}\begin{sidewaystable}

\usepackage{mathrsfs}
\DeclareMathAlphabet{\mathpzc}{OT1}{pzc}{m}{it}

\usepackage{romannum}

\usepackage{color}
\usepackage{amsmath}
\usepackage{graphicx}
\usepackage{amsfonts}
\usepackage{amssymb}%
\usepackage{latexsym}
\usepackage{psfrag}
\usepackage{subfigure}
\usepackage{accents}

\newtheorem{theorem}{Theorem}[section]
\newtheorem{corollary}[theorem]{Corollary}
\newtheorem{definition}[theorem]{Definition}
\newenvironment{proof}[1][Proof]{\noindent \emph{#1.} }
{\hfill \ \rule{0.5em}{0.5em}}
\newtheorem{lemma}[theorem]{Lemma}
\newtheorem{proposition}[theorem]{Proposition}

\newtheorem{assumption}[theorem]{Assumption}

\newtheorem{manualtheoreminner}{Theorem}
\newenvironment{manualtheorem}[1]{%
  \renewcommand\themanualtheoreminner{#1}%
  \manualtheoreminner
}{\endmanualtheoreminner}

\numberwithin{equation}{section}
\numberwithin{table}{section}
\numberwithin{figure}{section}

\usepackage[a4paper]{geometry}
\newtheorem{remark}[theorem]{Remark}
\newtheorem{example}[theorem]{Example}

\newcommand{\cA}{{\cal A}}

\newcommand{\cH}{{\cal H}}

\newcommand{\cS}{{\cal S}}

\newcommand{\cD}{{\cal D}}

\newcommand{\cO}{{\cal O}}

\newcommand{\bx}{x}%\mathbf{x}}

\newcommand{\bn}{n}%\mathbf{n}}

    \newcommand\quotient[2]{
        \mathchoice
            {% \displaystyle
                \text{\raise1ex\hbox{$#1$}\Big/\lower1ex\hbox{$#2$}}%
            }
            {% \textstyle
                #1\,/\,#2
            }
            {% \scriptstyle
                #1\,/\,#2
            }
            {% \scriptscriptstyle  
                #1\,/\,#2
            }
    }

\newcommand{\re}{{\rm e}}
\newcommand{\ri}{{\rm i}}
\newcommand{\rd}{{\rm d}}

\newcommand{\HoDk}{H^1_k(B_R\cap \obstacle_+)}

\newcommand{\beq}{\begin{equation}}
\newcommand{\eeq}{\end{equation}}
\newcommand{\beqs}{\begin{equation*}}
\newcommand{\eeqs}{\end{equation*}}
\newcommand{\bit}{\begin{itemize}}
\newcommand{\eit}{\end{itemize}}
\newcommand{\ben}{\begin{enumerate}}
\newcommand{\een}{\end{enumerate}}
\newcommand{\bal}{\begin{align}}
\newcommand{\eal}{\end{align}}
\newcommand{\bals}{\begin{align*}}
\newcommand{\eals}{\end{align*}}
\newcommand{\bse}{\begin{subequations}}
\newcommand{\ese}{\end{subequations}}
\newcommand{\bpr}{\begin{proposition}}
\newcommand{\epr}{\end{proposition}}
\newcommand{\bre}{\begin{remark}}
\newcommand{\ere}{\end{remark}}
\newcommand{\bpf}{\begin{proof}}
\newcommand{\epf}{\end{proof}}
\newcommand{\ble}{\begin{lemma}}
\newcommand{\ele}{\end{lemma}}
\newcommand{\bco}{\begin{corollary}}
\newcommand{\eco}{\end{corollary}}
\newcommand{\bex}{\begin{example}}
\newcommand{\eex}{\end{example}}
\newcommand{\bth}{\begin{theorem}}
\newcommand{\enth}{\end{theorem}}

\newcommand{\Rea}{\mathbb{R}}
\newcommand{\Com}{\mathbb{C}}

\newcommand{\GR}{{\partial B_R}}

\newcommand{\eps}{\varepsilon}

\newcommand{\pdiff}[2]{\frac{\partial #1}{\partial #2}}

\newcommand{\gu}{\nabla u}

\newcommand{\vb}{\overline{v}}
\newcommand{\gvb}{\overline{\nabla v}}

\newcommand{\tendi}{\rightarrow \infty}

\def\XXint#1#2#3{{\setbox0=\hbox{$#1{#2#3}{\int}$}
     \vcenter{\hbox{$#2#3$}}\kern-.5\wd0}}

\usepackage{hyperref}
\definecolor{myblue}{rgb}{0,0,0.6}
\hypersetup{colorlinks=true,linkcolor=myblue,citecolor=myblue,filecolor=myblue,urlcolor=myblue}
\newcommand*{\N}[1]{\left\|#1\right\|}

\allowdisplaybreaks[4]
\newcommand{\tfa}{\text{ for all }}
\newcommand{\tfor}{\text{ for }}

\newcommand{\tif}{\text{ if }}

\newcommand{\ton}{\text{ on }}

\newcommand{\tand}{\text{ and }}
\newcommand{\tst}{\text{ such that }}
\newcommand{\tfind}{\text{ find }}
\newcommand{\tthen}{\text{ then }}

\newcommand{\Hilb}{\cH}

\newcommand{\vertiii}[1]{{\left\vert\kern-0.25ex\left\vert\kern-0.25ex\left\vert #1
    \right\vert\kern-0.25ex\right\vert\kern-0.25ex\right\vert}}

\newcommand{\DtN}{{\rm DtN}_k}

\definecolor{jwcol}{RGB}{27, 137, 18}  %{rgb}{1,0.88,0.21} changed color for visibility (david)

\definecolor{dalcol}{rgb}{0.8,0,0}

\definecolor{escol}{rgb}{0,0,0.8}
\definecolor{estcol}{rgb}{0,0.5,0}
\definecolor{esnewcol}{rgb}{0,0.5,0}

\newcommand{\obstacle}{{\cO}}

\newcommand{\Imag}{{\rm Im}}

\newcommand{\supp}{{\rm supp}}

\newcommand{\Cosc}{C_{\rm{osc}}}

\newcommand{\abs}[1]{{\left\lvert{#1}\right\rvert}}
\newcommand{\norm}[1]{{\left\lVert{#1}\right\rVert}}
\newcommand{\ang}[1]{{\left\langle{#1}\right\rangle}}

\newcommand{\RR}{\mathbb{R}}
\newcommand{\ZZ}{\mathbb{Z}}

\newcommand{\OR}{B_R}

\newcommand{\truncbound}{\partial B_R}

\newcommand{\mymatrix}[1]{\mathsf{#1}}
\newcommand{\MA}{{\mymatrix{A}}}

\newcommand{\widehatx}{\widehat{\bx}}

\newcommand{\Cqo}{{C_{\rm qo}}}

\newcommand{\Csol}{{C_{\rm sol}}}

\newcommand{\Ccont}{{C_{\rm cont}}}

\newcommand{\Cint}{{C_{\rm int}}}

\newcommand{\CDTN}{{C_{\rm DtN}}}

\newcommand{\Csob}{{C_{\rm Sob}}}

\newcommand{\mythmname}[1]{\textbf{\emph{(#1)}}}

\newcommand{\uhigh}{u_{H^2}}
\newcommand{\ulow}{u_{\mathcal A}}
\newcommand{\vhigh}{v_{H^2}}
\newcommand{\vlow}{v_{\mathcal A}}

\newcommand{\Pilow}{\Pi_{\rm Low}}%\flat}}
\newcommand{\Pihigh}{\Pi_{\rm High}}%\sharp}}
\newcommand{\Lap}{\Delta}
\newcommand{\hsc}{\hbar}

\newcommand{\WFh}{\operatorname{WF}_{\hsc}}
\newcommand{\pa}{\partial}

\newcommand{\Op}{{\rm Op}}

\newcommand{\hilbert}{\mathcal{H}}
\newcommand{\domain}{\mathcal{D}}
\newcommand{\torus}{\mathbb{T}}
\newcommand{\CI}{C^\infty}
\newcommand{\F}{\mathcal{F}}
\newcommand{\Tbar}{\overline{T}}
\DeclareMathOperator{\RC}{\mathsf{RC}}

\newcommand{\RlocA}{R_{_{\rm  \Romannum{4}}}}
\newcommand{\RlocB}{R_{_{\rm \Romannum{3}}}}
\newcommand{\RfarA}{R_{_{\rm \Romannum{1}}}}
\newcommand{\RfarB}{R_{_{\rm \Romannum{2}}}}

\newcommand{\residual}{O(\hsc^\infty)_{\mathcal D_\hsc^{\sharp,-\infty} \rightarrow \mathcal D_\hsc^{\sharp,\infty}}}

\newcommand{\uH}{u_{\rm High}}
\newcommand{\uL}{u_{\rm Low}}

\newcommand{\subsetH}{\frak H}
\newcommand{\EMZ}{\tau}
\newcommand{\Optorus}{\operatorname{Op}_\hsc^{\!\!\mathbb T^d_{R_{\sharp}}}}

\makeatletter
\newcommand{\settheoremtag}[1]{% \settheoremtag{<tag>}
  \let\oldthetheorem\thetheorem% Store \thetheorem
  \renewcommand{\thetheorem}{#1}% Redefine it to a fixed value
  \g@addto@macro\endtheorem{% At \end{theorem}, ...
    \addtocounter{theorem}{-1}% ...restore theorem counter value and...
    \global\let\thetheorem\oldthetheorem}% ...restore \thetheorem
  }
\makeatother

\usetikzlibrary{positioning}

\begin{document}

\pagenumbering{arabic}
\title{Decompositions of high-frequency Helmholtz solutions via functional calculus, and application to the finite element method}

\author{J.~Galkowski\footnotemark[1]\,\,, D.~Lafontaine\footnotemark[2]\,\,, E.~A.~Spence\footnotemark[3]\,\,, J.~Wunsch\footnotemark[4]}

\date{\today}

\footnotetext[1]{Department of Mathematics, University College London, 25 Gordon Street, London, WC1H 0AY, UK,     \tt J.Galkowski@ucl.ac.uk}
\footnotetext[2]{CNRS and Institut de Mathématiques de Toulouse, UMR5219;
Universit\'e de Toulouse, CNRS; UPS, F-31062 Toulouse Cedex 9, France; \tt david.lafontaine@math.univ-toulouse.fr}
\footnotetext[3]{Department of Mathematical Sciences, University of Bath, Bath, BA2 7AY, UK, \tt E.A.Spence@bath.ac.uk }
\footnotetext[4]{Department of Mathematics, Northwestern University, 2033 Sheridan Road, Evanston IL 60208-2730, US, \tt jwunsch@math.northwestern.edu}

\maketitle

%\begin{AMS}
%35J05, 65N15, 65N30, 78A45
%\end{AMS}

\begin{abstract}
Over the last ten years, results 
from \cite{MeSa:10}, \cite{MeSa:11}, \cite{EsMe:12}, and \cite{MePaSa:13}
decomposing high-frequency Helmholtz solutions into ``low"- and ``high"-frequency components have had a large impact in the numerical analysis of the Helmholtz equation. These results have been proved for the constant-coefficient Helmholtz equation in either the exterior of a Dirichlet obstacle or an interior domain with an impedance boundary condition. 

Using the Helffer--Sj\"ostrand functional calculus
\cite{HeSj:89}, this paper proves analogous decompositions for scattering problems fitting into the black-box scattering framework of Sj\"ostrand--Zworski \cite{SjZw:91}, thus covering Helmholtz problems with variable coefficients, impenetrable obstacles, and
penetrable obstacles all at once. 

These results allow us to prove new frequency-explicit convergence results for (i) the $hp$-finite-element method ($hp$-FEM) applied to the variable-coefficient Helmholtz equation in the exterior of an analytic Dirichlet obstacle, where the coefficients are analytic in a neighbourhood of the obstacle, and (ii) the $h$-FEM applied to the Helmholtz penetrable-obstacle transmission problem. In particular, the result in (i) shows that the $hp$-FEM applied to this problem does not suffer from the pollution effect.

%\paragraph{Keywords:}
%Helmholtz equation, high frequency, pollution effect, finite element method, error estimate, semiclassical analysis.

\end{abstract}

%\begin{AMS}
%35J05, 65N15, 65N30, 78A45
%\end{AMS}

%\section*{Main changes Euan made}
%\bit
%\item We assume that the coefficients of $Q_\hsc$ in \S\ref{subsec:bb} are smooth, but then previously required them to be Lipschitz in Lemma \ref{lem:obstacle}. Solution -- have now assumed coefficients (and obstacle) are $C^\infty$ from the start in the exterior Dirichlet problem.
%\eit
%
%\section*{To discuss}
%\bit
%\item Lemma \ref{thm:funcloc2} 
%\item reasoning that $u_{\cA} \in \cD^{\sharp,\infty}_\hsc$.
%\eit

\section{Introduction}\label{subsec:intro}

\subsection{Context: the results of \cite{MeSa:10}, \cite{MeSa:11}, \cite{EsMe:12}, \cite{MePaSa:13} and their impact on numerical analysis of the Helmholtz equation.}
\label{sec:1.1}

At the heart of the papers 
 \cite{MeSa:10}, \cite{MeSa:11}, \cite{EsMe:12}, and \cite{MePaSa:13}
are results that decompose solutions of the high-frequency Helmholtz equation, i.e., 
\beq\label{eq:Helmholtz}
\Delta u +k^2 u =-f 
\eeq 
with $k$ large, into 
\bit
\item[(i)] a component with $H^2$ 
regularity, satisfying bounds with improved $k$-dependence compared to those satisfied by the full Helmholtz solution, and
\item[(ii)] an analytic component, satisfying bounds with the same $k$-dependence as those satisfied by the full Helmholtz solution, 
\eit
with these components corresponding to the ``high"- and ``low"-frequency components of the solution. 
In the rest of this paper, we write this decomposition as  $u=\uhigh+\ulow$.

Such a decomposition was obtained for
\bit
\item the Helmholtz equation \eqref{eq:Helmholtz} posed in $\Rea^d$, $d=2,3$, with compactly-supported $f$, and 
with the Sommerfeld radiation condition
\beq\label{eq:src}
\pdiff{u}{r}(\bx) - \ri k u(\bx) = o \left( \frac{1}{r^{(d-1)/2}}\right)
\eeq
as $r:= |\bx|\tendi$, uniformly in $\widehatx:= \bx/r$ \cite[Lemma 3.5]{MeSa:10},
\item the Helmholtz exterior Dirichlet problem where the obstacle has analytic boundary \cite[Theorem 4.20]{MeSa:11}, and 
\item the Helmholtz interior impedance problem where the domain is either analytic ($d=2,3$)  \cite[Theorem 4.10]{MeSa:11},
 \cite[Theorem 4.5]{MePaSa:13},
 or polygonal \cite[Theorem 4.10]{MeSa:11}, \cite[Theorem 3.2]{EsMe:12},
\eit
in all cases under an assumption that the solution operator grows at most polynomially in $k$ (which has recently been shown to hold, for most frequencies, for a variety of scattering problems in \cite{LaSpWu:20}).

These decompositions have had a large impact in the numerical analysis of the Helmholtz equation in that they allow one to prove convergence, explicit in the frequency $k$, of so-called $hp$-finite-element methods ($hp$-FEM) applied to discretisations of the Helmholtz equation. Recall that the $hp$-FEM approximates solutions of PDEs by piecewise polynomials of degree $p$ on a mesh with meshwidth $h$ and obtains convergence by both decreasing $h$ and increasing $p$; this is in contrast to the $h$-FEM where $p$ is fixed and only $h$ decreases.

Indeed, these decompositions were used to prove frequency-explicit convergence of a variety of $hp$ methods in 
 \cite{MeSa:10, MeSa:11,EsMe:12,MePaSa:13, %}, \cite%[Lemmas 3.1 and 3.3]
 ZhWu:13,
%\cite%[Lemma 3.1 and Appendix A1]
ZhDu:15,
%\cite%[Lemmas 1 and 5]
DuZh:16,
%\cite%[Lemma 3.1]
BeMe:18}.
%and $h$ methods in 
%%$h$ methods with large but fixed $p$
%\cite%[Lemmas 3.1 and 3.3]
%{DuWu:15}.
These results about $hp$ methods are particularly significant, since they show that, if $h$ and $p$ are chosen appropriately, the FEM solution is uniformly accurate as $k\tendi$ with
the total number of degrees of freedom proportional to $k^d$; i.e., the $hp$-FEM does not suffer from the so-called pollution effect (i.e.~the total number of degrees of freedom needing to be $\gg k^d$) which plagues the $h$-FEM \cite{BaSa:00}.

 These decompositions were also used to prove sharp results about the convergence of $h$ methods with large but fixed $p$
\cite{EsMe:14}, \cite%[Lemmas 3.1 and 3.3]
{DuWu:15}, \cite{LaSpWu:19}.
Furthermore, analogous decompositions and analogous convergence results were obtained for 
$hp$-boundary-element methods
\cite{Me:12}, \cite{LoMe:11}, $hp$ methods applied to Helmholtz problems with arbitrarily-small dissipation \cite{MeSaTo:20} 
and $hp$ methods applied to formulations of the time-harmonic Maxwell equations
\cite{MeSa:20}, \cite{NiTo:20}. This work has also motivated attempts to provide simpler decompositions valid for a variety of variable-coefficient problems  \cite{ChNi:20}.

The recent paper \cite{LaSpWu:22} obtained the analogous decomposition to that in \cite{MeSa:10}  for the Helmholtz problem in $\Rea^d$ but now for the variable-coefficient Helmholtz equation
\beq\label{eq:Helmholtzvar}
\nabla \cdot(A\nabla u) + \frac{k^2}{c^2}u=-f
\eeq
with $A$ and $c \in C^\infty$.  The goal of the present paper is to obtain decompositions for more-general Helmholtz problems.

\subsection{Informal statement of the main results}%\label{sec:informal_statement}

We show a decomposition of the form $u = u_{H^2} + u_{\mathcal A}$ for the solutions of  the following three Helmholtz problems.
\begin{itemize}
\item[(P1)] The $C^\infty$-variable-coefficient Helmholtz exterior Dirichlet problem  where the obstacle has analytic boundary and the coefficients are analytic near the obstacle. The corresponding result, discussed in \S \ref{sec:mainintro} below, is  stated as Theorem \ref{thm:Dirichlet}, and applied to prove
quasi-optimality of the $hp$-FEM in Theorem \ref{thm:FEM1}. In particular, Theorem \ref{thm:FEM1} shows that the $hp$-FEM applied to this Helmholtz problem does not suffer from the pollution effect.
\item[(P2)] The transmission problem with finite regularity of the interface and the coefficients - that is, the problem of scattering by a penetrable obstacle. This result is discussed in \S \ref{sec: transmission}, where it is stated as Theorem \ref{thm:transmission}, and applied to prove 
quasi-optimality of the $h$-FEM in Theorem \ref{thm:FEM2}.
\item[(P3)] The $C^\infty$-variable-coefficient Helmholtz equation in the full space $\mathbb R^d$; this situation was studied in \cite{LaSpWu:22} and we recover the results of \cite{LaSpWu:22} with the more general method presented here; see \S \ref{sec:LSW3} and Theorem \ref{thm:LSW3}. In \S\ref{sec:informal} we discuss the ideas behind both \cite{LaSpWu:22} and the present method, and the relationship between them.
\end{itemize}

We highlight that, just as in the earlier works \cite{MeSa:10}, \cite{MeSa:11}, \cite{EsMe:12}, and \cite{MePaSa:13}, $\uhigh$ and $\ulow$ 
correspond to ``high'' and ``low'' frequencies of the solution, respectively -- this is discussed further in the informal discussion in \S\ref{sec:informal}.

The three results stated outlined above are obtained as applications of a single, more general, albeit abstract result, Theorem \ref{thm:mainbb} below. This theorem is  stated using the black-box framework of Sj\"ostrand--Zworski \cite{SjZw:91}, and  covers Helmholtz problems with variable coefficients, impenetrable obstacles, and
penetrable obstacles all at once.   We postpone the rigorous statement of Theorem \ref{thm:mainbb} to \S \ref{sec:mainresult} and give an informal version of it here.

\begin{manualtheorem}{A$^{\prime}$}[Informal statement of our main general result] 
Let $P$ be a 
formally self-adjoint operator with $P = -\Delta$ outside $B(0,R_0)$ (``the black-box''). We assume that $P -k^2$ is well defined and that
\begin{itemize}
    \item[(H1)] the solution operator associated with $P-k^2$ is polynomially bounded:~there exists $M\geq 0$ so that for any $\chi \in C^\infty_{\rm comp}$
    and any compactly-supported $g\in L^2$, the outgoing solution of $(P-k^2)u = g$ satisfies
    $$
    \Vert \chi u \Vert_{L^2} \lesssim  k^M \Vert g\Vert_{L^2},
    $$
    \item[(H2)] one has an estimate quantifying the regularity of $P$ 
    inside $B(0,R_0)$ (i.e., ``inside the black-box'').
\end{itemize}
Then, for any $R>R_0$, any solution of $(P-k^2) u = g$ splits as
$$
u|_{B(0, R)} = u_{H^2} + u_{\mathcal A},
$$
where
\begin{itemize}
    \item[(i)] $u_{H^2}$ satisfies
    $$\Vert u_{H^2}\Vert_{L^2} + k^{-2}\Vert Pu_{H^2}\Vert_{L^2} \lesssim \Vert g \Vert_{L^2},$$
    \item[(ii)] $u_{\mathcal A}$ is \emph{regular}, with an estimate depending on \emph{both} the regularity of the underlying problem (as measured by (H2)) \emph{and} $M$. In addition, the part of $u_{\mathcal A}$ away from ``the black-box'' $B(0,R_0)$ is entire (in the sense of Lemma \ref{lem:analytic}(i) below).
\end{itemize}
\end{manualtheorem}

When $P$ is the Dirichlet Laplacian, for example, $\|Pu_{H^2}\|_{L^2}$ controls $\|u_{H^2}\|_{H^2}$ by elliptic regularity, and thus the bound in (i) is a bound on $\|u_{H^2}\|_{H^2}$ (hence the notation $u_{H^2}$).

The paper \cite{LaSpWu:19} shows that Assumption (H1) holds in the black-box framework for ``most'' frequencies (see Part (ii) of Theorem \ref{thm:polyboundD} for a more precise statement of this).
The key point, therefore, to apply this result to specific situations is to check that an estimate of the type (H2) holds. 
In the three applications  to problems (P1), (P2), and (P3) above, this estimate (H2) corresponds to, respectively, a heat-flow estimate, an elliptic estimate, and regularity of the eigenfunctions of the Laplace operator on the torus. Theorem \ref{thm:mainbb} could be applied to a range of other specific situations, provided an estimate of type (H2) is at hand.
For a reader
%We highlight the fact that the reader mainly 
interested in applying Theorem  \ref{thm:mainbb} without going into the details of the proof,
%such applications doesn't need to 
%understand fully the machinery behind the general result, and we refer them to 
\S \ref{sec:howto} gives a short summary on how to do this. %apply our result.

Before stating the main result applied the problems (P1), (P2), and (P3) above, we record the following lemma about 
how the bound an analytic function depending on $k$ satisfies dictate the $k$-dependence of the region of analyticity; we use this below to understand the properties of the $\ulow$s in (P1) and (P3).
%analyticity of functions depending on a parameter $k$.

\ble[$k$-explicit analyticity]\label{lem:analytic}
Let $u\in C^\infty(D)$ be a family of functions depending on $k$.

(i) If 
there exist $C, C_u>0$, independent of $\alpha$, such that 
%$u \in C^\infty(D)$ with 
\beqs%\label{eq:analytic1}
\N{\partial^\alpha u}_{L^2(D)}\leq C_u (C k)^{|\alpha|}\quad\text{ for all multiindices } \alpha,
\eeqs
%(with $C, C_u$ independent of $\alpha$)
then $u$ is real analytic in $D$ with infinite radius of convergence, i.e., $u$ is entire.

(ii) If 
there exist $C, C_u>0$, independent of $\alpha$, such that 
%$u \in C^\infty(D)$ with 
\beqs%\label{eq:analytic1}
\N{\partial^\alpha u}_{L^2(D)}\leq C_u (Ck)^{|\alpha|} |\alpha|! \quad\text{ for all multiindices } \alpha,
\eeqs
%(with $C, C_u$ independent of $\alpha$)
then $u$ is real analytic in $D$ with radius of convergence proportional to $(C k)^{-1}$. 

(iii) If 
there exist $C, C_u>0$, independent of $\alpha$, such that 
%$u \in C^\infty(D)$ with 
\beqs%\label{eq:analytic1}
\N{\partial^\alpha u}_{L^2(D)}\leq C_u C^{|\alpha|}\max \big\{ |\alpha|, k \big\}^{|\alpha|} \quad\text{ for all multiindices } \alpha,
\eeqs
%(with $C, C_u$ independent of $\alpha$ and $k$), 
then $u$ is real analytic in $D$ with radius of convergence independent of $k$.

%(iiv) If %$u \in C^\infty(D)$ with 
%there exist $C, C_u>0$, independent of $\alpha$ and $k$, such that 
%\beqs%\label{eq:analytic1}
%\N{\partial^\alpha u}_{L^2(D)}\leq C_u  C^{|\alpha|} |\alpha|^{|\alpha|/2}\max \big\{ |\alpha|^{|\alpha|/2}, k^{|\alpha|} \big\} \quad\text{ for all multiindices } \alpha,
%\eeqs
%%(with $C, C_u$ independent of $\alpha$ and $k$) 
%then $u$ is real analytic in $D$ with radius of convergence independent of $k$.
%%\ele
\ele

\bpf
In each case, we use the Sobolev embedding theorem to obtain a bound on $\|\partial^\alpha u \|_{L^\infty(D)}$, and then sum the 
remainder in the truncated Taylor series. For this procedure carried out in Case (iii), see, e.g., \cite[Proof of Lemma C.2]{MeSa:10}; the proofs for the other cases are similar.
\epf

%as soon as one understand the regularity of the underlying problem as measured by (H2). Indeed, the assumption (H1) is verified in pratice at least for most frequencies in most Helmholtz scattering problems \cite{LaSpWu:19}, and the key point in order to apply our result to practical situations is to exhibit an estimate of the type (H2). Examples of such estimates are estimates on the heat-flow, an elliptic estimates, estimates on the Poisson kernel,  on the (compactified) underlying operator eigenfunctions, and so on. 

\subsection{The main result applied to the exterior Dirichlet problem}\label{sec:mainintro}

\subsubsection{Background definitions}

\begin{definition}[Exterior Dirichlet problem]\label{def:EDP}
Let $\obstacle_-\subset\Rea^d$, $d\geq 2$ be a bounded open set such that $\partial \obstacle_+$ is smooth, the open complement $\obstacle_+:=\Rea^d\setminus\overline{\obstacle_-}$ is connected, and $\obstacle_- \subset B_{R_0}$.  
Let  $A \in C^{\infty} (\obstacle_+ , \Rea^{d\times d})$ be such that $\supp(I -A)\subset B_{R_1}$,
with $R_1>R_0$, 
$A$ is symmetric, and there exists $A_{\min}>0$ such that
\beq\label{eq:Aelliptic}
\big(A(\bx) \xi\big) \cdot\overline{ \xi} 
\geq  A_{\min} |\xi|^2 \quad \tfa x \in \obstacle_+ \text{ and for all }  \xi \in \Com^d.
\eeq
Let $c\in C^\infty(\obstacle_+)$ be such that $\operatorname{supp}(1-c) \subset B_{R_1}$, and $c_{\rm min}\leq c \leq c_{\rm max}$ with $c_{\rm min}, c_{\rm max} >0$.

Given $f \in L^2(\obstacle_+)$ with $\supp f \Subset \Rea^d$ and $k>0$, $u \in H^1_{\rm loc}(\obstacle_+)$ satisfies the exterior Dirichlet problem if 
\begin{align}\label{eq:HelmholtzD}
c^2\nabla \cdot (A\nabla u) + k^2u&= -f \,\quad\text{ in } \obstacle_+,\\
 u &= 0 \qquad\text{ on }\partial\obstacle_+, \label{eq:Dirtr0}
\end{align}
and $u$ satisfies the Sommerfeld radiation condition \eqref{eq:src}.
\end{definition}

We highlight from Definition \ref{def:EDP} that the obstacle $\obstacle_-$ is contained in $B_{R_0}$, and the variation of the coefficients $A$ and $c$ is contained inside the larger ball $B_{R_1}$.

We use the standard weighted $H^1$ norm, $\|\cdot\|_{H^1_k(B_R\cap\obstacle_+)}$, defined by
\begin{equation} \label{eq:1knorm}
\|u\|^2_{H^1_k(B_R\cap\obstacle_+)} :=\N{\nabla u}_{L^2(B_R\cap\obstacle_+)}^2 + k^2 \N{u}_{L^2(B_R\cap\obstacle_+)}^2.
\eeq

\begin{definition}[$\Csol$]\label{def:Csol}
Given $f\in L^2(\obstacle_+)$ supported in $B_R$ with $R\geq R_1$, let $u$ be the solution of the exterior Dirichlet problem of Definition \ref{def:EDP}.  Given $k_0>0$, let $\Csol= \Csol(k,A,c,R,k_0)>0$ be such that
\beq\label{eq:Csol}
\N{u}_{\HoDk} \leq \Csol \N{f}_{L^2(B_R\cap \obstacle_+)} \quad\tfa k>0.
\eeq
\end{definition}

$\Csol$ exists by standard results about uniqueness of the exterior Dirichlet problem and Fredholm theory; see, e.g., \cite[\S1]{GrPeSp:19} and the references therein.
How $\Csol$ depends on $k$ is crucial to our analysis, and to emphasise this we write $\Csol=\Csol(k)$. 
A key assumption in our analysis is that $\Csol(k)$ is polynomially bounded in $k$ in the following sense.

\begin{definition}[$\Csol$ is polynomially bounded in $k$]\label{def:polybound}
Given $k_0$ and $K \subset [k_0,\infty)$, $\Csol(k)$ is polynomially bounded for $k\in K$ if 
 there exists $C>0$ and $M\geq 0$ such that
\beq\label{eq:polybound}
\Csol(k) \leq C k^M \tfa k \in K,
\eeq
where $C$ and $M$ are independent of $k$ (but depend on $k_0$ and possibly also on $K, A, c, d, R$).
\end{definition}

There exist $C^\infty$ coefficients $A$ and $c$ such that $\Csol(k_j) \geq C_1 \exp (C_2 k_j)$ for $0<k_1<k_2<\ldots$ with $k_j\tendi$ as $j\tendi$, see \cite{Ra:71}, but this exponential growth is the worst-possible, since 
$\Csol(k) \leq c_3 \exp( c_4 k)$ for all $k\geq k_0$ by \cite[Theorem 2]{Bu:98}.
We now recall results on when $\Csol(k)$ is polynomially bounded in $k$.

\begin{theorem}\mythmname{Conditions under which $\Csol(k)$ is polynomially bounded in $k$ for the exterior Dirichlet problem}\label{thm:polyboundD}

(i) If $A$ and $c$ are $C^\infty$ and nontrapping (i.e.~all the trajectories of the generalised bicharacteristic flow defined by the semiclassical principal symbol of \eqref{eq:HelmholtzD} starting in $B_R$ leave $B_R$ after a uniform time), then $\Csol(k)$ is independent of $k$ for 
all sufficiently large $k$; i.e., \eqref{eq:polybound} holds for all $k\geq k_0$ with $M=0$.

(ii) Under no additional assumptions on $\obstacle_-$, $A$, and $c$,
given $k_0>0$ and $\delta>0$ there exists a set $J\subset [k_0,\infty)$ with $|J|\leq \delta$ such that  
\beqs
\Csol(k) \leq C k^{5d/2%+1
+\eps} \quad\tfa k \in [k_0,\infty)\setminus J,
\eeqs
 for any $\eps>0$, where $C$ depends on $\delta, \eps, d, k_0,$ and $A$.
\end{theorem}

\bpf[References for the proof]
(i) follows from \emph{either} the results of \cite{MeSj:82} combined with either \cite[Theorem 3]{Va:75}/ \cite[Chapter 10, Theorem 2]{Va:89} or \cite{LaxPhi}, \emph{or} \cite[Theorem 1.3 and \S3]{Bu:02}. It has recently been proved that, for this situation, $\Csol$ is proportional to the length of the longest trajectory in $B_R$; see \cite[Theorems 1 and 2, and Equation 6.32]{GaSpWu:20}. 
(ii) is proved for $c=1$ in \cite[Theorem 1.1 and Corollary 3.6]{LaSpWu:20}; the proof for more-general $c$ follows from Lemma \ref{lem:obstacle} below.
\epf

\subsubsection{Theorem \ref{thm:mainbb} applied to the exterior Dirichlet problem}

\begin{manualtheorem}{B} \mythmname{Theorem \ref{thm:mainbb} applied to the exterior Dirichlet problem with analytic $\obstacle_-$
and locally analytic $A, c$}\label{thm:Dirichlet}
Suppose that $\obstacle_-, A, c,R_0$, and $R_1$ are as in Definition \ref{def:EDP}.
In addition, assume that $\obstacle_-$ is analytic, and that $A$ and $c$ are analytic in $B_{R_*}$ for some $R_0<R_*<R_1$.

If $\Csol(k)$ is polynomially bounded for $k\in K$ (in the sense of Definition \ref{def:polybound}), then 
given $f\in L^2(\obstacle_+)$ supported in $B_{R}$ with $R\geq R_1$, the solution $u$ of the exterior Dirichlet problem is such that
there exists $u_{\mathcal A} \in C^\infty(B_R \cap \obstacle_+)$, and $u_{H^2} \in H^2(B_R \cap\obstacle_+)$, both with zero Dirichlet trace on $\partial\obstacle_+$, such that 
\beqs
u|_{B_{R}} = u_{\mathcal A} + u_{H^2}.
\eeqs
Furthermore, there exists $C_1$, independent of $k$ and $\alpha$, such that 
\begin{equation} \label{eq:decHFDir}
\Vert \partial^\alpha u_{H^2} \Vert_{L^2(B_R \cap \obstacle_+)} \leq C_1 k^{|\alpha|-2}\Vert f \Vert_{L^2(B_R \cap \obstacle_+)} \quad \tfa  k \in K \text{ and for all } |\alpha|\leq 2,
\end{equation}
and there exist $C_2,C_3, C_4$ and $C_5$, all independent of $k$ and $\alpha$, and 
$\RfarA, \RfarB, \RlocB, \RlocA$ with $R_0<\RfarA<\RfarB<\RlocB<\RlocA<R$
%\esnote{this could be  $\widetilde R$ for any $\widetilde R$ close to $R_0$ (but I suggest leaving it as it is for simplicity)}
such that $u_{\mathcal A}$ decomposes as $u_{\mathcal A} = u^{R_0}_{\mathcal A} + u^{\infty}_{\mathcal A}$, 
where $ u^{R_0}_{\mathcal A} $ is analytic in  $B_{\RlocA}$ and has zero Dirichlet trace on $\partial \obstacle_+$, and $u^{\infty}_{\mathcal A}$ is analytic in $(B_{\RfarA})^c$ with, for all $k\in K$ and all $\alpha$,  
\begin{equation} \label{eq:decLFDir1}
\Vert \partial^\alpha u^{R_0}_{\mathcal A} \Vert_{L^2(B_{\RlocA}\cap \obstacle_+) } \leq C_2  (C_3)^{|\alpha|}
\max \big\{ |\alpha|^{|\alpha|}, k^{|\alpha|}\big\} k ^{ -1+M} \,\Vert f \Vert_{L^2(B_R \cap \obstacle_+)},
\end{equation}
\begin{equation} \label{eq:decLFDir2}
\Vert \partial^\alpha u^{\infty}_{\mathcal A} \Vert_{L^2((B_{\RfarA})^c\cap \obstacle_+) } \leq C_4 (C_5)^{|\alpha|} k ^{ |\alpha| -1+M} \,\Vert f \Vert_{L^2(B_R \cap \obstacle_+)}, %\quad \tfa  k \in K \text{ and for all }   \alpha,
\end{equation}
and, for any $N,m>0$ there exists $C_{N,m}>0$ so that
\begin{equation} \label{eq:decLFDirRes}
\Vert u^{\infty}_{\mathcal A} \Vert_{H^{m}(B_{\RfarB}\cap \obstacle_+) } + \Vert u^{R_0}_{\mathcal A} \Vert_{H^{m}((B_{\RlocB})^c\cap \obstacle_+) }  \leq C_{N,m}k^{-N}  \,\Vert f \Vert_{L^2(B_R \cap \obstacle_+)}    \tfa  k \in K.
\end{equation}
\end{manualtheorem}

By Parts (iv) and (i) of Lemma \ref{lem:analytic}, $u^{R_0}_{\mathcal A}$ is analytic in $B_{\RlocA}$ with $k$-independent radius of convergence, and $u^{\infty}_{\mathcal A}$ is entire in $(B_{\RfarA})^c$; see Figure \ref{fig:line2}.

\begin{figure}
\begin{center}
    \includegraphics[scale=0.8]{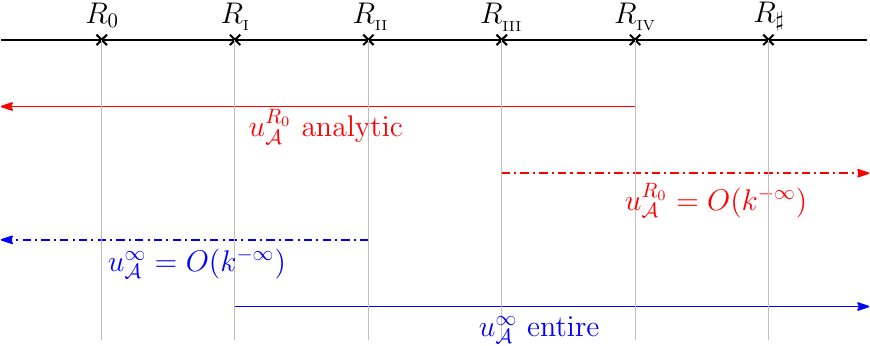}
  \end{center} 
      \caption{The regions where $u^{R_0}_{\mathcal A}$ and $u^\infty_{\mathcal A}$ appearing in Theorem \ref{thm:Dirichlet} are analytic, entire, or $O(k^{-\infty})$.} %Recall from \eqref{eq:decLF0} that $u_{\mathcal A} = u^{R_0}_{\mathcal A} + u^\infty_{\mathcal A}$.}   
\label{fig:line2}
\end{figure}

\subsubsection{Corollary about frequency-explicit convergence of the $hp$-FEM}

As discussed in \S\ref{sec:1.1}, Theorem \ref{thm:Dirichlet} implies a frequency-explicit convergence result about the $hp$-FEM applied to the exterior Dirichlet problem; we now give the necessary definitions to state this result. Recall that the FEM is based on the standard variational formulation of the exterior Dirichlet problem: 
Let 
\beqs
H^1_{0,\partial \obstacle_+}(B_R\cap \obstacle_+):= \Big\{ v \in H^1(B_R\cap \obstacle_+) \text{ with } v=0 \text{ on } \partial \obstacle_+\Big\}.
\eeqs
Given $R\geq R_1$ and $F\in (H^1_{0,\partial \obstacle_+}(B_R\cap\obstacle_+))^*$, 
\beq\label{eq:vf}
\text{ find } u \in H^1_{0,\partial \obstacle_+}(B_R\cap\obstacle_+) \tst \,\, a(u,v)=F(v) \,\, \tfa v\in H^1_{0,\partial \obstacle_+}(B_R\cap\obstacle_+),
\eeq
where
\begin{align}\label{eq:sesqui}
a(u,v)&:= \int_{B_R\cap\obstacle_+} 
\left((A \nabla u)\cdot\overline{\nabla v}
 - \frac{k^2}{c^2} u\overline{v}\right) - \big\langle \DtN (u), v\big\rangle_{\GR},
 \end{align}
where $\langle\cdot,\cdot\rangle_{\GR}$ denotes the duality pairing on $\GR$ that is linear in the first argument and antilinear in the second, and 
$\DtN: H^{1/2}(\truncbound) \rightarrow H^{-1/2}(\truncbound)$ is the Dirichlet-to-Neumann map for the equation $\Delta u+k^2 u=0$ posed in the exterior of $B_R$ with the Sommerfeld radiation condition \eqref{eq:src}; the definition of $\DtN$ in terms of Hankel functions and polar coordinates (when $d=2$)/spherical polar coordinates (when $d=3$) is given in, e.g., \cite[Equations 3.7 and 3.10]{MeSa:10}. 
We use later the fact that there exist $\CDTN=\CDTN(k_0 R_0)$ such that 
\beq\label{eq:CDTN1}
\big|\big\langle \DtN(u), v\rangle_{\partial B_R}\big\rangle\big| \leq \CDTN \N{u}_{H^1_k(B_R\cap \obstacle_+)}  \N{v}_{H^1_k(B_R\cap \obstacle_+)} 
\eeq
for all $u,v \in H^1_{0,\partial \obstacle_+}(B_R\cap \obstacle_+)$ and for all $k\geq k_0$; see \cite[Lemma 3.3]{MeSa:10}. 

If $F(v)= \int_{B_R \cap \obstacle_+} f v$, then the solution of the variational problem \eqref{eq:vf} is the restriction to $B_R$ of the solution of the exterior Dirichlet problem of Definition \ref{def:EDP}. If
\beq\label{eq:planewavedata}
F(v) = \int_{\partial B_R} \big(\partial_n u^I - \DtN (u^I)\big) \overline{v},
\eeq
 where $u^I$ is a solution of $\Delta u^I +k^2 u^I=0$ in $B_R\cap\obstacle_+$, then the solution of the variational problem \eqref{eq:vf} is the restriction to $B_R\cap\obstacle_+$ of the sound-soft scattering problem (see, e.g, \cite[Page 107]{ChGrLaSp:12}).

Given a sequence, $\{V_N\}_{N=0}^\infty$, of finite-dimensional subspaces of $H^1_{0,\partial \obstacle_+}(B_R\cap \obstacle_+)$, the finite-element method for the variational problem \eqref{eq:vf} is the Galerkin method applied to the variational problem \eqref{eq:vf}, i.e.,
\beq\label{eq:FEM}
\text{ find } u_N \in V_N \tst\,\, a(u_N,v_N)=F(v_N) \,\, \tfa v_N\in  V_N.
\eeq
\begin{manualtheorem}{B1}[Quasioptimality of $hp$-FEM for the exterior Dirichlet problem]\label{thm:FEM1}
Let $d=2$ or $3$. Suppose that $\obstacle_-, A, c,R, \RfarA$, and $\RlocA$ are as in Theorem \ref{thm:Dirichlet}.
Let $(V_N)_{N=0}^\infty$ be the piecewise-polynomial approximation spaces described in \cite[\S5]{MeSa:10}, \cite[\S5.1.1]{MeSa:11} (where, in particular, the triangulations are quasi-uniform, allow curved elements, and thus fit $B_R\cap \obstacle_+$ exactly).
%assume further that the triangulations fit $B_{\RfarA}$ and $B_{\RlocA}$ exactly
Let $u_N$ be the Galerkin solution defined by \eqref{eq:FEM}.

If $\Csol(k)$ is polynomially bounded (in the sense of Definition \ref{def:polybound}) for $k\in K\subset [k_0,\infty)$ then there exist $k_1, C_1,C_2>0$, depending on $A, c, R$, and $d$, but independent of $k$, $h$, and $p$, such that if 
\beq\label{eq:thresholdD}
\frac{hk}{p} \leq C_1 \quad\tand\quad p \geq C_2 \log k, 
\eeq 
then, for all $k\in K\cap[k_1,\infty)$, the Galerkin solution exists, is unique, and satisfies the quasioptimal error bound 
\beq\label{eq:qo}
\N{u-u_N}_{H^1_k(B_R\cap\obstacle_+)}\leq \Cqo \min_{v_N\in V_N} \N{u-v_N}_{H^1_k(B_R\cap\obstacle_+)},
\eeq
with 
\beq\label{eq:Cqo}
\Cqo:= \frac{2\big(\max\{ A_{\max}, c_{\min}^{-2} \} + \CDTN\big)}{A_{\min}}.
\eeq
\end{manualtheorem}

\bre\mythmname{The significance of Theorem \ref{thm:FEM1}:~the $hp$-FEM does not suffer from the pollution effect}\label{rem:FEMD}
For finite-dimensional subspaces consisting of piecewise polynomials of degree $p$ on meshes with meshwidth $h$, 
the total number of degrees of freedom $\sim (p/h)^d$. 
Therefore Theorem \ref{thm:FEM1}, as well as the results in 
\cite{MeSa:10}, \cite{MeSa:11}, \cite{EsMe:12}, \cite{MePaSa:13}, \cite{LaSpWu:22}, 
show that there is a choice of $h$ and $p$ such that the $hp$-FEM is quasioptimal with the total number of degrees of freedom $\sim k^d$. As highlighted in \S\ref{sec:1.1}, the significance of this is that 
when the total number of degrees of freedom $\sim k^d$ the $h$-FEM (i.e., with $p$ fixed) 
does not satisfy the quasioptimal error estimate \eqref{eq:qo} with $\Cqo$ independent of $k$; this is called the \emph{pollution effect} -- see \cite{BaSa:00} and the references therein.

The results in \cite{MeSa:10}, \cite{MeSa:11}, \cite{EsMe:12}, \cite{MePaSa:13} are for constant-coefficient Helmholtz problems, and 
those in \cite{LaSpWu:22} are for the Helmholtz equation with smooth variable coefficients and no obstacle.
Theorem \ref{thm:FEM1} is therefore 
the first result showing that the 
$hp$-FEM applied the Helmholtz exterior Dirichlet problem with variable coefficients does not suffer from the pollution effect.
\ere

In the specific case of the plane-wave scattering problem, the recent results of \cite[Theorem 9.1 and Remark 9.10]{LaSpWu:19} allow us to bound the best approximation error on the right-hand side of \eqref{eq:qo} and obtain a bound on the relative error. 

\begin{corollary}[Bound on the relative error of the Galerkin solution]
\label{cor:1}
Let the assumptions of Theorem \ref{thm:FEM1} hold and, furthermore, let $F(v)$ be given by \eqref{eq:planewavedata} 
with $u^I(x) = \exp(\ri k x \cdot a)$ for some $a\in \Rea^d$ with $|a|=1$ (so that $u$ is then the solution of the plane-wave scattering problem). 
Suppose $\Csol(k)$ is polynomially bounded (in the sense of Definition \ref{def:polybound}) for $k\in K\subset [k_0,\infty)$. Then 
there exists $C_3>0$, independent of $k$, $h$, and $p$, such that, with $k_1, C_1$, and $C_2$ as in Theorem \ref{thm:FEM1}, 
if %\eqref{eq:thresholdD} holds, 
$hk/p\leq C_1$ and $p\geq C_2\log k$ then, 
for all $k\in K\cap[k_1,\infty)$, 
\beq\label{eq:rel_error}
\frac{
\N{u-u_N}_{H^1_k(B_R\cap \obstacle_+)}
}{
\N{u}_{H^1_k(B_R\cap \obstacle_+)}
}
\leq \Cqo C_3
\frac{hk}{p} \left(1+ \frac{hk}{p}\right),
\eeq
with $\Cqo$ given by \eqref{eq:Cqo}; i.e.~the relative error can be made arbitrarily small by making $hk/p$ smaller. 
\end{corollary}

\subsection{The main result applied to the  transmission problem} \label{sec: transmission}

\subsubsection{Background definitions}

\begin{definition}[Transmission problem (i.e.~scattering by a penetrable obstacle)]\label{def:transmission}
Let $\obstacle_-\subset\Rea^d$, $d\geq 2$ be a bounded Lipschitz open set such that the open complement $\obstacle_+:=\Rea^d\setminus\overline{\obstacle_-}$ is connected and such that $\obstacle_- \subset B_{R_0}$.  
Let 
$A=(A_-,A_+)$ with $A_\pm \in C^{0,1} (\obstacle_\pm, \Rea^{n\times n})$ be such that $\supp(I -A)\subset B_{R_0}$, $A$ is symmetric, and there exists $A_{\min}>0$ such that
\eqref{eq:Aelliptic} holds (with $\obstacle_+$ replaced by $\Rea^d$).
Let $c\in L^\infty(\obstacle_-)$ be such that 
$c_{\rm min}\leq c \leq c_{\rm max}$ with $0<c_{\rm min}\leq c_{\rm max}<\infty$.
Let $\beta>0$.

Let $\nu$ be the unit normal vector field on $\partial \obstacle_-$  pointing from $\obstacle_-$ into $\obstacle_+$, and let 
$\partial_{\nu,A}$ denote the corresponding conormal derivative
defined by, e.g., \cite[Lemma 4.3]{Mc:00} (recall that this is such that, when $v \in H^2(\obstacle_{+})$, $\partial_{\nu, A}v = \nu \cdot \gamma (A\nabla v)$).

Given $f \in L^2(\obstacle_+)$ with $\supp f \Subset \Rea^d$ and $k>0$, $u=(u_-,u_+) \in H^1_{\rm loc}(\Rea^d)$ satisfies 
the transmission problem if 
\begin{align}\nonumber
c^{2} \nabla \cdot (A_-\nabla u_-) + k^2u_-= -f  &\quad\text{ in } \obstacle_-, \\ \nonumber
\nabla \cdot (A_+\nabla u_+) + k^2u_+= -f &\quad\text{ in } \obstacle_+, \\
u_-= u_+, \quad  \partial_{\nu,A_-} u_-= \beta\partial_{\nu,A_+}  u_+  &\quad\text{ on }\partial\obstacle_-, \label{eq:trcont} 
\end{align}
and $u_+$ satisfies the Sommerfeld radiation condition \eqref{eq:src}.
\end{definition}

When $A_-, A_+,$ and $c$ are constant, two of the four parameters $A_-,A_+,c$, and $\beta$ are redundant. For example, by rescaling $u_-,u_+$, and $f$, all such transmission problems can be described by the parameters $c$ and $\beta$ (with $A_-=A_+=1$), as in, e.g., \cite{CaPoVo:99}, or by the parameters $A_-$ and $c$ (with $A_+=\beta=1$); see, e.g., the discussion and examples after \cite[Definition 2.3]{MoSp:19}.

The definition of $\Csol$ for the transmission problem is almost identical to Definition \ref{def:Csol}, except that
the norms in \eqref{eq:Csol} are now over $B_R$ (as opposed to $B_R\cap \obstacle_+$) and 
 now $\Csol$ depends additionally on $\beta$ 

\begin{theorem}\mythmname{Conditions under which $\Csol(k)$ is polynomially bounded in $k$ for the transmission problem}\label{thm:polyboundT}
In each of the following conditions we assume that $\obstacle_-$, $A$, and $c$ are as in Definition \ref{def:transmission}.

(i) If $\obstacle_-$ is smooth and strictly convex with strictly positive curvature, $A=I$, $c$ is a constant $\geq 1$, and $\beta>0$, 
then $\Csol(k)$ is independent of $k$ for all sufficiently large $k$; i.e., \eqref{eq:Csol} holds for all $k\geq k_0$ with $M=0$

(ii) If $\obstacle_-$ is Lipschitz and star-shaped, $A=I$, and $c$ is a constant with 
\beqs
\frac{1}{c^2}\leq \frac{1}{\beta}\leq 1
\eeqs
then $\Csol(k)$ is independent of $k$ for all sufficiently large $k$.

(iii) 
If $\obstacle_-$ is star-shaped, $\beta=1$, and both $A$ and $c$ are monotonically non-increasing in the radial direction (in the sense of \cite[Condition 2.6]{GrPeSp:19})
then $\Csol(k)$ is independent of $k$ for all sufficiently large $k$.

(iv) 
Under no additional assumptions on $\obstacle_-$, $A$, and $c$, given $k_0>0$ and $\delta>0$ there exists a set $J\subset [k_0,\infty)$ with $|J|\leq \delta$ such that  
\beqs
\Csol(k) \leq C k^{5d/2+1+\eps} \quad\tfa k \in [k_0,\infty)\setminus J,
\eeqs
 for any $\eps>0$, where $C$ depends on $\delta, \eps, d, k_0, A, c$, and $\beta$.
\end{theorem}

\bpf[References for the proof]
(i) is proved in \cite[Theorem 1.1]{CaPoVo:99} (we note that, in fact, a stronger result with $A_-$ variable is also proved there).
(ii) is proved in \cite[Theorem 3.1]{MoSp:19}.
(iii) is proved in \cite[Theorem 2.7]{GrPeSp:19}.
(iv) is proved for constant $c$ and globally Lipschitz $A$ in \cite[Theorem 1.1 and Corollary 3.6]{LaSpWu:20};
the proof for these more-general $c$ and $A$ follows from Lemma \ref{lem:obstacle} below.
\epf

\subsubsection{Theorem \ref{thm:mainbb} applied to the transmission problem}
\begin{manualtheorem}{C} \mythmname{Theorem \ref{thm:mainbb} applied to the transmission problem}\label{thm:transmission}
Suppose that $\obstacle_-$, $A, c,$ and $\beta,$ are as in Definition \ref{def:transmission} and, additionally, $A$ and $c$ are $C^{2m-2,1}$ and $\obstacle_-$ is $C^{2m-1,1}$ for some integer $m \geq 1$.

If $\Csol(k)$ is polynomially bounded for $k\in K$ (in the sense of Definition \ref{def:polybound}), then 
given $f\in L^2(\Rea^d)$ supported in $B_{R}$ with $R\geq R_0$, the solution $u$ of the transmission problem is such that
there exists $u_{\mathcal A}  = (u_{+, \mathcal A}, u_{-, \mathcal A} ) \in C^\infty(B_R \cap\obstacle_+) \times C^\infty(\obstacle_-)$ and $u_{H^2} = (u_{+, H^2}, u_{-, H^2}) \in H^2(B_R \cap\obstacle_+) \times H^2(\obstacle_-) $, satisfying  (\ref{eq:trcont}), and such that
\beqs
u|_{B_{R}} = u_{\mathcal A} + u_{H^2}.
\eeqs
Furthermore there exist $C_1, C_2>0$, independent of $k$ but with $C_2= C_2(m)$, such that
\begin{equation} \label{eq:decHFTR}
\Vert \partial^\alpha u_{\pm, H^2} \Vert_{L^2(B_R \cap \obstacle_\pm)} \leq C_1 k^{|\alpha|-2}\Vert f \Vert_{L^2(B_R)} \quad \tfa  k \in  K\text{ and for all }  |\alpha|\leq 2,
\end{equation}
and
\begin{equation} \label{eq:decLFTR}
\Vert \partial^\alpha u_{\pm, \mathcal A} \Vert_{L^2(B_R \cap \obstacle_\pm)}  \leq C_2(m) k ^{ |\alpha| - 1 + M}\Vert f \Vert_{L^2(B_R)} \quad\tfa  k \in  K\text{ and for all }  |\alpha|\leq 2m.
\end{equation}
\end{manualtheorem}

\subsubsection{Corollary about frequency-explicit convergence of the $h$-FEM}

For simplicity we consider the case where the parameter $\beta$ in the transmission condition \eqref{eq:trcont} equals one; recall from the comments below Definition \ref{def:transmission} that, at least in the constant-coefficient case, this is without loss of generality.
The variational formulation of the transmission problem is then \eqref{eq:vf} with $B_R\cap\obstacle_+$ replaced by $B_R$ 
and $a(\cdot,\cdot)$ given by \eqref{eq:sesqui} with $c$ understood as equal to one in $B_R\cap \obstacle_+$ 

Since the constant $C_2$ in \eqref{eq:decLFTR} depends on
$m$, we cannot prove a result about the $hp$-FEM for the transmission problem
of Definition \ref{def:transmission}. We therefore consider the
$h$-FEM and prove the first sharp quasioptimality result for this
problem (see Remark \ref{rem:sig2} below for more discussion on the
novelty of our result). 

\begin{assumption}\label{ass:VN}
$(V_N)_{N=0}^\infty$ is a sequence of piecewise-polynomial approximation spaces on quasi-uniform meshes with mesh diameter $h$ and polynomial degree $p$. Furthermore, (i) the mesh consists of curved elements that exactly triangulate $B_R$ and $\obstacle_-$, so that each element in the mesh is included in either $\obstacle_-$ or $B_R\cap \obstacle_+$, and (ii) there exists an interpolant operator $I_{h,p}$ such that
for all $0\leq j \leq \ell \leq p$,  there exists $C(j,\ell,d)>0$ such that  
\beq\label{eq:interp}
\big| v- I_{h,p} v\big|_{H^j(B_R)}\leq C(j,\ell,d) \,h^{\ell+1-j} \Big( \N{v_+}_{H^{\ell+1}(B_R\cap\obstacle_+)} +
\N{v_-}_{H^{\ell+1}(\obstacle_-)} \Big)
\eeq
for all $v=(v_+, v_-) \in H^{\ell+1}(B_R\cap\obstacle_+)\times H^{\ell+1}(\obstacle_-)$.
\end{assumption}

Assumption \ref{ass:VN} is satisfied by the $hp$ approximation spaces described in \cite[\S5]{MeSa:10}, \cite[\S5.1.1]{MeSa:11} (with \eqref{eq:interp} holding by \cite[Theorem B.4]{MeSa:10}, and also by curved Lagrange finite-element spaces in \cite{Be:89} (with \eqref{eq:interp} holding by \cite[Theorem 4.1 and Corollary 4.1]{Be:89}).
\begin{manualtheorem}{C1}[Quasioptimality of $h$-FEM for the transmission problem]\label{thm:FEM2}
Let $d=2$ or $3$. Suppose that $\beta=1$, $A, c,$ and $\obstacle_-$ are as in Definition \ref{def:transmission}.
Given an integer $p$, if $p$ is odd assume that $\obstacle_-$ is $C^{p,1}$ and both $A$ and $c$ are $C^{p-1,1}$; if $p$ is even, assume that 
$\obstacle_-$ is $C^{p+1,1}$ and both $A$ and $c$ are $C^{p,1}$.

 Let $(V_N)_{N=0}^\infty$ be a sequence of piecewise-polynomial approximation spaces of degree $p$ satisfying Assumption \ref{ass:VN}
and let $u_N$ be the Galerkin solution defined by \eqref{eq:FEM}.

If $\Csol(k)$ is polynomially bounded (in the sense of Definition \ref{def:polybound}) for $k\in K\subset [k_0,\infty)$ then there exists $C>0$, depending on $A, c, R$, $d$, $k_0$, and $p$, but independent of $k$ and $h$, such that if 
\beqs
h^{p}k^{p+1+M} \leq C 
\eeqs
then, for all $k\in K$, the Galerkin solution exists, is unique, and satisfies the quasioptimal error bound 
\beqs
\N{u-u_N}_{H^1_k(B_R)}\leq \Cqo \min_{v_N\in V_N} \N{u-v_N}_{H^1_k(B_R)},
\eeqs
with $\Cqo$ given by \eqref{eq:Cqo}.
\end{manualtheorem}

The regularity assumptions in Theorem \ref{thm:FEM2} are optimal with $p$ is odd, but suboptimal when $p$ is even. This is due to Theorem \ref{thm:transmission} controlling Sobolev norms of even order of the solution, which is ultimately due to our using powers of the operator (which is of order two) to obtain regularity of the solution (see \eqref{eq:induction1} in the proof of Theorem \ref{thm:transmission}). For example, when $p=2$ we require $u\in H^3$ in Theorem \ref{thm:FEM2}, but we achieve this by requiring that $\obstacle_-,A,$ and $c$ are such that $u\in H^4$.

\bre[The significance of Theorem \ref{thm:FEM2}]\label{rem:sig2}
The fact that ``$h^p k^{p+1}$ sufficiently small'' is a sufficient condition for quasioptimality of the Helmholtz $h$-FEM in nontrapping situations (i.e.~$M=0$) was proved for a variety of Helmholtz problems for $p=1$ in \cite[Prop.~8.2.7]{Me:95}, 
\cite[Theorem 4.5]{GrSa:20}, \cite[Theorem 3]{GaSpWu:20} (building on the 1-d results of \cite[Theorem 3.2]{AzKeSt:88}, \cite[Theorem 3]{IhBa:95a}, \cite[Theorem 4.13]{Ih:98}, and \cite[Theorem 3.5]{IhBa:97}) 
and for $p>1$ in \cite[Corollary 5.6]{MeSa:10}, \cite[Remark 5.9]{MeSa:11}, \cite[Theorem 5.1]{GaChNiTo:18}, and
\cite[Theorem 2.15]{ChNi:20}.
Numerical experiments indicate that this condition is also necessary -- see, e.g., \cite[\S4.4]{ChNi:20}.

Of these existing results, only \cite[Theorem 2.15]{ChNi:20} covers the Helmholtz equation with variable $A$ and $c$ that are also allowed to be discontinuous. However, the results in \cite{ChNi:20} hold only when 
an impedance boundary condition is imposed on the truncation boundary (in our case $\partial B_R$), which is equivalent to approximating the exterior Helmholtz Dirichlet-to-Neumann map by $\ri k$. Furthermore, the proof of \cite[Theorem 2.15]{ChNi:20} uses 
the impedance boundary condition in an essential way. Indeed, in \cite[Proof of Lemma 2.13]{ChNi:20} the solution is expanded in powers of $k$, i.e.~$u = \sum_{j=0}^{\infty} k^j u_j$, and then on $\GR$ one has $\partial_\bn u_{j+1} = \ri u_j$; this relationship between $u_{j+1}$ and $u_j$ on $\GR$ no longer holds if $\DtN$ is not approximated by $\ri k$.

The Helmholtz equation with an impedance boundary condition is often used as a model problem for numerical analysis (see, e.g., the references in \cite[\S1.8]{GaLaSp:21}). However, it has recently been shown that, in the limit $k\tendi$ with the truncation boundary fixed, the error incurred in approximating the Dirichlet-to-Neumann map with $\ri k$ is bounded away from zero, independently of $k$, even in the best-possible situation when the truncation boundary equals $\partial B_R$ for some $R$; see \cite[\S1.2]{GaLaSp:21}. Therefore, even if one solves the problem truncated with an impedance boundary condition with a high-order method (i.e., $p$ large), the solution of the truncated problem will not be a good approximation to the true scattering problem when $k$ is large.
\ere

\subsection{The main result applied to the Helmholtz equation in $\mathbb R^d$ with $C^\infty$ coefficients}\label{sec:LSW3}

Theorem \ref{thm:mainbb} can also be used to recover the main result of \cite{LaSpWu:22}, namely \cite[Theorem 3.1]{LaSpWu:22}.
\begin{manualtheorem}{D} \mythmname{The main result of \cite{LaSpWu:22} as a corollary of Theorem \ref{thm:mainbb}}\label{thm:LSW3}
Assume that $\obstacle_- = \emptyset$ and that $A, c$ are as in Definition \ref{def:EDP} and are furthermore $C^\infty$.
If $\Csol(k)$ is polynomially bounded (in the sense of Definition \ref{def:polybound}), then, 
given $f\in L^2(B_R)$, the solution $u$ of the Helmholtz problem \eqref{eq:HelmholtzD}, \eqref{eq:src} is such that
there exists $u_{\mathcal A}$, analytic in $B_R$, and $u_{H^2} \in H^2(B_R)$, such that 
\beqs
u|_{B_{R}} = u_{\mathcal A} + u_{H^2}.
\eeqs
Furthermore, there exist $C_1, C_2$ and $C_3$, all independent of $k$ and $\alpha$, such that 
\begin{equation} \label{eq:decHFLSW3}
\Vert \partial^\alpha u_{H^2} \Vert_{L^2(B_R)} \leq C_1 k^{|\alpha|-2}\Vert f \Vert_{L^2(B_R)} \quad \tfa  k \in K \text{ and for all } |\alpha|\leq 2,
\end{equation}
and
\begin{equation} \label{eq:decLFLSW3}
\Vert \partial^\alpha u_{\mathcal A} \Vert_{L^2(B_R) } \leq C_2 (C_3)^{|\alpha|} k ^{ |\alpha| -1+M} \,\Vert f \Vert_{L^2(B_R)} \quad\tfa  k \in K \text{ and for all }   \alpha.\end{equation}
\end{manualtheorem}

Observe that, by Part (i) of Lemma \ref{lem:analytic}, $u_\cA$ is entire.
The decomposition in Theorem \ref{thm:LSW3} can be used to show that the $hp$-FEM applied to the Helmholtz equation in $\mathbb R^d$ with $C^\infty$ coefficients is quasioptimal (with constant independent of $k$) if the conditions \eqref{eq:thresholdD} hold; see \cite[Theorem 3.4]{LaSpWu:22}. 

\subsection{Informal discussion of the ideas behind Theorem \ref{thm:mainbb}}\label{sec:informal}

It is instructive to first recall the ideas behind the results of \cite{MeSa:10, MeSa:11, EsMe:12, MePaSa:13}.

\paragraph{How the results of \cite{MeSa:10, MeSa:11, EsMe:12, MePaSa:13} were obtained.}

The paper \cite{MeSa:10} considered the Helmholtz equation \eqref{eq:Helmholtz} posed in $\Rea^d$ with the Sommerfeld radiation condition \eqref{eq:src}. The decomposition $u= \uhigh + \ulow$ was obtained by decomposing the data $f$ in \eqref{eq:Helmholtz}
 into ``high-'' and ``low-'' frequency components, with $\uhigh$ the Helmholtz solution for the high-frequency component of $f$,
and $\ulow$ then the Helmholtz solution for the low-frequency component of $f$.
  The frequency cut-offs were defining using the indicator function
\beq\label{eq:MScutoff}
1_{B_{\lambda k}}(\zeta) := 
\begin{cases}
1 & \tfor |\zeta|\leq \lambda\, k,\\
0 & \tfor |\zeta|\geq \lambda\, k,
\end{cases}
\eeq
with $\lambda$ a free parameter 
(see \cite[Equation 3.31]{MeSa:10} and the surrounding text). 
In \cite{MeSa:10} the frequency cut-off \eqref{eq:MScutoff} was then used with (a) the expression for $u$ as a convolution of the fundamental solution and the data $f$, and (b) the fact that the fundamental solution is known explicitly for the PDE \eqref{eq:Helmholtz} to obtain the appropriate bounds on $\ulow$ and $\uhigh$ using explicit calculation (involving Bessel and Hankel functions).
The decompositions in \cite{MeSa:11, EsMe:12, MePaSa:13} for the exterior Dirichlet problem and interior impedance problem were obtained using the results of \cite{MeSa:10} combined with extension operators (to go from problems with boundaries to problems on $\Rea^d$).

 Because the proof technique in  \cite{MeSa:10} does not generalise to the variable-coefficient Helmholtz equation \eqref{eq:Helmholtzvar}, until the recent paper \cite{LaSpWu:22} there did not exist in the literature analogous decomposition results for the variable-coefficient Helmholtz equation. This was despite the increasing interest in the numerical analysis of \eqref{eq:Helmholtzvar}
see, e.g., \cite{Ch:16, BaChGo:17, ChNi:20, GaMo:19, Pe:20, GrSa:20, GaSpWu:20, LaSpWu:19, GoGrSp:20}. 

\paragraph{The recent results of \cite{LaSpWu:22}:~the decomposition for the variable-coefficient Helmholtz equation in free space.}

The paper \cite{LaSpWu:22} obtained the analogous decomposition to that in \cite{MeSa:10} 
for the Helmholtz problem in $\Rea^d$
but now for the variable-coefficient Helmholtz equation \eqref{eq:Helmholtzvar} with $A$ and $c \in C^\infty$. 
This result was obtained again using frequency cut-offs (as in \cite{MeSa:10}) but now applying them to the solution $u$ as opposed to the data $f$.
Any cut-off function that is zero for $|\zeta|\geq Ck$ is a cutoff to a compactly-supported set in phase space, and hence
enjoys analytic estimates. The main difficulty in \cite{LaSpWu:22}, therefore, was in showing that the high-frequency component $\uhigh$ satisfies a bound with one power of $k$ improvement over the bound satisfied by $u$. This was achieved by choosing the cut-off so that the (scaled) Helmholtz operator 
$k^{-2}\nabla\cdot (A \nabla) +  c^{-2}$ 
is \emph{semiclassically elliptic} on the support of the high-frequency cut-off. 
Then, choosing the cut-off function to be smooth (as opposed to discontinuous, as in \eqref{eq:MScutoff}) allowed \cite{LaSpWu:22} to use basic facts about the ``nice'' behaviour of elliptic semiclassical pseudodifferential operators (namely, they are invertible up to a small error) to prove the required bound on $\uhigh$.
The expository paper \cite{Sp:22} shows that, when $A= I$ and $c=1$, the arguments in \cite{LaSpWu:22} involving pseudodifferential operators reduce to using the Fourier transform, and in this case a frequency cut-off of the form \eqref{eq:MScutoff} can be used.

\paragraph{The frequency decomposition achieved in Theorem \ref{thm:mainbb}.}

In this paper, we achieve the desired decomposition into low- and
high-frequency pieces in the manner best adapted to the functional
analysis of the Helmholtz equation:~by using the functional calculus
for the Helmholtz operator itself.  Recall that once we realise the
operator
\beq\label{eq:P}
P=-c^2\nabla \cdot (A\nabla)
\eeq
with appropriate domain as a self-adjoint operator (on a space weighted by $c^{-2}$), the functional calculus for self-adjoint operators
allows us to define 
$\phi(P)$ for a broad class of functions $g.$  In
particular, given $k>0$, we take $\phi$ a cutoff function on $\Rea^d$ 
equal to $1$ on $B(0, \mu k)$ for some $\mu>1$.  Then, for fixed $k$,
$(1-\phi)(P)$ is a high-frequency cutoff and $\phi(P)$ a low-frequency
cutoff. 
We emphasise that working with functions of the operator can be thought of as just the classic idea of using expansions in terms of eigenfunctions of the differential operator. Indeed, in the special case $A=I, c=1$, these frequency cut-offs are simply Fourier
multipliers of the type used in \cite{LaSpWu:20}.

The novelty of the approach used here is to make the functional
calculus approach work in the much more general setting of
\emph{semiclassical black-box scattering} introduced by Sj\"ostrand-Zworski \cite{SjZw:91}, which
allows us to treat variable (possibly rough) media, impenetrable obstacles, and
penetrable obstacles all at once. We rescale,
setting $\hsc=k^{-1}$, and study operators 
$P_\hsc$ equal to a variable-coefficient Laplacian outside
the ``black-box'' $B_{R_0}$, and equal to $-\hsc^2\Delta$ outside a larger ball $B_{R_1}$.
 We are now
interested in functions of $P_{\hsc}$ of the form $\psi(P_{\hsc})$
with $\psi=1$ in $B(0, \mu)$ and $0$ in $(B(0, 2\mu))^c$. 
After multiplying the solution $u$ by a cut-off function $\varphi$ that equals one near the black box 
(since $u$ is only locally $L^2$), we split
$$
\varphi u=\Pihigh (\varphi u)+\Pilow (\varphi u)
$$
with
$$
\Pilow\equiv \psi(P_{\hsc}),\quad \Pihigh  \equiv (1-\psi)(P_{\hsc}),
$$
and both pieces again defined by the spectral theorem.
We now discuss the two pieces separately.

We wish to analyze $\Pihigh \varphi u$ by using the semiclassical ellipticity of
$P_{\hsc}-I$ on its support in phase space.  The latter notion would be
well-defined if $\Pihigh$ were globally a pseudodifferential
operator.
In the broad context of the black-box theory, though,
while the function $\psi(P_{\hsc})$ is well-defined as an abstract operator
on a Hilbert space, its
\emph{structure} is much less manifest than it would be for the flat
Laplacian in Euclidean space.  Not much can be said in any generality
about $\Pihigh$ on the black-box, but this is unnecessary in any event:
we use an abstract ellipticity argument based on the Borel
functional calculus, with the ellipticity in question now amounting to
the bounded invertibility of $P_\hsc-1$ on the range of $\Pihigh,$ which
just follows from the boundedness of the function $(\lambda-1)^{-1}(1-\psi(\lambda)).$
%Since our outgoing solution $u$ is only locally $L^2$, 
However, we do additionally need to understand the commutator of $\Pihigh$ with the
localiser $\varphi$. % that equals one near the black-box.  
Fortunately, we are able to use the
Helffer--Sj\"ostrand approach to the functional calculus
\cite{HeSj:89} to
describe this commutator explicitly.  The method of \cite{HeSj:89} is a
powerful tool for obtaining the structure theorem that a decently-behaved function of a self-adjoint
elliptic differential operator is, as one might hope, in fact a pseudodifferential
operator \cite[Chapter 8]{DiSj:99} (a result originally due to
Strichartz \cite{St:72} in the setting of the homogeneous
pseudodifferential calculus and Helffer--Robert \cite{HeRo:83} in the
semiclassical setting used here). Additionally,
 Davies \cite{Da:95} later pointed out that in fact the
same method affords a novel proof of the functional calculus
formulation of the spectral theorem itself.
Here, we use some refinements of Sj\"ostrand \cite{Sj:97}  to learn
that \emph{away} from the black-box we can in fact treat $\Pihigh$ as a
pseudodifferential operator (see Lemma \ref{thm:funcloc2}), and hence deal with $[\Pihigh, \varphi]$ as
an element of the pseudodifferential calculus, solving it away by once
again using ellipticity (this time in the context of pseudodifferential
operators) together with our polynomial
resolvent estimate.

While the analysis of $\Pihigh \varphi u$ is insensitive to the contents of the
black-box,
our study of the low-frequency piece $\Pilow \varphi u$ necessarily entails
``opening'' the black-box and studying the local question of elliptic or parabolic estimates
within it.  Intuitively the compact support in the spectral parameter
of the spectral measure of
$P$ applied to $\Pilow \varphi u$ should imply that strong elliptic
estimates hold, but knowing Cauchy-type estimates on high derivatives  
is dependent on analyticity of the underlying problem. 
We therefore make the abstract regularity hypothesis (\ref{eq:lowenest})
locally near the black-box, which allows us to estimate the part of $\Pilow u$
spatially localised near its content. The remaining part living
in $\mathbb R^d$ is then given, thanks to  Sj\"ostrand \cite{Sj:97} 
again, by a Fourier multiplier up to negligible terms, and hence enjoys the analytic estimate \eqref{eq:decLF3} thanks to the properties
of the Fourier transform, as used in \cite{LaSpWu:22}.

If, for instance, $P$ is given by \eqref{eq:P}
exterior to a $C^\infty$
obstacle with Dirichlet boundary condition, we know by the functional
calculus that $P^m \Pilow \varphi u$ is bounded for all $m \in \mathbb{N}.$  This 
yields elliptic estimates which allow us to estimate all derivatives
of $\Pilow \varphi u$ up to the obstacle, but the resulting
estimates on $\pa^\alpha \Pilow u$ grow non-optimally in $\alpha$; see Corollary \ref{cor:reg} and
Theorem \ref{thm:transmission}. Such
estimates, which indeed are the only ones we have been able to obtain
in the case of penetrable obstacles, suffice for
applications to the $h$-FEM but are far from optimal in dealing with $hp$-FEM.
In the boundary case we
therefore use a stronger property of $\Pilow \varphi u:$ we can run the \emph{backward heat equation} on $\Pilow \varphi u$ for
as long as we like and obtain $L^2$ estimates on the result.  If the
boundary is analytic then known heat kernel estimates (see \cite{EsMoZh:17}) yield satisfactory
Cauchy-type estimates on $\pa^\alpha \Pilow \varphi u$; see Corollary \ref{cor:heat} and Theorem \ref{thm:Dirichlet}.

\subsection{Statement of the main result in the black-box setting}\label{sec:mainresult}

The following theorem (Theorem \ref{thm:mainbb}) obtains the decomposition $u=\uhigh+\ulow$ in the framework of black-box scattering
 introduced by Sj\"ostrand--Zworski in \cite{SjZw:91}. In this framework, the operator $P_\hsc$, where $\hsc:= k^{-1}$ is the semiclassical parameter
\footnote{The semiclassical parameter is often denoted by $h$, but we use $\hsc$ to avoid a notational clash with the meshwidth of the FEM appearing in \S\ref{sec:1.1} and used in Theorems \ref{thm:FEM1} and \ref{thm:FEM2}.}
, is a variable-coefficient Helmholtz operator outside $B_{R_0}$ (the ball of radius $R_0$ and centre zero) for some $R_0>0$, but is not specified inside this ball (i.e., inside the ``black box'').
In particular, this framework includes the Helmholtz exterior Dirichlet and transmission problems, and 
Theorems \ref{thm:Dirichlet} and \ref{thm:transmission} above are Theorem \ref{thm:mainbb} is specialised to those settings.

The theorem is stated using notation from the black-box framework, recapped in \S\ref{sec:blackbox}. 
The only non-standard concept we use is that of a \emph{black-box differentiation operator}, which is a family of operators 
agreeing with differentiation outside the black-box (see Definition
\ref{def:BBdiff} below).  

To understand the statement of the following theorem,
  the reader not familiar with black box scattering should read it
  with the following identifications, which always hold away from the
  black box, and, with suitable interpretation, continue to hold inside it
  in the examples considered below:~the Hilbert space $\mathcal H$ is $L^2$, 
 the operator $P_\hbar$ is $-\hbar^2  \Lap$, and the subspace $\mathcal D \subset \mathcal H$ is the domain of $P_\hbar$.
  The superscript $\sharp$ denotes the corresponding object
compactified onto a large reference torus $\mathbb T^d_{R_{\sharp}} := \quotient{\mathbb{R}^d}{(2R_{\sharp}\mathbb{Z})^d}$, so that 
$P_\hbar^\sharp$ is $-\hbar^2  \Lap,$ on the torus, and $\mathcal{D}^{\sharp,m}_\hsc$ the domain of $(P_\hbar^\sharp)^m$, with norms weighted in the standard way with $\hbar$ (see \eqref{eq:Hhnorm} below, and compare to \eqref{eq:1knorm}).
Finally, the notation $\lesssim$ indicates that the omitted constant is independent of $\hsc$ and $\alpha$ and 
\beq\label{eq:C0}
C_0(\RR) :=\Big\{f \in C(\RR)\colon \lim_{\lambda \to
  \pm \infty} f(\lambda)=0\Big\}.
\eeq

\begin{manualtheorem}{A} \mythmname{The decomposition in the black-box setting} \label{thm:mainbb}
Let $P_{\hsc}$ be a semiclassical black-box operator on $\mathcal H$ (in the sense of Definition \ref{def:bb}). Then there exists $\Lambda>0$ 
such that the following holds. Suppose that, for some $\hsc_0>0$,
there exists 
$\subsetH \subset (0,\hsc_0]$ such that the following two assumptions hold. 
\begin{enumerate}
\item There exists $\mathcal D_{\rm out} \subset \mathcal D_{\rm loc}$ 
and $M\geq0$ such that for any $\chi \in C^\infty_{\rm comp}(\mathbb R^d)$ equal to one near
$B_{R_0}$, there exists $C>0$ such that if $v \in \mathcal D_{\rm
  out}$ is a solution to $(P_\hsc- I)v=\chi g$, then
\begin{equation} \label{eq:res}
\Vert \chi v \Vert_{\mathcal H} \leq C \hsc^{-M-1}\Vert g\Vert_{\mathcal H} \quad \tfa  \hsc \in \subsetH.
\end{equation} 
\item 
There exists $\mathcal E \in C_0(\mathbb R)$ that is nowhere zero on $[-\Lambda, \Lambda]$ such that 
\beq\label{eq:Friday1}
\mathcal E(P^\sharp_\hsc) = E + \residual,
\eeq
where $E$ has the following property:~there exists $\rho \in C^\infty(\mathbb T^d_{R_{\sharp}})$ equal to one near $B_{R_0}$, such that, for some $\alpha$-family of black-box differentiation operators $(D(\alpha))_{\alpha \in \frak A}$, 
\begin{equation} \label{eq:lowenest}
\N{\rho D(\alpha) E v}_{\mathcal H^{\sharp}} \leq C_{\mathcal E}(\alpha, \hsc)\Vert v \Vert_{\mathcal H^\sharp} \quad \tfa  v \in \mathcal D_\hsc^{\sharp, \infty} \tand \hsc \in \subsetH,
\end{equation}
for some $C_{\mathcal E}(\alpha, \hsc)>0$.
\end{enumerate}
Given $R>0$ such that $R_0<R<R_\sharp$, 
if $g \in \mathcal H$ is compactly supported in $B_{R}$ and $u \in \mathcal D_{\rm out}$ satisfies
\begin{equation} \label{eq:pde}
(P_{\hsc} - 1)u = g,
\end{equation}
then there exists $u_{ H^2} \in \mathcal D^{\sharp}$ and $u_{\mathcal A} \in \mathcal D_\hsc^{\sharp,\infty}$ 
such that
\begin{equation} \label{eq:maindec}
u|_{B_{R}} = \big( \uhigh+ \ulow\big)|_{B_{R}}.
\end{equation}
Furthermore, $u_{H^2}$ satisfies
\begin{equation} \label{eq:decHF}
\Vert u_{H^2} \Vert_{\mathcal H^{\sharp} } + \big\| P^{\sharp}_{\hsc}u_{H^2}\big\|_{\mathcal H^{\sharp} }  \lesssim \Vert g \Vert_{\mathcal H} \quad \tfa  \hsc \in \subsetH,
\end{equation}
and for any $\widetilde R>0$ 
with $R_0<\widetilde R<R_\sharp$, there exist
$\RfarA, \RfarB, \RlocB, \RlocA$ with
$R_0<\RfarA<\RfarB<\RlocB<\RlocA<\widetilde R$
such that $u_{\mathcal A}$ decomposes as
\begin{equation} \label{eq:decLF0}
u_{\mathcal A} = u^{R_0}_{\mathcal A} + u^\infty_{\mathcal A},
\end{equation}
where $ u^{R_0}_{\mathcal A}\in \mathcal{D}^\sharp$ is regular near the black-box and negligible away from it, in the sense that
\begin{equation} \label{eq:decLF1}
\Vert D(\alpha) u^{R_0}_{\mathcal A} \Vert_{\mathcal H^{\sharp}(B_{\RlocA}) } \lesssim C_{\mathcal E}(\alpha, \hsc) \sup_{\lambda\in[-\Lambda, \Lambda]} \big|\mathcal E(\lambda)^{-1}\big|\; \hsc^{-M-1} \Vert g \Vert_{\mathcal H} \quad \tfa  \hsc \in \subsetH \tand \alpha \in \mathfrak A,
\end{equation}
and, for any $N,m>0$ there exists $C_{N,m}>0$ such that
\begin{equation} \label{eq:decLF2}
\Vert u^{R_0}_{\mathcal A} \Vert_{\mathcal D_\hsc^{\sharp,m}((B_{\RlocB})^c) } \leq C_{N,m}\hsc^N \Vert g \Vert_{\mathcal H}   \quad \tfa  \hsc \in \subsetH 
\end{equation}
and  $ u^{\infty}_{\mathcal A}$ is entire away from the black-box and negligible near it, in the sense that for some $\lambda>1$
\begin{equation} \label{eq:decLF3}
\Vert \partial^\alpha u^{\infty}_{\mathcal A} \Vert_{\mathcal H^{\sharp}((B_{\RfarA})^c) } \lesssim \lambda^{|\alpha|} \hsc^{-|\alpha|-M-1} \Vert g \Vert_{\mathcal H} \quad \tfa  \hsc \in \subsetH \tand \alpha \in \mathfrak A,
\end{equation}
and, for any $N,m>0$ there exists $C_{N,m}>0$ such that
\begin{equation} \label{eq:decLF4}
\Vert u^{\infty}_{\mathcal A} \Vert_{\mathcal D_\hsc^{\sharp,m}(B_{\RfarB}) } \leq C_{N,m}\hsc^N \Vert g \Vert_{\mathcal H}   \quad \tfa  \hsc \in \subsetH. 
\end{equation}
In addition, if $\mathcal E(P^\sharp_\hbar) = E$ (i.e., with no $\residual$ remainder in \eqref{eq:Friday1}), then the functions $u_\mathcal A, u^\infty_\mathcal A, u^{R_0}_\mathcal A , u_{H^2}$ are all independent of $\mathcal E$, and all the implicit constants above are independent of $\mathcal E$ as well.

Finally, if $\rho = 1$, the decomposition (\ref{eq:maindec}) can be constructed in such a way that instead of (\ref{eq:decLF0})--(\ref{eq:decLF4}), $u_{\mathcal A}$ satisfies
the global regularity estimate
\begin{equation} \label{eq:decLF5}
\Vert D(\alpha) u_{\mathcal A} \Vert_{\mathcal H^{\sharp} } \lesssim C_{\mathcal E}(\alpha, \hsc)  \sup_{\lambda\in[-\Lambda, \Lambda]} \big|\mathcal E(\lambda)^{-1}\big| \; \hsc^{-M-1} \Vert g \Vert_{\mathcal H} \quad \tfa  \hsc \in \subsetH\tand \alpha \in \mathfrak A;
\end{equation}
here as well, if $\mathcal E(P^\sharp_\hbar) = E$, then the functions $u_\mathcal A, u_{H^2}$ and all the above estimates do not depend on $\mathcal E$.
\end{manualtheorem}

Point 1 in Theorem \ref{thm:mainbb} is the assumption that the solution operator is polynomially bounded in $\hsc$. 
In the black-box setting, \cite{LaSpWu:20} proved that this assumption always holds with $M>5d/2$ and 
$\{\hbar^{-1}: \hbar \in \subsetH\}^c$ having arbitrarily small measure in $\Rea^+$ (see Part (ii) of Theorem \ref{thm:polyboundD} and Part (iv) of Theorem \ref{thm:polyboundT}). The solution operator is then polynomially bounded because $\subsetH$ excludes (inverse) frequencies close to resonances.
%(We note that, 
(Under an additional assumption about the location of resonances, a similar result with a larger $M$ can also be extracted from \cite[Proposition 3]{St:01} by using the Markov inequality.)

Point 2 in Theorem \ref{thm:mainbb} is a regularity assumption that depends on the contents of the black box.
We later refer to \eqref{eq:lowenest} as the ``low-frequency estimate'', since the fact that $\mathcal E$ is nowhere zero on $[-\Lambda,\Lambda]$ means that it bounds low-frequency components.
The cutoff $\rho$ in \eqref{eq:lowenest} is needed when the black box
  contains, e.g., an analytic obstacle and the operator inside has
  analytic coefficients; indeed the analyticity estimates that we use for \eqref{eq:lowenest} in this case cannot hold in the transition region outside the black
  box, where the coefficients cannot be analytic.
    
Regarding $\uhigh$:~comparing \eqref{eq:res} and \eqref{eq:decHF}, and recalling that in the nontrapping case \eqref{eq:res} holds with $M=0$, we see that $\uhigh$ satisfies a bound 
that is better, by at least one power of $\hsc$, than the bound satisfied by $u$;
this is the analogue of the property (i) in \S\ref{sec:1.1} of the results of 
\cite{MeSa:10, MeSa:11, EsMe:12, MePaSa:13}, 
and is a consequence of the semiclassical ellipticity of $P_\hsc -1$ on high-frequencies (discussed in \S\ref{sec:informal}). The regularity of $\uhigh$ depends on the domain of the operator ($\uhigh \in \mathcal D^{\sharp}$) but not on any other features of the black box (in particular, not on the regularity estimate \eqref{eq:lowenest}).

Regarding $\ulow$:
$\ulow$ is in the domain of arbitrary powers of the operator ($u_{\mathcal A} \in \mathcal D_\hsc^{\sharp,\infty}$) and so is smooth in an abstract sense. $\ulow$ is split further into two parts: $u^{R_0}_{\mathcal A}$ and $u^\infty_{\mathcal A}$, with 
$u^{R_0}_{\mathcal A}$ regular near the black-box and negligible away from it, and
$ u^{\infty}_{\mathcal A}$ entire away from the black-box and negligible near it;
Figure \ref{fig:line2} illustrates this set up (with ``$u^{R_0}_{\mathcal A}$ analytic'' replaced by ``$u^{R_0}_{\mathcal A}$ regular'').
Comparing \eqref{eq:res} and \eqref{eq:decLF1}/\eqref{eq:decLF3}, we see that, in the regions where they are not negligible,
$u^{R_0}_{\mathcal A}$ and $u^\infty_{\mathcal A}$
satisfy bounds with the same $\hsc$-dependence as $u$, but with improved regularity. These properties are the analogue of the property (ii) in \S\ref{sec:1.1} of the results of 
\cite{MeSa:10}, \cite{MeSa:11}, \cite{EsMe:12}, \cite{MePaSa:13}. In particular, the regularity of $\ulow$ depends on the regularity inside the black-box (from \eqref{eq:lowenest}), and, for the exterior Dirichlet problem with analytic obstacle and coefficients analytic in a neighbourhood of the obstacle, $\ulow$ is analytic.

\subsubsection{How to use Theorem \ref{thm:mainbb}} \label{sec:howto}

To apply Theorem \ref{thm:mainbb} to a scattering problem not discussed in this paper, 
the steps are the following. 
\begin{enumerate}
\item \label{i:a1} Check that the problem fits in the black-box scattering framework of  Sj\"ostrand--Zworski \cite{SjZw:91}.
\item \label{i:a2} Check that a polynomial bound on the solution operator (\ref{eq:res}) holds.
\item \label{i:a3} Show a ``low-frequency'' estimate of type (\ref{eq:lowenest}) for the  corresponding  compactified problem.
\end{enumerate}
Concerning Point \ref{i:a1}:~the black-box framework is specifically designed to include most scattering problems. Examples treated in the literature include scattering by a Lipschitz Dirichlet or Neumann obstacle (Lemma \ref{lem:obstacle},  \cite[\S 2.2]{LaSpWu:19}), by a Lipschitz penetrable obstacle (Lemma \ref{lem:transmission},  \cite[\S 2.2]{LaSpWu:19}),  by a compactly supported potential, by elliptic compactly supported perturbations of
the Laplacian, and scattering on finite volume surface (see for example \cite[\S 4.1]{DyZw:19} for these three last problems). For problems not already covered in the literature, of the conditions in 
\S  \ref{subsec:bb}, the condition on the growth of eigenvalues for the compactified operator (\ref{eq:gro}) will be the main non-trivial assumption to check (for examples of checking this assumption, see, e.g., \S\ref{app:gro}, \cite[Appendix A]{LaSpWu:19}).

Concerning Point \ref{i:a2}:~as mentioned below Theorem \ref{thm:mainbb}, this assumption holds for any $M>5d/2$ 
and for most frequencies by \cite{LaSpWu:19}.
 For nontrapping problems, one expects 
(\ref{eq:res}) to hold with $M=0$ and $\subsetH = (0, h_0]$ (see, e.g., Theorem \ref{thm:polyboundD} below and the references therein).

Therefore, the key step in applying Theorem \ref{thm:mainbb} is Point \ref{i:a3}:~show a ``low-frequency'' estimate of type (\ref{eq:lowenest}) for the corresponding  compactified problem (i.e., the same problem, but considered in a large reference torus). This estimate dictates the regularity estimate on the component $u_{\mathcal A}$, hence, the better the estimate, the better the decomposition. In practical applications, the operator $D(\alpha)$ in (\ref{eq:lowenest}) will be nothing but differentiation $D(\alpha) := \partial^\alpha$. The two main considerations are then the following.
\begin{itemize}
\item[\ref{i:a3}-a.] Understand if one needs $\rho = 1$, or $\rho$ vanishing away from the scatterer. If one aims for an analytic-type estimate, because the problem under consideration has constant coefficients outside a compact set, it cannot typically be analytic everywhere, 
and one needs to take 
$\rho$ vanishing away from the scatterer. For lower-regularity estimates, one can  use a global estimate, i.e., with $\rho = 1$.
\item[\ref{i:a3}-b.] Choose the operator $E$ and the function $\mathcal E$. In the first instance, one can ignore the flexibility given by the error term and aim for $E=\mathcal E(P^\sharp_\hbar)$. 
The function $\mathcal E$ is then dictated by the type of estimate used. 
 For example:
\begin{itemize}
\item $\mathcal E(\lambda) = \re^{-|\lambda|}$ corresponds to a heat-flow estimate (see the proof of Corollary \ref{cor:heat}),
\item $\mathcal E(\lambda) = \sqrt{1+\lambda^2}^{-L}$, $L\geq 1$ corresponds to an elliptic estimate (see the proof of Corollary \ref{cor:reg}),
\item $\mathcal E \in C^\infty_{\rm comp}$ with $\mathcal E = 1$ in $[-M, M]$ corresponds to an estimate on the eigenfunctions of the compactified operator
(see the proof of Theorem \ref{thm:LSW3} in \S\ref{sec:LSW3proof}).
\end{itemize}
An example where the error term in $\mathcal E(P^\sharp_\hbar)=E + \residual$ gives more flexibility is 
the proof of Theorem \ref{thm:LSW3}, where the 
error term is used to take advantage of the regularity of the eigenfunctions of $-\Delta$ on the torus, instead of those of the variable-coefficient operator.

On the other hand, the fact that if $\mathcal E(P^\sharp_\hbar) = E$ (i.e., with no $\residual$ remainder in \eqref{eq:Friday1}) then the decomposition is independent of $\mathcal E$, allows us to use a \emph{family} of $\mathcal E$'s in \eqref{eq:Friday1} and hence a family of estimates as \eqref{eq:lowenest}. This feature allows us to tune the choice of $\mathcal E$, depending on $\hbar$ and $\alpha$, to get the best possible estimate; this procedure is used in the proof of Theorem \ref{thm:Dirichlet} using 
Corollary \ref{cor:heat}, using a heat-flow estimate with a time depending on $\hbar$ and $\alpha$.

Finally, note that Theorem \ref{thm:mainbb} assumes that $\mathcal{E}\in C_0(\Rea)$, but  this is not essential. We could replace $\mathcal{E}$ with an element of $C^\infty_{\rm comp}$ by extending the functions above smoothly from $\{\lambda \geq 0\}$ to $\{\lambda < 0\}$ and multiplying by a cut-off; this is possible since the spectrum of $P^\sharp_\hbar$ is in $[0,\infty)$.
\end{itemize} 

\subsection{Outline of the rest of the paper}
Section~\ref{sec:blackbox} recalls the black-box framework and sets up the associated functional calculus.
Section~\ref{sec:blackboxresult} proves Theorem \ref{thm:mainbb}.
Section~\ref{sec:obstacles} proves Theorems \ref{thm:Dirichlet} and \ref{thm:transmission} (i.e., Theorem \ref{thm:mainbb} specialised to the exterior Dirichlet and transmission problems), and Theorem \ref{thm:LSW3}.
Section~\ref{sec:FEM} proves Theorems \ref{thm:FEM1} and \ref{thm:FEM2} (i.e., the convergence results for the $hp$-FEM for the exterior Dirichlet problem and the $h$-FEM for the transmission problem).
Appendix~\ref{app:sct} recalls results about semiclassical pseudodifferential operators on the torus.
Appendix~\ref{app:gro} %and Appendix~\ref{app:C} 
proves a subsidiary result used to prove Lemma \ref{lem:transmission}. % and Theorem \ref{thm:Dirichlet}, respectively.  

\section{Recap of the black-box framework}\label{sec:blackbox}

\subsection{Abstract framework} \label{subsec:bb}

We now briefly recap the abstract framework of \emph{black-box scattering} introduced in \cite{SjZw:91};
for more details, see the comprehensive presentation in \cite[Chapter 4]{DyZw:19}. A brief overview of black-box scattering with an emphasis on the counting of resonances is contained in \cite[\S2]{LaSpWu:20}.

We emphasise that here we use the approach of \cite[\S2]{Sj:97}, where the black-box operator is a variable-coefficient Laplacian (with smooth coefficients) outside the black-box, and not
the Laplacian $-\hsc^2\Delta$ itself as in \cite[Chapter 4]{DyZw:19}
(although the operator still agrees with $-\hsc^2 \Delta$ outside a
  sufficiently large ball).

\subsubsection*{The Hilbert-space decomposition}

Let $\mathcal{H}$ be an Hilbert space with an orthogonal decomposition
\begin{equation} \label{eq:bbdec}\tag{BB1}
\mathcal{H}=\mathcal{H}_{R_{0}}\oplus L^{2}(\mathbb{R}^{d}\backslash B_{R_0}, \omega(x)\rd x),
\end{equation}
where the weight-function $\omega : \mathbb R^d \rightarrow \mathbb R$ is measurable
 and $\supp(1-\omega)$ is compact in $\Rea^d$.
 Let $\mathbf 1_{B_{R_0}}$ and $\mathbf 1_{\mathbb{R}^{d}\backslash B_{R_0}}$ denote the corresponding orthogonal projections. 
Let $P_{\hsc}$ be a family in $\hsc$ of self adjoint operators $\mathcal{H}\rightarrow\mathcal{H}$
with domain $\mathcal{D}\subset\mathcal{H}$ independent of $\hsc$ (so that, in particular, $\mathcal{D}$ is dense in $\mathcal{H}$). 
Outside the black-box $\mathcal H_{R_0}$,
we assume that $P_\hsc$ equals $Q_\hsc$ defined as follows.
We assume that, for any multi-index $|\alpha|\leq 2$, there exist functions $a_{\hsc, \alpha} \in C^\infty(\mathbb R^d)$, uniformly bounded with respect to $\hsc$, independent of $\hsc$ for $|\alpha|=2$, and such that (i) for some $C_1>0$
\begin{equation} \label{eq:propq_ell}
\sum_{|\alpha|=2} a_{\hsc, \alpha}(x)\xi^\alpha \geq C_1|\xi|^2\quad \tfa x\in \Rea^d,
\end{equation}
(ii) for some $R_{1}>R_0$ 
\begin{equation*} 
\sum_{|\alpha|\leq 2} a_{\hsc, \alpha}(x)\xi^\alpha = |\xi|^2
\hspace{0.3cm}\text{for }|x|\geq R_1,
\end{equation*}
and (iii) the operator $Q_\hsc$ defined by
\beq\label{eq:Qdef}
Q_\hsc := \sum_{|\alpha|\leq2}a_{\hsc,\alpha}(x)(\hsc D_x)^\alpha
\eeq
(where $D:= -\ri \partial$) is formally self-adjoint on $L^2(\mathbb R^d, \omega(x) \rd x)$.

We require the operator $P_{\hsc}$ to be equal to $Q_\hsc$ outside 
the black-box $\mathcal{H}_{R_{0}}$ in the sense that 
\beq\label{eq:bbreq1}\tag{BB2}
 \boldsymbol{1}_{\mathbb{R}^{d}\backslash B_{R_0}}(P_{\hsc}u)=Q_\hsc (u\vert_{\mathbb{R}^{d}\backslash B_{R_0}}) \quad\tfor u\in \cD, \quad\tand\quad
\boldsymbol{1}_{\mathbb{R}^{d}\backslash B_{R_0}}\mathcal{D} \subset H^{2}(\mathbb{R}^{d}\backslash B_{R_0}).
\eeq
We further assume that if, for some $\eps>0$,
\beq\label{eq:bbreq2a}\tag{BB3}
v\in H^{2}(\mathbb{R}^{d})\quad\tand \quad v\vert_{B_{R_{0}+\eps}}=0,\quad \text{ then }\quad v\in\mathcal{D},
\eeq
(with the restriction to $B_{R_0+\epsilon}$ defined in terms
  of the projections in \eqref{eq:bbreq1}; see also \eqref{eq:restriction} below)
and that
\beq\label{eq:bbreq2}\tag{BB4}
\boldsymbol{1}_{ B_{R_0}}(P_{\hsc}+\ri)^{-1}\text{ is compact from } \cH \rightarrow \cH.
\eeq
Under these assumptions, the semiclassical resolvent 
$$
R(z,\hsc):=(P_{\hsc}-z)^{-1}:\mathcal{H}\rightarrow\mathcal{D}
$$
is meromorphic for $\Imag \,z>0$ and extends to a meromorphic
family of operators of $\mathcal{H_{\rm comp}}\rightarrow\mathcal{D}_{\rm loc}$
in the whole complex plane when $d$ is odd and in the logarithmic
plane when $d$ is even \cite[Theorem 4.4]{DyZw:19};
where $\mathcal H_{\rm comp}$ and $\mathcal{D}_{\rm loc}$ are defined by
$$
\mathcal H_{\rm comp} := \Big\{ u\in \mathcal H \; : \; \mathbf 1_{\mathbb R^d\backslash B_{R_0}}u \in L^2_{\rm comp}(\mathbb R^d \backslash B_{R_0}) \Big\},
$$
(where $L^2_{\rm comp}$ denotes compactly-supported $L^2$ functions) and
\begin{align}\nonumber
\mathcal D_{\rm loc} := \Big\{ u\in \mathcal{H}_{R_{0}}\oplus L_{\rm loc}^{2}(\mathbb{R}^{d}\backslash B_{R_0})\; : \; &\tif\,\chi \in C^\infty_{\rm comp}(\mathbb R^d), \; \chi|_{B_{R_0}} =1  \\ 
&\text{then }  (\mathbf 1_{B_{R_0}}u,
\chi \mathbf 1_{\mathbb R^d\backslash B_{R_0}} u) \in \mathcal D \Big\}.\label{eq:Dloc}
\end{align}

\subsubsection*{The reference operator $P^{\text{\ensuremath{\sharp}}}_\hsc$}
Let $R_{\sharp}>R_{1}$
 be such that $\supp(1-\omega)\subset B_{R_\sharp}$, and let $ \mathbb T^d_{R_{\sharp}} := \quotient{\mathbb{R}^d}{(2R_{\sharp}\mathbb{Z})^d}$; we work with $[-R_{\sharp}, R_{\sharp}]^d$ as a fundamental domain for this torus.
Let
\begin{equation*}
\mathcal H^{\sharp} := \mathcal H_{R_0} \oplus  L^{2}(\mathbb T^d_{R_{\sharp}}\backslash B_{R_0}, \omega(x) \, \rd x),
\end{equation*}
and let $\mathbf 1_{B_{R_0}}$ and  $\mathbf 1_{\mathbb T^d_{R_{\sharp}}\backslash B_{R_0}}$ denote the corresponding orthogonal
projections. We
 define
\begin{gather} 
\mathcal D^{\sharp} := \Big\{ u \in \mathcal H^{\sharp}: \; \tif\,\chi \in C^\infty_{\rm comp}(B_{R_{\sharp}}), \; \,\chi = 1 \text{ near } B_{R_0}, \tthen (\mathbf 1_{B_{R_0}}u,
\chi \mathbf 1_{\mathbb T^d_{R_{\sharp}}\backslash B_{R_0}} u) \in \mathcal D, \nonumber \\
 \tand\, (1-\chi)\mathbf 1_{\mathbb T^d_{R_{\sharp}}\backslash B_{R_0}}u \in H^2(\mathbb T^d_{R_{\sharp}}) \Big\}, \label{eq:defdomsharp}
\end{gather}
and, for any $\chi$ as in \eqref{eq:defdomsharp} and $u \in \mathcal D^{\sharp}$,
\begin{gather} 
P^{\sharp}_{\hsc} u := P_{\hsc}(\mathbf 1_{B_{R_0}}u, \chi \mathbf 1_{\mathbb T^d_{R_{\sharp}}\backslash B_{R_0}} u) 
+ Q_\hsc\big((1-\chi)\mathbf 1_{\mathbb T^d_{R_{\sharp}}\backslash B_{R_0}}u\big), \label{eq:defref}
\end{gather}
where we have identified functions supported in $B(0,
  R_{\sharp})\backslash B(0, R_0)\subset
\mathbb{T}^d_{R_\sharp}\backslash B(0, R_0)$ with the corresponding functions on $\mathbb R^d\backslash
B(0, R_0)$ -- see the paragraph on notation below.

Let $q_\hsc \in S^2(\mathbb T^d_{R_\sharp})$ denote the principal symbol of $Q_\hsc$ as an operator acting on the torus $\mathbb T_{R_\sharp}^d$ (see Appendix \ref{app:sct} for a review of semiclassical pseudodifferential operators on $\mathbb
T^d_{R_{\sharp}}$). 
We record for later the fact that \eqref{eq:propq_ell}, \eqref{eq:Qdef}, and the uniform boundedness of $a_{\hsc,\alpha}(x)$ with respect to $\hsc$ imply that there exist $C_1, C_2>0$ such that 
\beq\label{eq:Qnew}
C_1|\xi|^2 \leq q_\hsc \leq C_2|\xi|^2 \quad\text{ for sufficiently large $\xi$}.
\eeq
The idea behind these definitions is that we have glued our black box
into a torus instead of $\mathbb{R}^d$, and then defined on the torus an operator $P^{\sharp}_{\hsc}$ that can be thought of as $P_{\hsc}$ in $\mathcal{H}_{R_{0}}$ and $Q_\hsc$ in
$(\mathbb{R}/2R_{\sharp}\mathbb{Z})^{d}\backslash B_{R_0}$;
see Figure \ref{fig:bb}.
The resolvent $(P^\sharp_\hsc + i)^{-1}$ is compact (see \cite[Lemma 4.11]{DyZw:19}), and hence the spectrum of $P^\sharp_\hsc$, denoted by $\operatorname{Sp}P^\sharp_\hsc$, is discrete (i.e., countable and with no accumulation point).

We assume that
the eigenvalues of $P^{\sharp}_{\hsc}$ satisfy the \emph{polynomial
growth of eigenvalues condition }
\begin{equation}\label{eq:gro}\tag{BB5}
N\big(P^{\sharp}_{\hsc},[-C,\lambda]\big)=O(\hsc^{-d^{\sharp}}\lambda^{d^{\sharp}/2}),
\end{equation}
for some $d^{\sharp}\geq d$ and $N(P^{\sharp}_\hsc,I)$ is the number of eigenvalues of
$P^{\sharp}_{\hsc}$ in the interval $I$, counted with their multiplicity. 
When $d^{\sharp}=d$, the asymptotics \eqref{eq:gro} correspond to a Weyl-type upper bound, 
and thus \eqref{eq:gro} can be thought of as a weak Weyl law. 

We summarise with the following definition.

\begin{definition} \mythmname{Semiclassical black-box operator} \label{def:bb}
We say that a family of self-adjoint operators $P_{\hsc}$ on a Hilbert space $\mathcal H$, with dense domain $\cD$, independent of $\hsc$, is a semiclassical black-box operator if $(P_{\hsc}, \mathcal H)$ satisfies (\ref{eq:bbdec}),  (\ref{eq:bbreq1}),  (\ref{eq:bbreq2a}), (\ref{eq:bbreq2}), (\ref{eq:gro}).
\end{definition}

\begin{figure}[h!]
\begin{center}
    \includegraphics[scale=0.4]{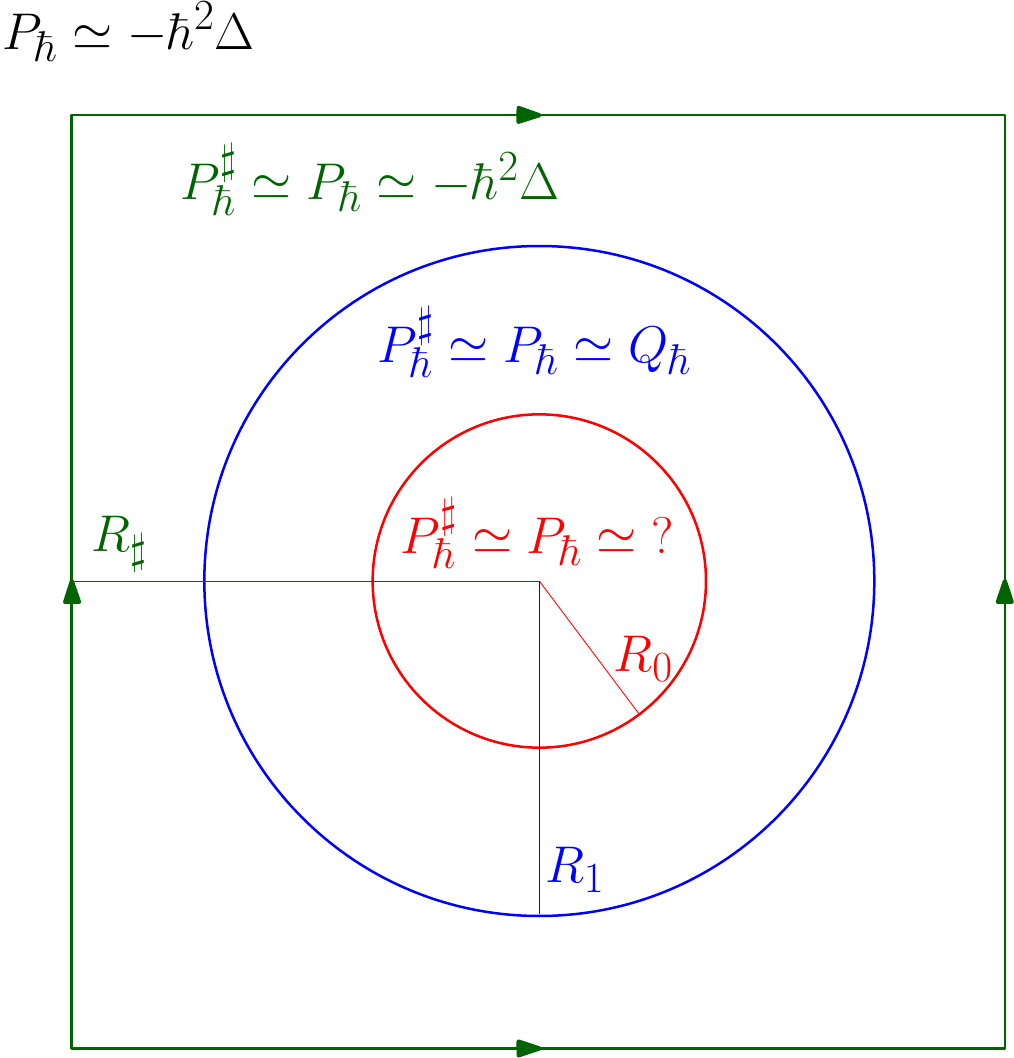}
  \end{center}
    \caption{The black-box setting. The symbol $\simeq$ is used to denote equality in the sense of \eqref{eq:bbreq1} and \eqref{eq:defref}.}
       \label{fig:bb}
\end{figure}

We define a family of black-box differentiation operators as a family of operators 
agreeing with differentiation outside the black-box (note that there is no notion of derivative inside the black-box itself).

\begin{definition}[Black-box differentiation operator]\label{def:BBdiff}
$(D(\alpha))_{\alpha \in \frak A}$ is a family of black-box
differentiation operators on $\mathcal D^{\sharp, \infty}_{\hsc}$ (defined by  \eqref{eq:Dinfty} below)
 if $\frak A$ is a family of $d$--multi-indices, and 
for any $\alpha$ and any $v \in C^\infty_{\rm comp}(\mathbb T^d_{R_{\sharp}}\backslash \overline{B_{R_0}})$,
$$
D(\alpha)v 
=\partial^\alpha v.
$$
\end{definition}

\subsubsection*{Notation}

We identify in the natural way:
\begin{itemize}
\item the elements of $\{ 0 \} \oplus L^2(\mathbb T_{R_{\sharp}}^d \backslash B_{R_0}) \subset \mathcal H^{\sharp}$,
\item the elements of $L^2(\mathbb T_{R_{\sharp}}^d \backslash B_{R_0})$, 
\item the elements of $L^2(\mathbb T_{R_{\sharp}}^d)$ essentially supported outside $B_{R_0}$, 
    \item the elements of $L^2(\mathbb R^d)$ essentially supported in $[-R_{\sharp},R_{\sharp}]^d  \backslash B_{R_0}$,
\item and the elements of $\{ 0 \} \oplus L^2(\mathbb R^d \backslash B_{R_0}) \subset \mathcal H$ whose orthogonal projection onto 
$L^2(\mathbb R^d \backslash B_{R_0})$ is essentially supported in $[-R_{\sharp},R_{\sharp}]^d \backslash B_{R_0}$.
\end{itemize}
If $v \in \mathcal H$ and $\chi \in C^\infty_{\rm comp}(\mathbb R^d)$ is equal to some constant $\alpha$ near $B_{R_0}$, we \emph{define}
\beq\label{eq:mult}
\chi v := (\alpha\mathbf 1_{B_{R_0}} v,  \chi \mathbf 1_{\mathbb R^d \backslash B_{R_0}} v) \in \mathcal H.
\eeq
(for example, using this notation, the requirements on $u$ in the definition of $\cD^\sharp$ 
\eqref{eq:defdomsharp}
are $\chi u \in \cD$ and $(1-\chi)u \in H^2(\mathbb{T}^d_{R_{\sharp}})$ for $\chi$ equal to $1$ near $B_{R_0}$).

If $v\in \mathcal H$ and $R>R_0$, we define
\beq\label{eq:restriction}
v|_{B_{R}} := \big(\mathbf 1_{B_{R_0}} v, (\mathbf 1_{\mathbb R^d \backslash B_{R_0}} v \big)|_{B_{R}}) \in \mathcal{H}_{R_{0}}\oplus L^{2}(B_{R}\backslash B_{R_0}),
\eeq
and, if  $v\in \mathcal H^{\sharp}$,
$$
v|_{B_{R}} := \big(\mathbf 1_{B_{R_0}} v, (\mathbf 1_{\mathbb T_{R_{\sharp}}^d \backslash B_{R_0}} v \big)|_{B_{R}}) \in \mathcal{H}_{R_{0}}\oplus L^{2}(B_{R}\backslash B_{R_0}).
$$
Furthermore, we say that $g\in\mathcal H$ is compactly supported in $B_{R}$ if $g = \chi_0 g$ for some $\chi_0 \in C^\infty_{\rm comp}(\mathbb R^d)$ equal to one near $B_{R_0}$ and supported in $B_{R}$.

Finally, if $R_0\leq r \leq R_\sharp$, we define the partial norms 
$$
\Vert u \Vert_{\mathcal H^\sharp(B_r)} =\Vert u \Vert_{\mathcal
  H(B_r)}:= \Vert u \Vert_{\mathcal H_{R_0} \oplus L^2(B_r\backslash B_{R_0})}, 
\qquad\Vert u \Vert_{\mathcal H^\sharp(B^c_r)} := \Vert \mathbf 1_{\mathbb{T}^d_{R_{\sharp}}\backslash B_{R_0}} u \Vert_{ L^2(\mathbb T^d_{R_\sharp}\backslash B_r)}
$$
and
$$
\Vert u \Vert_{\mathcal H(B^c_r)} := \Vert \mathbf 1_{\mathbb R^d \backslash B_{R_0}} u \Vert_{ L^2(\mathbb{R}^d\backslash B_r)}.
$$

\subsection{Scattering problems fitting in the black-box framework}\label{sec:bbexamples}

The two following lemmas show that both scattering by Dirichlet obstacles with variable coefficients and scattering by penetrable obstacles fit in the black-box framework. For other examples of scattering problems fitting in the black-box framework, see \cite[\S4.1]{DyZw:19}.

\begin{lemma}\mythmname{Scattering by a Dirichlet  obstacle fits in the black-box framework}\label{lem:obstacle}
Let $\obstacle_-, A, c, R_0,$ and $R_1$ be as in Definition \ref{def:EDP}. 
Then 
the family of operators
$$
P_\hsc v := - \hsc^{2} c^2 \nabla \cdot \big(A\nabla v)
$$
with the domain 
\beqs
\cD_D :=H^2(\obstacle_+)\cap H^1_0(\obstacle_+)
%\Big\{ v\in H^1(\obstacle_+), \, \nabla\cdot \big(A\nabla v\big) \in L^2(\obstacle_+), \, v=0 \ton \partial \obstacle_+\Big\}
\eeqs
is a semiclassical black-box operator (in the sense of Definition \ref{def:bb}) with $\omega=c^{-2}$,
$Q_\hsc= -\hsc^2 c^2\nabla\cdot(A\nabla)$, and
\beqs 
\cH_{R_0} = L^2\big( B_{R_0} \cap \obstacle_+; c^{-2}(x)\rd x\big)
\quad \text{ so that } \quad \cH = L^2\big(\obstacle_+; c^{-2}(x)\rd x\big).
\eeqs
Furthermore the corresponding reference operator $P^{\sharp}_{\hsc}$ satisfies \eqref{eq:gro} with $d^{\sharp}=d$.
\end{lemma}

\begin{proof}
The non-semiclassically-scaled version of this lemma
with Lipschitz $\Omega_-$ and $A_{\rm scat}$ and $c\in L^\infty$ and domain 
\beq\label{eq:domain_temp}
\Big\{ v\in H^1(\obstacle_+), \, \nabla\cdot \big(A_{\rm scat}\nabla v\big) \in L^2(\obstacle_+), \, v=0 \ton \partial \obstacle_+\Big\}
\eeq
is proved for $c=1$ in \cite[Lemma 2.1]{LaSpWu:20}. The proof of (\ref{eq:bbreq1}), (\ref{eq:bbreq2a}), and  (\ref{eq:bbreq2}) is essentially
the same in the present semiclassically-scaled setting. The bound (\ref{eq:gro}) follows from comparing the counting function for $P^\sharp_\hsc$ to the counting function for the problem with $c=1$ by a similar argument to \cite[Lemma B.2]{LaSpWu:20}/Appendix \ref{app:gro}, and then using the result for the problem with $c=1$ proven in \cite[Lemma B.1]{LaSpWu:20}.
Finally, by elliptic regularity, the domain \eqref{eq:domain_temp} equals $H^2(\obstacle_+)\cap H^1_0(\obstacle_+)$ since $\Omega_-$ and $A_{\rm scat}$ are smooth in Definition \ref{def:EDP}.
 \end{proof}

\begin{lemma}\mythmname{Scattering by a penetrable Lipschitz obstacle fits in the black-box framework}\label{lem:transmission}
Let $\obstacle_-, A,$ $c$, $\beta$, and $R_0$ be as in Definition \ref{def:transmission}.
Let $\nu$ be the unit normal vector field on $\partial \obstacle_-$  pointing from $\obstacle_-$ into $\obstacle_+$, and let 
$\partial_{\nu,A}$ the corresponding conormal derivative from either $\obstacle_-$ or $\obstacle_+$.
Let
\beqs
\mathcal{H}_{R_{0}}=L^{2}\big(\mathcal{O}_-,c(x)^{-2}\beta^{-1}\rd x\big)\oplus L^{2}\big(B_{R_0}\backslash\overline{\mathcal{O}_-}\big),
\eeqs
so that 
\beqs
\mathcal{H}= L^{2}\big(\mathcal{O}_-;c(x)^{-2}\beta^{-1}\rd x\big)\oplus L^{2}\big(B_{R_0}\backslash\overline{\mathcal{O}_-}\big) \oplus L^{2}\big(\mathbb{R}^{d}\backslash B_{R_0}\big).
\eeqs
Let
\begin{align}\nonumber
\cD :=&\Big\{ v= (v_1,v_2,v_3) \quad\text{ where }\quad v_1 \in 
%H^1\big(\obstacle_-, \nabla\cdot(A_-\nabla)\big),
H^1(\obstacle_-), \quad\nabla\cdot(A_-\nabla v_1) \in L^2(\obstacle_-),
\\ \nonumber
&\quad v_2 \in
%\in H^1\big(B_{R_0}\setminus \overline{\obstacle_-}, \nabla\cdot(A_+\nabla)\big),
H^1\big(B_{R_0}\setminus \overline{\obstacle_-}\big), \quad\nabla\cdot(A_+\nabla v_2)\big)\in 
L^2\big(B_{R_0}\setminus \overline{\obstacle_-}\big),\\ \nonumber
& \quad
 v_3 
 \in H^1\big(\Rea^d\setminus \overline{B_{R_0}}\big),\quad\Delta v_3 \in L^2\big(\Rea^d\setminus \overline{B_{R_0}}\big),
 %\in H^1\big(\Rea^d\setminus \overline{B_{R_0}}, \Delta\big),
 \\ \nonumber
 & \quad  v_1 = v_2\quad\tand\quad \partial_{\nu,A_-} v_1 = \beta\, \partial_{\nu,A_+} v_2 \quad\text{ on } \partial \obstacle_-, \tand\\ 
 & \quad  v_2 = v_3\quad\tand \quad \partial_{\nu} v_2 = \partial_{\nu} v_3\quad \text{ on } \partial B_{R_0} \qquad \Big\}
 \label{eq:domain_transmission}
\end{align}
(observe that the conditions on $v_2$ and $v_3$ on $\partial B_{R_0}$ in the definition of $\cD$ are such that $(v_2,v_3) \in H^1(\Rea^d\setminus \overline{\obstacle_-})$ and $ \nabla\cdot(A_+\nabla(v_2,v_3))\in L^2(\Rea^d\setminus \overline{\obstacle_-})$% \nabla\cdot(A_+\nabla))$
). 
Then the family of operators 
\beqs
P_\hsc v:=- \hsc^{2}\Big(c^{2}\nabla\cdot(A_-\nabla  v_{1}),\nabla\cdot(A_+\nabla v_{2}),\Delta v_{3}\Big),
\eeqs
defined for $v=(v_1,v_2,v_3)$, is a semiclassical black-box operator 
(in the sense of Definition \ref{def:bb}) on $\mathcal H$, with $Q_\hsc = -\hsc^2\Delta$, and any $R_1>R_0$.
Furthermore, the corresponding reference operator $P^{\sharp}_\hsc$ satisfies \eqref{eq:gro} with $d^{\sharp}=d$.
\end{lemma}

\begin{proof}
The non-semiclassically-scaled version of this lemma was proved for $c=1$ in \cite[Lemma 2.3]{LaSpWu:20}. The proof of (\ref{eq:bbreq1}), (\ref{eq:bbreq2a}), and  (\ref{eq:bbreq2}) is essentially
the same in the present semiclassically-scaled setting. The proof of the bound (\ref{eq:gro}) 
is similar to the the analogous proof for $c=1$ and $A$ Lipschitz in \cite[Lemma B.1]{LaSpWu:20}; for completeness we include the proof in \S\ref{app:gro}.
 \end{proof}

\bre
Lemma \ref{lem:obstacle} has the obstacle $\obstacle_-$ in the black box (i.e., in $B_{R_0}$) but not all the variation of the coefficients $A$ and $c$ (which are contained in $B_{R_1} \supset B_{R_0}$). In contrast, Lemma \ref{lem:transmission} has \emph{both} the obstacle $\obstacle_-$ \emph{and} all the variation of the coefficients $A$ and $c$ in the black box. The transmission problem also fits in the black-box framework with some of the variation of the coefficients outside the black box (i.e., in $B_{R_1}$), but we do not need this formulation to prove Theorem \ref{thm:transmission}.
\ere

\subsection{A black-box functional calculus for $P^{\sharp}_\hsc$}\label{BBFC}

The operator $P^\sharp_\hsc$ on the torus with domain
$\mathcal{D}^\sharp$ is self-adjoint with compact resolvent
\cite[Lemma 4.11]{DyZw:19}, hence we can describe the Borel functional
calculus \cite[Theorem VIII.6]{ReeSim72} for this operator explicitly in terms of the orthonormal
basis of eigenfunctions $\phi^\sharp_j \in \hilbert^\sharp$ 
(with
eigenvalues $\lambda_j^\sharp$, appearing with multiplicity and depending on $\hsc$):~for $f$ a real-valued Borel function on $\RR,$
$f(P_\hsc^\sharp)$ is self-adjoint with domain
\begin{equation*}
\domain_f:=\bigg \{\sum a_j \phi_j^\sharp \in \cH^\sharp
\,\,: \,\,
%\sum \big \lvert a_j \big \rvert^2+
\sum \big \lvert f(\lambda^\sharp_j)
a_j \big \rvert^2 <\infty\bigg\},
\end{equation*}
and if $v =\sum a_j \phi_j^\sharp\in \domain_f$ then
\begin{equation} \label{eq:deffP}
f(P_\hsc^\sharp)(v):=\sum a_j f(\lambda^\sharp_j) \phi^\sharp_j.
\end{equation}
For $f$ a bounded Borel function, $f(P^\sharp)$ is a bounded
operator, hence in this case we can dispense with the definition of
the domain and allow $f$ to be complex-valued.

For $m \geq 1$, we then define
$\mathcal D_\hsc^{\sharp,m}$ as the domain of $(P^\sharp_\hsc)^m$ equipped with the norm 
\begin{equation} \label{eq:BB:norms}
\Vert v \Vert_{\mathcal D_\hsc^{\sharp, m}} :=  \Vert  v\Vert_{\mathcal H^{\sharp}} + \Vert (P^{\sharp}_{\hsc})^m  v\Vert_{\mathcal H^{\sharp}},
\end{equation}
and $\mathcal D_\hsc^{\sharp,-m}$ as its dual (note that, in the exterior of the black box, the regularity imposed in the definition of $\mathcal D_\hsc^{\sharp,m}$ is that of periodic functions on the torus with $2m$ derivatives in $L^2$).
We define also the partial norms, for $m>0$, 
$\Vert v \Vert_{\mathcal D_\hsc^{\sharp, m}(B)} :=  \Vert  v\Vert_{\mathcal H^{\sharp}(B)} + \Vert (P^{\sharp}_{\hsc})^m  v\Vert_{\mathcal H^{\sharp}(B)}$, where $B = B_r$ or $B=B_r^c$ with $R_0 \leq r \leq R_\sharp$. In addition, we let
\beq\label{eq:Dinfty}
\mathcal D_\hsc^{\sharp,\infty} := \bigcap_{m\geq0} \mathcal D_\hsc^{\sharp,m},
\eeq
so that $v \in \mathcal{D}_\hsc^{\sharp,\infty}$ iff $(P_\hsc^{\sharp})^m v \in \mathcal D_\hsc^\sharp$ for all $m\in \mathbb{Z}^+$.

The following theorem is proved in \cite[Pages 23 and 24]{Da:96}; see also \cite[Theorem VIII.5]{ReeSim72}.

\begin{theorem} \label{thm:fundfc}
The Borel functional calculus enjoys the following properties.
  \begin{enumerate}
\item $f\rightarrow f(P^{\sharp}_{\hsc})$ is a $\star$-algebra homomorphism. 
\label{it:fc3}
\item for $z \notin \mathbb R$, 
if $r_z(w):=(w-z)^{-1}$ then $r_z(P^\sharp)= (P^{\sharp}_{\hsc} - z)^{-1}$. \label{it:fc4}
\item If $f$ is bounded, $f(P^\sharp_\hsc)$ is a bounded operator
  for all $\hsc$, with $\Vert f(P^{\sharp}_{\hsc}) \Vert_{\mathcal L(\mathcal H^{\sharp})} \leq \sup_{\lambda \in \mathbb R}|f(\lambda)|$. \label{it:fc5}
\item If $f$ has disjoint support from $\operatorname{Sp}P_\hsc^{\sharp}$, then $f(P^{\sharp}_{\hsc}) = 0$. \label{it:fc6}
\end{enumerate}
\end{theorem}

In describing the \emph{structure} of the operators produced by the
functional calculus, at least for well-behaved functions $f,$ it is
useful to recall the Helffer--Sj\"ostrand construction of the
functional calculus \cite{HeSj:89}, \cite[\S2.2]{Da:96} (which can also be used to prove the spectral
theorem to begin with; see \cite{Da:95}).

We say that $f \in \mathcal A$ if $f \in C^\infty(\mathbb R)$ and
there exists $\beta < 0$, such that, for all $r>0$, there exists  $C_r>0$ such that $|f^{(r)}(x)| \leq C_r (1+|x|^2)^{(\beta - r)/2}$.

Let $\tau \in C^\infty(\mathbb R)$ be such that $\tau(s) = 1$ for $|s|\leq 1$ and $\tau(s) = 0$ for $|s|\geq 2$. 
Finally, let $\frak n \geq 1$.
We define an $\frak n$-almost-analytic
extension of $f$, denoted by $\widetilde f$, by
$$
\widetilde f(z) :=\left( \sum_{m=0}^{\frak n} \frac{1}{m!} \big(\partial^m f(\operatorname{Re}z)\big)\,(\ri\operatorname{Im}z)^m \right) \tau\left(\frac{\operatorname{Im}z}{\langle\operatorname{Re}z\rangle}\right)
$$
(observe that $\widetilde{f}(z)=f(z)$ if $z$ is real).
For $f\in \mathcal A$, we define
\beq\label{eq:HS1}
f(P_{\hsc}^{\sharp}) := - \frac 1 \pi \int_{\mathbb C} \frac{\partial \widetilde f}{\partial \bar z}(P^{\sharp}_{\hsc} - z)^{-1} \; \rd x\rd y,
\eeq
where $\rd x\rd y$ is the Lebesgue measure on $\mathbb C$. The integral on
the right-hand side of \eqref{eq:HS1} converges; see, e.g., \cite[Lemma 1]{Da:95}, \cite[Lemma 2.2.1]{Da:96}.  This definition can be shown to be independent of the choices of
$\frak n$ and $\tau,$ and to agree with the operators defined by the
Borel functional calculus for $f \in \mathcal{A}$; see \cite[Theorems 2-5]{Da:95}, \cite[Lemmas 2.2.4-2.2.7]{Da:96}. 

 When $P$ is a
self-adjoint elliptic semiclassical differential operator on a compact manifold,
the Helffer--Sj\"ostrand construction can be used to show that $f(P)$ is a
pseudodifferential operator \cite{HeSj:89}.  Here, in the presence of
a black box, it can instead be used to show that, modulo residual
errors, $f(P^\sharp_\hsc)$ agrees with $f(Q_\hsc)$ on the region of the
torus outside the black box, with the latter being a
  pseudodifferential operator. Furthermore, the operator wavefront set
of $f(Q_\hsc)$ can be seen to be included in $q_\hsc^{-1}(\supp f)$. 
We now state these results, obtained originally in \cite{Sj:97}. 

We say that $E_\infty \in \mathcal L(\mathcal H^{\sharp})$ is $O(\hsc^\infty)_{\mathcal D_\hsc^{\sharp,-\infty} \rightarrow \mathcal D_\hsc^{\sharp,\infty}}$ if, for any
$N>0$ and any $m>0$, there exists $C_{N,m}>0$ such that
\begin{equation}\label{eq:BBresidual}
\Vert E_\infty \Vert _{\mathcal D_\hsc^{\sharp,-m} \rightarrow \mathcal D_\hsc^{\sharp,m}} \leq C_{N,m} \hsc^N
\end{equation}
(compare to \eqref{eq:residual} below).
Operators in the functional calculus are pseudo-local in the following sense.

\begin{lemma}
\label{thm:funcloc1}
Suppose $f\in \mathcal{A}$ is independent of $\hsc$, and $\psi_1, \psi_2 \in C^\infty(\mathbb T_{R_{\sharp}}^d)$ are constant
near $B_{R_0}$. If $\psi_1$ and $\psi_2$ have disjoint supports, then
\beq\label{eq:pseudolocal1}
\psi_1 f(P^{\sharp}_{\hsc}) \psi_2= O(\hsc^\infty)_{\mathcal D_\hsc^{\sharp,-\infty} \rightarrow \mathcal D_\hsc^{\sharp,\infty}}.
\eeq
\end{lemma}

% \begin{proof}
% Since $\psi_1,\psi_2, f(P^{\sharp}_{\hsc}) \in
% \Psi^{-\infty}(\mathbb{T}^d_{R_\sharp}),$ the properties of the
% semiclassical wavefront set from Appendix~\ref{app:sct} (in
% particular, \eqref{eq:WFprod}) imply that
% $$
% \WFh \psi_1 f(P^{\sharp}_{\hsc}) \psi_2=\emptyset,
% $$
% and the desired mapping property follows from \eqref{eq:WFempty}.
%   \end{proof}

\bpf
In the usual case of a smooth manifold with boundary, this result
follows from the fact that $f(P^{\sharp}_{\hsc})$ is a
pseudodifferential operator, and hence pseudo-local.  Here, it follows from combining the corresponding result about the
resolvent \cite[Lemma 4.1]{Sj:97} (i.e., \eqref{eq:pseudolocal1} with
$f(w):= (w-z)^{-1})$) with \eqref{eq:HS1} and then integrating (as
discussing in a slightly different context in \cite[Paragraph after
proof of Lemma 4.2]{Sj:97}).
\epf

\

Furthermore, we can show from \cite[\S4]{Sj:97} that, modulo a negligible term,  away from the black-box the functional calculus is given by the semiclassical pseudodifferential calculus in the following sense. 
The following lemma uses the notion of semiclassical pseudodifferential operators on $\mathbb
T^d_{R_{\sharp}}$ (including the concept of the \emph{operator wavefront set} $\operatorname{WF}_\hsc$), recapped in Appendix \ref{app:sct}.

\begin{lemma} \label{thm:funcloc2}
Suppose $f\in C^\infty_{\rm comp}(\mathbb R)$ is independent of $\hsc$. If $\chi\in C^\infty(\mathbb T_{R_{\sharp}}^d)$ is equal to zero near $B_{R_0}$, then,
\beq\label{july12}
\chi f(P^{\sharp}_\hsc) \chi = \chi f(Q_\hsc) \chi + O(\hsc^\infty)_{\mathcal D_\hsc^{\sharp,-\infty} \rightarrow \mathcal D_\hsc^{\sharp,\infty}}.
\eeq
Furthermore,
$f(Q_\hsc)\in \Psi^{-\infty}_\hsc(\mathbb T^d_{R_\sharp})$ with 
\begin{equation}\label{fcsymbol}
\sigma_{\hsc}(f(Q_\hsc)) = f(q_\hsc)
  \end{equation}
and
\beq\label{microsupport}
\operatorname{WF}_\hsc f(Q_\hsc)  \subset q_\hsc^{-1}(\supp f).
\eeq
If, instead, $f \in C^\infty(\RR)$ is identically equal to
  $1$ near $+\infty,$ then
  $f(Q_\hsc)\in \Psi^{0}_\hsc(\mathbb T^d_{R_\sharp})$ and
  \eqref{july12}, \eqref{fcsymbol}, \eqref{microsupport} continue to hold.
\end{lemma}
(Here we are adopting the convention that if $\rho_0=(x_0,
  \zeta_0) \in \Tbar^*\torus_{R_\sharp}^d$ lies at fiber-infinity (see the section ``Phase space" in Appendix \ref{app:sct}), then the
  notion of support is to be interpreted in the following generalized sense:
  $q_\hsc(\rho_0)=+\infty$ and this is in $\supp f$ if $f=1$ near $+\infty.$)

\

\begin{proof}
First, assume $f$ has compact support.  By \cite[Lemma 4.2 and the subsequent two paragraphs]{Sj:97},
$$
\chi f(P^{\sharp}_\hsc) \chi = \chi f(Q_\hsc) \chi + O(\hsc^\infty)_{\mathcal D_\hsc^{\sharp,-\infty} \rightarrow \mathcal D_\hsc^{\sharp,\infty}}.
$$
The results of Helffer-Robert \cite{HeRo:83} (see the account in
\cite{Rob87} and in particular Remarques III-14 for verification
  of the hypotheses on $f$) imply that for $f$ compactly supported,
  $f(Q_\hsc) \in \Psi_\hsc^{-\infty},$ with principal symbol
  $f(q_\hsc)$.

  That the analogous statements
  hold for $f=1$ near $+\infty$ instead simply follows by noting that
  for such a function $f$, $g(s)=1-f(s)$ is zero for $s>C$ for some
  $C.$  Then $f(Q_\hsc)=I-g(Q_\hsc)$; since $Q_\hsc$ is bounded
  below, we may assume without loss of generality that $g$ is
  compactly supported Thus the previous results show that
  \eqref{july12}, \eqref{fcsymbol} hold for $g(Q_\hsc)$, which is in
  $\Psi_\hsc^{-\infty}.$  We thus obtain \eqref{july12}, \eqref{fcsymbol} for
  $f(Q_\hsc),$ which lies in $\Psi_\hsc^0$
  with symbol $f(q_\hsc),$ hence we have established \eqref{july12},
  \eqref{fcsymbol} under either of our hypotheses on $f.$ 

It remains to show that
  $\operatorname{WF}_\hsc f(Q_\hsc) \subset
  q_\hsc^{-1}(\operatorname{supp} f)$.  To this end, pick any $\rho_0
  \notin  q_\hsc^{-1}(\operatorname{supp} f);$ we aim to show $\rho_0
  \notin \WFh f(Q_\hsc).$ There exists a
  smooth function $g$ on $\RR$ with $g(q_\hsc(\rho_0))=1$ and $\supp\, g \cap
  \supp\, f=\emptyset.$  We may take $g$ to be either compactly
  supported (if $\rho_0$ is in $T^*\torus_{R_\sharp}^d$) or equal to $1$ near $+\infty$ (if
  $\rho_0$ is at fiber-infinity).  Then by Part \ref{it:fc3} of Theorem \ref{thm:fundfc}
\beq\label{eq:fg}
f(Q_\hsc) g(Q_\hsc)=g(Q_\hsc) f(Q_\hsc)=0
\eeq
(the Borel calculus is a homomorphism).
Since $\sigma_\hsc(g(Q_\hsc))=1$ by \eqref{fcsymbol}, $g(Q_\hsc)$ is
elliptic at $\rho_0.$

Now pick $b \in C^\infty (\Tbar^*\mathbb T^d_{R_\sharp})$ equal to $1$
in a small neighbourhood of $\rho_0$ and supported on the elliptic set
of $g(Q_\hsc).$ Thus, writing $B=\Optorus(b),$ $\rho_0 \notin \WFh (I-B)$ and $\WFh B$ lies in
the elliptic set of $g(Q_\hsc)$.  Then by Theorem ~\ref{thm:para}, we may factor
$$
B=Zg(Q_\hsc)+R 
$$
with $Z \in \Psi_\hsc^0$ and $\rho_0 \notin \WFh R$ (by \eqref{eq:WFempt}).  Now write
\begin{align*}
  f(Q_\hsc) &= B f(Q_\hsc)+(I-B)  f(Q_\hsc)\\
  &=  Z g(Q_\hsc) f(Q_\hsc)+R f(Q_\hsc)+(I-B)  f(Q_\hsc),
\end{align*}
The first term on the right-hand side is zero by \eqref{eq:fg}.  The point $\rho_0$ is not in the
semiclassical operator wavefront set
of the second term or third terms since it is not in $\WFh R$ or $\WFh(I-B)$
(see \eqref{eq:WFprod}).  Hence by \eqref{eq:WFsum}, $\rho_0 \notin
\WFh f(Q_\hsc)$, as desired.
\end{proof}

\section{Proof of Theorem \ref{thm:mainbb} (the main result in the black-box framework)}\label{sec:blackboxresult}

The decomposition (\ref{eq:maindec}) is defined in \S \ref{subsec:abdec}
(and illustrated schematically in Figures \ref{fig:split} and \ref{fig:split_uL}). The estimates (\ref{eq:decHF}) and  (\ref{eq:decLF1})--(\ref{eq:decLF5})  are proved in \S \ref{subsec:high} and
 \ref{subsec:low} respectively.
 
\subsection{The decomposition} \label{subsec:abdec}

Let $\varphi \in C^\infty_{\rm comp}(\mathbb{R}^{d})$ be equal to one in 
$B_{R}$ and supported in $B_{R_{\sharp}}$. 
 For $v \in \mathcal H$,
we define
$$
M_\varphi v := \varphi v,
$$
where the multiplication is in the sense of \eqref{eq:mult}.
Let $u \in \mathcal D_{\rm out}$ be solution to 
$$
(P_{\hsc} - 1)u = g,
$$
and let 
$$
w := M_\varphi u.
$$
We view $w$ as an element of $\mathcal H^{\sharp}$ and work in the torus $\mathbb T^d_{R_{\sharp}}$.

We now define our frequency cut-offs. By (\ref{eq:propq_ell}), there exists $\widetilde{\mu} > 1$ and $c_{\rm ell}>0$ such that
\beqs
|\xi| \geq \widetilde{\mu} \quad\text{ implies that } \langle \xi \rangle^{-2}(q_\hsc(x, \xi) -1) \geq c_{\rm ell} >0.
\eeqs
Therefore, by \eqref{eq:Qnew}, 
there exists $\mu > 1$ such that 
\begin{equation} \label{eq:newdefmu}
q_\hsc(x,\xi) \geq \mu \quad\text{ implies that } \langle \xi \rangle^{-2}(q_\hsc(x, \xi) -1) \geq c_{\rm ell} >0.
\eeq
We increase $\mu$ further, if necessary, so that
\beq\label{eq:John1}
\big\{ (x,\xi) \, :\, |q_\hsc(x,\xi)|\geq \mu\big\}=\big\{ (x,\xi) \, :\, q_\hsc(x,\xi)\geq \mu\big\}
\eeq
(note that the conditions  imposed on $q_\hsc(x,\xi)$ in \S\ref{subsec:bb} allow it to be $<0$ for some $(x,\xi)$).

Let
$\psi \in C^\infty_{\rm comp}(\mathbb R)$ be such that
\begin{equation} \label{eq:psi}
\psi = 
\begin{cases}
1 \text{ in } B(0,1), \\
0 \text{ in } (B(0, 2))^c.
\end{cases}
\end{equation}
We now fix $1 \leq \mu' \leq \mu/ 2$,
and define
\begin{equation} \label{eq:psimu}
\psi_\mu(\cdot) := \psi\left(\frac{\cdot}{\mu}\right), \qquad \psi_{\mu'}(\cdot) := \psi\left(\frac{\cdot}{\mu'}\right).
\end{equation}
These definitions imply that
\begin{equation} \label{eq:supppsi}
(1- \psi_{\mu'})(1- \psi_{\mu})=(1- \psi_{\mu})
\end{equation}
(since $2\mu' \leq \mu$), and
\begin{equation} \label{eq:proppsimup}
1 \notin \operatorname{supp}(1 - \psi_{\mu'})
\end{equation}
(since $\mu'\geq 1$). Let 
\beq\label{eq:Lambda}
\Lambda := 5 \mu
\eeq
(note that, by \eqref{eq:newdefmu}, both $\mu$ and $\Lambda$ only depend on $q_\hsc$),
and observe that
\begin{equation} \label{eq:propLambda}
\operatorname{supp} \psi_{\mu} \subset [-\Lambda, \Lambda].
\end{equation}

We define, by the Borel functional calculus for $P^{\sharp}_{\hsc}$ (Theorem \ref{thm:fundfc}), in $\mathcal L(\mathcal H^{\sharp})$
\begin{equation}\label{eq:PiL}
\Pilow := \psi_{\mu}(P^{\sharp}_{\hsc}),
\end{equation}
and additionally
\beqs
\Pihigh :=  (1 - \psi_{\mu})(P^{\sharp}_{\hsc}) = I-\Pilow \quad\tand\quad \Pihigh' := (1 - \psi_{\mu'})(P^{\sharp}_{\hsc}).
\eeqs
By (\ref{eq:supppsi}) and the fact the Borel functional calculus is an algebra homomorphism (Part \ref{it:fc3} of Theorem \ref{thm:fundfc}), 
\begin{equation} \label{eq:PiPi}
\Pihigh' \Pihigh = \Pihigh.
\end{equation}
By Part  \ref{it:fc5} of Theorem \ref{thm:fundfc}, the operators $\Pilow, \Pihigh,$ and $\Pihigh'$ are bounded on $\Hilb^{\sharp}$, with
\begin{equation} \label{eq:boundPi}
\Vert \Pilow\Vert_{\mathcal L(\mathcal H^{\sharp})}, \; \Vert \Pihigh\Vert_{\mathcal L(\mathcal H^{\sharp})}, \; \Vert \Pihigh'\Vert_{\mathcal L(\mathcal H^{\sharp})} \leq 1,
\end{equation}
and they commute with $P^{\sharp}_{\hsc}$ by Part \ref{it:fc3} of Theorem \ref{thm:fundfc}.

Since $u \in \mathcal D_{\rm loc}$ (defined by \eqref{eq:Dloc}), the definition of $\mathcal D^\sharp$ (\ref{eq:defdomsharp}), (\ref{eq:bbreq1}), and the fact that $\varphi$ is compactly supported imply that $w \in\mathcal D^\sharp$. 
By the definition of $\psi_\mu$ (\ref{eq:psimu}), (\ref{eq:deffP}), and the fact that $\operatorname{Sp}P_\hsc^\sharp$ is discrete,
 $\Pilow w$ projects non-trivially only on a finite number of eigenspaces of $P^\sharp_\hsc$, and thus   $\Pilow w \in\mathcal D_\hsc^{\sharp, \infty}$. Therefore $\Pihigh w = w - \Pilow w\in \mathcal D^\sharp$.
We now define
\beq\label{eq:newlabel1}
\uH := \Pihigh w \in \mathcal D^{\sharp}, \quad \uL := \Pilow w\in \mathcal D_\hsc^{\sharp, \infty}.
\eeq
We show in \S\ref{subsec:low} below that we can split $\uL$ as
\begin{equation} \label{eq:LF_dec}
\uL = u_{\mathcal A} + u_\epsilon,
\end{equation}
where $u_{\mathcal A} \in\mathcal D_\hsc^{\sharp, \infty}$ satisfies
(\ref{eq:decLF0})--(\ref{eq:decLF4}) (or (\ref{eq:decLF5}) if
$\rho=1$), and that $\uH$ and $u_\epsilon$ satisfy
\begin{equation} \label{eq:HF_eq}
\Vert \uH \Vert_{\mathcal H^{\sharp} } + \big\| P^{\sharp}_{\hsc}\uH\big\|_{\mathcal H^{\sharp} }  \lesssim \Vert g \Vert_{\mathcal H},
\end{equation}
and
\begin{equation} \label{eq:eps_eq}
\Vert u_\epsilon \Vert_{\mathcal H^{\sharp} } + \big\| P^{\sharp}_{\hsc} u_\epsilon\big\|_{\mathcal H^{\sharp} }  \lesssim \Vert g \Vert_{\mathcal H},
\end{equation}
with additionally $u_\epsilon \in \cD_\hsc^{\sharp,\infty}$.
We then define
$$
u_{H^2} := \uH + u_{\epsilon} \in \cD^\sharp,
$$
so that the decomposition (\ref{eq:maindec}), (\ref{eq:decHF}) and (\ref{eq:decLF0})--(\ref{eq:decLF4}) (or (\ref{eq:decLF5}) if $\rho=1$) holds. Our splitting strategy is
summed-up in Figure \ref{fig:split}; with an overview of the splitting of the low-frequency component $\uL$ in Figure \ref{fig:split_uL}.

In \S \ref{subsec:high} we prove the estimate (\ref{eq:HF_eq}) for $\uH$.
In \S\ref{subsec:low} we prove that the decomposition (\ref{eq:LF_dec}) holds, with $u_{\mathcal A}$ satisfying (\ref{eq:decLF0})--(\ref{eq:decLF4}) (or (\ref{eq:decLF5}) if $\rho=1$) and $u_\epsilon$ satisfying (\ref{eq:eps_eq})  We highlight that all the arguments from now on consider $\hsc \in \subsetH$.

\begin{figure}
%\begin{adjustwidth}{0em}{-18em}
\hspace{-18em}
 \begin{tikzpicture}%
  [>=stealth,
   shorten >=1pt,
   align = center,
   node distance=3cm and 4.5cm,
   on grid
  ]
\node (1)  {$u$};
\node (2) [below=of 1] {\begin{minipage}{\textwidth}
            \begin{gather*} 
            w := \varphi u \\ 
            \text{\small considered as an element} \\
            \text{\small of the reference torus} 
            \end{gather*}
        \end{minipage}};
\node (31) [below=of 2] {\begin{minipage}{\textwidth}
            \begin{gather*} 
            \uL \\ 
            \text{\small low-frequency part} 
            \end{gather*}
        \end{minipage}};
\node (32) [right=of 31] {\begin{minipage}{\textwidth}
            \begin{gather*} 
            \uH \\ 
            \text{\small high-frequency part} 
            \end{gather*}
        \end{minipage}};
\node (42) [below=of 31] {
\begin{minipage}{\textwidth}
            \begin{gather*} 
            u_{\mathcal A}^{\infty} \\ 
            \text{\small analytic away from $B_{R_0}$} 
            \end{gather*}
        \end{minipage}};
\node (41) [left=of 42] {\begin{minipage}{\textwidth}
            \begin{gather*} 
            u_{\mathcal A}^{R_0} \\ 
            \text{\small regular near $B_{R_0}$} 
            \end{gather*}
        \end{minipage}};
\node (43) [right=of 42] {\begin{minipage}{\textwidth}
            \begin{gather*} 
            u_{\epsilon} \\ 
            \text{\small not regular but small} 
            \end{gather*}
        \end{minipage}};
\node (44) [right=of 43] {$\;$};
\node (51) [below=of 42] {$u_{\mathcal A}$};
\node (52) [below=of 44] {$u_{H^2}$};

\path[->]
(1) edge node[right] {$\varphi \in C^\infty_{\rm comp}$} (2)
(2) edge node[left] {$\Pilow$} (31)
    edge node[right] {$\hspace{0.2cm}\Pihigh$} (32)   
(31) edge node {} (42)
(31) edge node[left] {} (41)
(31) edge node {} (43)
(41) edge node {} (51)
(42) edge node {} (51)
(43) edge node {} (52)
(32) edge node {} (52)
;
\end{tikzpicture}
%\end{adjustwidth}
\caption{Splitting of the Helmholtz solution} \label{fig:split}
\end{figure}
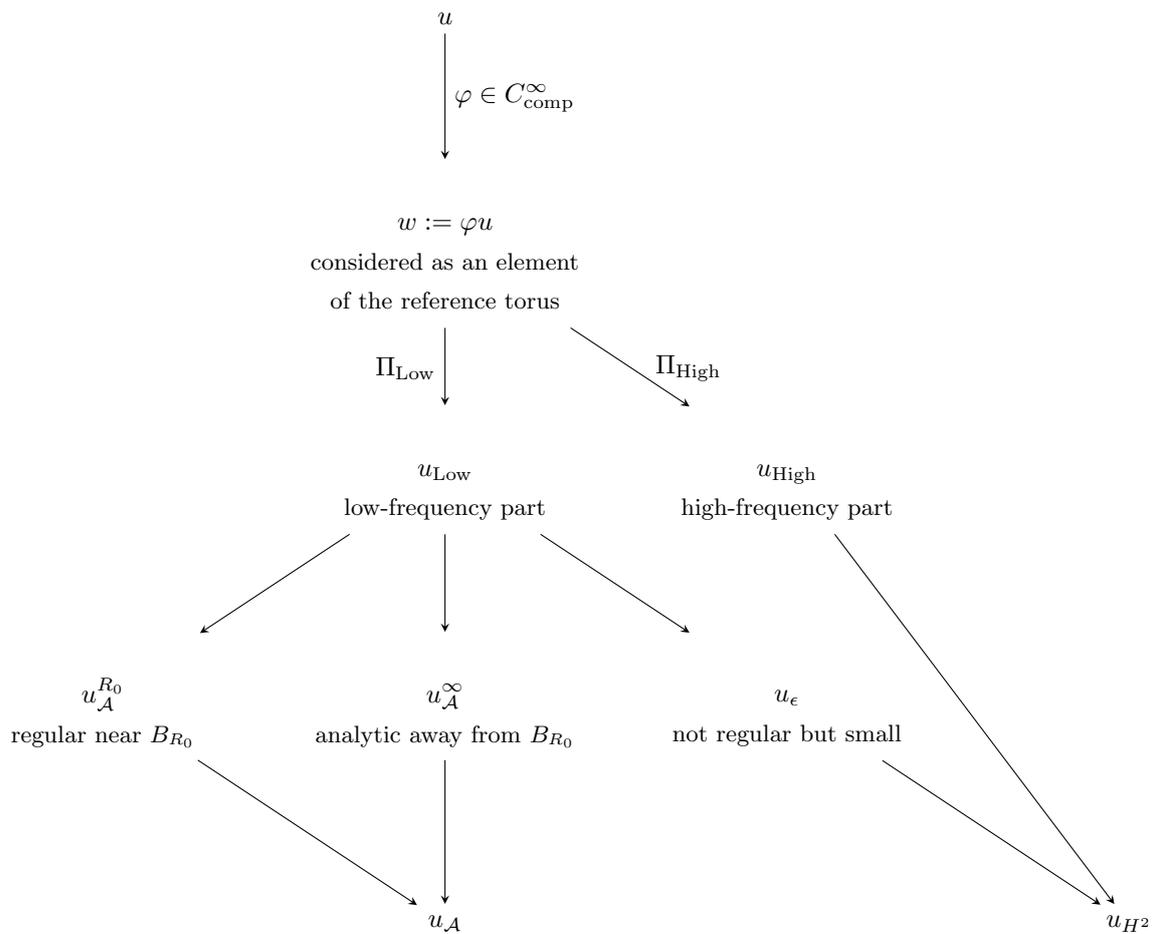

\subsection{Proof of the bound (\ref{eq:HF_eq}) on $\uH$ (the high-frequency component)} \label{subsec:high}

We proceed in three steps: we first use the abstract information we have about $P^\sharp_\hsc$ to bound $\Pihigh w$ by $\Vert g \Vert_{\mathcal H}$ modulo a commutator term living away from the black box $B_{R_0}$. We then use 
Lemmas \ref{thm:funcloc1} and \ref{thm:funcloc2}  to show that this commutator is given, up to negligible terms, by the semiclassical pseudodifferential calculus on the torus $\mathbb T^d_{R_{\sharp}}$. Finally, we work in the torus and use the semiclassical elliptic-parametrix construction (Theorem \ref{thm:para}) to estimate this commutator, seen as a semiclassical pseudodifferential operator on $\mathbb T^d_{R_{\sharp}}$.

\subsubsection*{Step 1:~An abstract estimate in $\mathcal H^\sharp$}
Since $\Pihigh$ commutes with $P^{\sharp}_{\hsc}$,
\begin{align} \nonumber
(P^{\sharp}_{\hsc}-I)(\Pihigh w) &= \Pihigh (P^{\sharp}_{\hsc}-I)(w) \\ &=  \Pihigh (P_{\hsc}-I)(w) = \Pihigh \varphi g+\Pihigh [P_{\hsc}, M_\varphi] u
= \Pihigh \varphi g+\Pihigh [P^{\sharp}_{\hsc}, M_\varphi] u,\label{eq:eqPiH}
\end{align}
where we used the fact that we can replace $P^{\sharp}_{\hsc}$ by  $P_{\hsc}$ (and vice versa) on $\supp \varphi\subset B_{R_0}$ by (\ref{eq:bbreq1}) and (\ref{eq:defref})). For $\lambda \in \mathbb R$, let
$$
f(\lambda) := (\lambda - 1)^{-1}(1 - \psi_{\mu'})(\lambda),
$$ 
where $f \in C_0(\mathbb R)$ (defined by \eqref{eq:C0}) by (\ref{eq:proppsimup}).
Using (\ref{eq:PiPi}), the fact that the Borel calculus in an algebra
homomorphism (Part \ref{it:fc3} of Theorem \ref{thm:fundfc}), and
finally (\ref{eq:eqPiH}), we get
\beq\label{eq:PiHcom}
\Pihigh w = \Pihigh' \Pihigh w 
= f(P_\hsc ^\sharp)(P_\hsc ^\sharp-I) \Pihigh w = f(P_\hsc ^\sharp) \big(  \Pihigh \varphi g+\Pihigh [P^{\sharp}_{\hsc}, M_\varphi] u \big).
\eeq
Since $f\in C_0(\mathbb R),$ $f(P^\sharp_\hsc)$ is uniformly bounded from $\Hilb^{\sharp}\to \Hilb^{\sharp}$ by Part \ref{it:fc5} of Theorem \ref{thm:fundfc}. 
Combining this fact with  (\ref{eq:PiHcom}), we obtain
$$
\norm{\Pihigh w}_{\Hilb^{\sharp}} \lesssim  \norm{\Pihigh \varphi g}_{\Hilb^{\sharp}}+ \norm{\Pihigh [P^{\sharp}_{\hsc},M_\varphi] u}_{\Hilb^{\sharp}}.
$$
Writing $P^{\sharp}_{\hsc}\Pihigh w = \Pihigh w + (P^{\sharp}_{\hsc}-I)\Pihigh w$ and using (\ref{eq:eqPiH}) again, we obtain
$$
\norm{\Pihigh w}_{\Hilb^{\sharp}} + \norm{P^{\sharp}_{\hsc} \Pihigh w}_{\Hilb^{\sharp}}  \lesssim  \norm{\Pihigh \varphi g}_{\Hilb^{\sharp}}+ \norm{\Pihigh [P^{\sharp}_{\hsc},M_\varphi] u}_{\Hilb^{\sharp}}.
$$
Hence, by (\ref{eq:boundPi})
\begin{align}
\norm{\Pihigh w}_{\Hilb^{\sharp}}+ \norm{P^{\sharp}_{\hsc} \Pihigh w}_{\Hilb^{\sharp}}  &\lesssim  \norm{\varphi g}_{\Hilb^{\sharp}}+ \norm{\Pihigh [P^{\sharp}_{\hsc},M_\varphi] u}_{\Hilb^{\sharp}} \nonumber  \\
& \lesssim \norm{ g}_{\Hilb}+ \norm{\Pihigh [P^{\sharp}_{\hsc},M_\varphi] u}_{\Hilb^{\sharp}}. \label{eq:high1}
\end{align}

\subsubsection*{Step 2:~Viewing $\Pihigh [P^{\sharp}_{\hsc},M_\varphi]$ as a semiclassical pseudodifferential operator on $\mathbb T^d_{R_{\sharp}}$}

To prove \eqref{eq:HF_eq} from \eqref{eq:high1}, it therefore remains to bound the commutator term $\Pihigh [P^{\sharp}_{\hsc},M_\varphi] u$. Since $[P^{\sharp}_{\hsc},M_\varphi]$ lives away from $\mathcal H_{R_0}$,
we consider the high-frequency cut-off \emph{in terms of the semiclassical pseudodifferential
calculus} thanks to Lemma \ref{thm:funcloc2}. 

Since $\varphi$ is compactly supported in $B_{R_{\sharp}}$ and equal to one near $B_{R_0}$, 
in $\mathcal H^\sharp$ we can write 
$[P^{\sharp}_{\hsc},M_\varphi] $ as (using the notation in \S\ref{subsec:bb})
\begin{equation} \label{eq:comsupp}
[P^{\sharp}_{\hsc},M_\varphi] = (0, [Q_\hsc, \varphi]) = (0, \phi [Q_\hsc, \varphi] \phi) =  (0, [Q_\hsc, \varphi] \phi)
\end{equation}
where $\phi \in C^\infty_{\rm comp}(\mathbb R^d)$ is supported in $B_{R_{\sharp}}$, equal to zero near $B_{R_0}$, and such that
\begin{equation} \label{eq:suppphi}
\phi = 1 \text{ near } \operatorname{supp}\nabla \varphi.
\end{equation}
Let $\chi \in C_c^\infty(\mathbb R^d)$ be supported in $B_{R_{\sharp}}$, equal to zero near $B_{R_0}$, and equal to one near $\operatorname{supp}\phi$. Using (\ref{eq:comsupp}) and Lemma \ref{thm:funcloc1} (i.e., the pseudo-locality of the functional calculus) with $\psi_1 =1-\chi$ and $\psi_2=\chi\phi=\phi$, we obtain that
\begin{align} 
\Pihigh [P^{\sharp}_{\hsc},M_\varphi] &= \chi \Pihigh \chi \phi [P^{\sharp}_{\hsc},M_\varphi] \phi  + O(\hsc^\infty)_{ {\mathcal D} _\hsc^{\sharp,-\infty} \rightarrow  {\mathcal D} _\hsc^{\sharp, \infty}} \nonumber \\
 &= \chi \Pihigh \chi [P^{\sharp}_{\hsc},M_\varphi] \phi  + O(\hsc^\infty)_{ {\mathcal D} _\hsc^{\sharp,-\infty} \rightarrow  {\mathcal D} _\hsc^{\sharp, \infty}}, \label{eq:PiHproper}
\end{align}
where we used the last equality in (\ref{eq:comsupp}) to obtain the second line.
By Lemma \ref{thm:funcloc2} with $f(P_\hsc^\sharp) = \psi_\mu(P_\hsc^\sharp) = \Pilow$, $\Pilow^\Psi:=\psi_\mu(Q_\hsc) \in \Psi^{-\infty}_\hsc(\mathbb T^d_{R_\sharp})$ is such that
$$
\chi \Pilow \chi = \chi \Pilow^\Psi \chi  + O(\hsc^\infty)_{ {\mathcal D} _\hsc^{\sharp,-\infty} \rightarrow  {\mathcal D} _\hsc^{\sharp, \infty}}%\quad\tand\quad\operatorname{WF}_\hsc \Pilow^\Psi \subset \operatorname{supp} \psi_\mu\circ q_\hsc.
$$ 
Hence, taking $\Pihigh^{\Psi} := I - \Pilow^{\Psi}= (1-\psi_\mu)(Q_\hsc) \in \Psi^0_\hsc(\mathbb T^d_{R_{\sharp}})$,
\begin{equation} \label{eq:PiHP}
\chi \Pihigh \chi = \chi \Pihigh^{\Psi}\chi  + O(\hsc^\infty)_{ {\mathcal D} _\hsc^{\sharp,-\infty} \rightarrow  {\mathcal D} _\hsc^{\sharp, \infty}}%\quad\tand\quad \operatorname{WF}_\hsc \Pihigh^\Psi \subset \operatorname{supp} (1-\psi_\mu)\circ q_\hsc; 
\end{equation}
in other words, modulo negligible terms, $\chi \Pihigh \chi $ is a high-frequency cut-off defined from the semiclassical pseudodifferential calculus. 
We here emphasise  that, since $\chi$ is supported in $B_{R_{\sharp}}$ and vanishes near $B_{R_0}$, $\chi\Pihigh^{\Psi}\chi $ can be seen as
an element of  \emph{both} $\mathcal L(\mathcal H ^{\sharp})$ \emph{and} $\Psi^0_\hsc(\mathbb T^d_{R_{\sharp}})$.

\ble
With 
$\Pilow^{\Psi}:= \psi_\mu(Q_\hsc)$ and 
$\Pihigh^{\Psi} := (1-\psi_\mu)(Q_\hsc)$,
\begin{equation} \label{eq:PiLPWF}
\WFh  \Pilow^\Psi \subset  q_\hsc^{-1}\big(\supp \,\psi_\mu\big) = \{|q_\hsc| \leq 2 \mu\}
\eeq
and 
\beq\label{eq:PiHPWF}
\WFh  \Pihigh^\Psi \subset  q_\hsc^{-1}\big(\supp  (1- \psi_\mu)\big) = \{|q_\hsc| \geq \mu\}.
\end{equation}
\ele

\bpf
This follows from \eqref{microsupport} (in Lemma \ref{thm:funcloc2}), first with $f= \psi_\mu$, and then with $f=1-\psi_\mu$.
\epf

\

By (\ref{eq:PiHproper}) and (\ref{eq:PiHP}), for any $N$ and any $m$,
$$
\big\Vert \Pihigh [P^{\sharp}_{\hsc},M_\varphi] u \big\Vert_{\mathcal H^{\sharp}} 
\leq \big\Vert\chi\Pihigh^{\Psi} \chi [P^{\sharp}_{\hsc},M_\varphi] \phi u \big\Vert_{\mathcal H^{\sharp}} + C_{N,m} \hsc^N \big\Vert  [P^{\sharp}_{\hsc},M_\varphi]   \phi u \big\Vert_{\mathcal D_\hsc^{\sharp, -m}} +  C'_{N} \hsc^N \big\Vert \widetilde \phi u \big\Vert_{\mathcal H^{\sharp}},
$$
with $\widetilde \phi$ compactly supported in $B_{R_{\sharp}} \backslash B_{R_0}$ and equal to one on $\operatorname{supp}\phi$. Taking $m=1$, then $N=M+1$ and using the resolvent estimate (\ref{eq:res}) we get
 \begin{align} \nonumber
\big\Vert \Pihigh [P^{\sharp}_{\hsc},M_\varphi] u \big\Vert_{\mathcal H^{\sharp}} &\leq \big\Vert \chi \Pihigh^{\Psi}  \chi [P^{\sharp}_{\hsc},M_\varphi] \phi u \big\Vert_{\mathcal H^{\sharp}} + C''_{M+1} \hsc^{M+1} \big\Vert  \widetilde \phi u \big\Vert_{\mathcal H^{\sharp}} \\
&= \big\Vert\chi \Pihigh^{\Psi}  \chi [P^{\sharp}_{\hsc},M_\varphi] \phi u \big\Vert_{\mathcal H^{\sharp}} + C''_{M+1} \hsc^{M+1} \big\Vert  \widetilde \phi u \big\Vert_{\mathcal H} \nonumber \\
&\lesssim \big\Vert \chi \Pihigh^{\Psi} \chi [P^{\sharp}_{\hsc},M_\varphi] \phi u \big\Vert_{\mathcal H^{\sharp}} +  \big\Vert  g \big\Vert_{\mathcal H}. \label{eq:PiHex}
\end{align}
Finally, by the definition of $P^\sharp_\hsc$ \eqref{eq:defref} and the fact that $\phi$ equals zero near $B_{R_0}$,
$$
\big\Vert \chi \Pihigh^{\Psi} \chi [P^{\sharp}_{\hsc},M_\varphi] \phi u \big\Vert_{\mathcal H^{\sharp}}
=\big\Vert \chi \Pihigh^{\Psi}\chi [Q_\hsc - I,\varphi] \phi u \big\Vert_{L^2(\mathbb T_{R_{\sharp}}^d)},
$$
hence by (\ref{eq:PiHex}),
\begin{equation} \label{eq:apprPiH}
\big\Vert \Pihigh [P^{\sharp}_{\hsc},M_\varphi] u \big\Vert_{\mathcal H^{\sharp}} \lesssim \big\Vert \chi \Pihigh^{\Psi}\chi [Q_\hsc - I ,\varphi] \phi u \big\Vert_{L^2(\mathbb T_{R_{\sharp}}^d)} +  \Vert  g \Vert_{\mathcal H}.
\end{equation}

\subsubsection*{Step 3:~A semiclassical elliptic estimate in $\mathbb T^d_{R_{\sharp}} $}

Combining \eqref{eq:high1} and \eqref{eq:apprPiH}, we see that to prove \eqref{eq:decHF} we only need to bound 
$ \chi\Pihigh^{\Psi} \chi[Q_\hsc - I ,\varphi] \phi u$ in $L^2(\mathbb T_{R_{\sharp}}^d)$.
To do this, we use the semiclassical elliptic parametrix construction given by Theorem \ref{thm:para}. 

\begin{lemma}
The operator $Q_\hsc - I $ is semiclassically elliptic
on the semiclassical wavefront set of $\hsc^{-1} \chi\Pihigh^{\Psi}\chi[Q_\hsc - I,  \varphi]$.
\end{lemma}
\begin{proof}
By (\ref{eq:WFprod}), (\ref{eq:support}), (\ref{eq:PiHPWF}) and \eqref{eq:John1},
$$
\operatorname{WF}_\hsc(\hsc^{-1} \chi\Pihigh^{\Psi}\chi[Q_\hsc - I, \varphi]) \subset \operatorname{WF}_\hsc {\Pihigh^{\Psi}} \subset q_\hsc^{-1}\big(\operatorname{supp} (1-\psi_\mu)\big) \subset \{ q_\hsc \geq \mu \}.
$$
But, on $\{ q_\hsc \geq \mu\}$, by definition of $\mu$ (\ref{eq:newdefmu}),
$$
\langle \xi \rangle^{-2} (q_\hsc(x,\xi) - 1)\geq c_{\rm ell} >0,
$$
and the proof is complete.
\end{proof}

\

Since $\hsc^{-1} \chi\Pihigh^{\Psi}\chi[Q_\hsc - I, \varphi] \in \Psi^1_\hsc(\mathbb T^d_{R_{\sharp}})$ by Theorem \ref{thm:basicP},
we can therefore apply the elliptic parametrix construction given by Theorem \ref{thm:para} with $A = \hsc^{-1}\chi \Pihigh^{\Psi}\chi[Q_\hsc - I, \varphi]$, $B = Q_\hsc - I$, and $\ell = 1$, $m= 2$. Hence, there exists $S \in \Psi^{-1}_\hsc(\mathbb T^d_{R_{\sharp}})$ and $R = O(\hsc^\infty)_{\Psi^{-\infty}_\hsc}$ with
\begin{equation} \label{eq:wfQ}
\operatorname{WF}_\hsc S \subset \operatorname{WF}_\hsc \big( \hsc^{-1} \Pihigh^{\Psi}[Q_\hsc - I, \varphi] \big),
\end{equation}
and such that
$$
\chi\Pihigh^{\Psi}\chi[Q_\hsc - I, \varphi] =  \hsc S ( Q_\hsc - I) + R.
$$
We apply both sides of this identity to $\phi u$ and then use \eqref{eq:bbreq1} and the fact that $\phi$ is equal to zero 
near $B_{R_0}$ and supported in $B_{R_{\sharp}}$; the result is that
\begin{align} \label{eq:parau}
\chi\Pihigh^{\Psi}\chi[Q_\hsc - I, \varphi] \phi u &=  \hsc S ( Q_\hsc - I) \phi u + R \phi u \nonumber \\
%&=  \hsc S ( P_\hsc- I) \phi u + R \phi u \nonumber \\
&=  \hsc S \phi ( Q_\hsc - I) u +\hsc S [ Q_\hsc - I, \phi] u + R \phi u  \nonumber\\
&=  \hsc S \phi ( P_\hsc- I) u +\hsc S [ Q_\hsc - I, \phi] u + R \phi u.
\end{align}
The following lemma combined with (\ref{eq:WFdis}) shows that 
\begin{equation} \label{eq:commneg}
S [ Q_\hsc - I, \phi] = O(\hsc^\infty)_{\Psi^{-\infty}_\hsc}.
\end{equation}

\begin{lemma}
$$
\operatorname{WF}_\hsc S \cap \operatorname{WF}_\hsc  [ Q_\hsc - I, \phi] = \emptyset.
$$
\end{lemma}
\begin{proof}
By (\ref{eq:wfQ}) and the definition of $Q_\hsc$ \eqref{eq:Qdef},
$$
\operatorname{WF}_\hsc S \subset \operatorname{WF}_\hsc [Q_\hsc - I, \varphi] \subset  (\operatorname{supp} \nabla \varphi)
\times \mathbb R^d$$
Similarly, 
$$
\operatorname{WF}_\hsc  [ Q_\hsc - I, \phi] \subset
(\operatorname{supp} \nabla \phi) \times \mathbb R^d,
$$
Now, by (\ref{eq:suppphi}), $\operatorname{supp} \nabla \varphi$ and $\operatorname{supp} \nabla \phi$ are disjoint, and the result follows.
\end{proof}

\

Therefore, by (\ref{eq:parau}), (\ref{eq:commneg}) and the definition of $O(\hsc^\infty)_{\Psi^{-\infty}_\hsc}$ (\ref{eq:residual}), for any $N$, there exists $C_N, C'_N >0$ such that
\begin{align*}
\Vert \chi\Pihigh^{\Psi}\chi[Q_\hsc - I, \varphi] \phi u \Vert_{L^2(\mathbb T^d_{R_{\sharp}})}
&\leq \hsc \Vert S \phi ( P_\hsc- I) u \Vert_{L^2(\mathbb T^d_{R_{\sharp}})} + C_N \hsc^N \Vert \widetilde \phi u \Vert_{L^2(\mathbb T^d_{R_{\sharp}})} + C'_N \hsc^N \Vert \phi u \Vert_{L^2(\mathbb T^d_{R_{\sharp}})} \\
&\quad= \hsc \Vert S \phi ( P_\hsc- I) u \Vert_{L^2(\mathbb T^d_{R_{\sharp}})} + C_N \hsc^N \Vert \widetilde \phi u \Vert_{\mathcal H} + C'_N \hsc^N \Vert \phi u \Vert_{\mathcal H},
\end{align*}
where $\widetilde \phi$ is compactly supported in $B_{R_{\sharp}} \backslash B_{R_0}$ and equal to one on $\operatorname{supp}\phi$.
Taking $N:=M+1$ and using the resolvent estimate (\ref{eq:res}), we then obtain that
\begin{align}
\Vert \chi\Pihigh^{\Psi}\chi[Q_\hsc - I, \varphi] \phi u \Vert_{L^2(\mathbb T^d_{R_{\sharp}})}
& \lesssim \hsc \Vert S \phi ( P_\hsc- I) u \Vert_{L^2(\mathbb T^d_{R_{\sharp}})} + \hsc \Vert g \Vert_{\mathcal H} \nonumber \\
& \lesssim \hsc \Vert \phi ( P_\hsc- I) u \Vert_{L^2(\mathbb T^d_{R_{\sharp}})} + \hsc \Vert g \Vert_{\mathcal H}, \label{eq:parauest}
\end{align}
where we used in the second line the fact that $S \in \Psi^{-1}(\mathbb T^d_{R_{\sharp}}) \subset  \Psi^{0}(\mathbb T^d_{R_{\sharp}})$ together with Part (iii) of Theorem \ref{thm:basicP}. Now,
since $\phi$ is equal to zero 
near $B_{R_0}$ and supported in $B_{R_{\sharp}}$, we get
$$
\Vert \phi ( P_\hsc- I) u \Vert_{L^2(\mathbb T^d_{R_{\sharp}})} = \Vert \phi (P_\hsc- I)u \Vert_{\mathcal H} = \Vert \phi g \Vert_{\mathcal H} \leq \Vert g \Vert_{\mathcal H}.
$$
Thus,  (\ref{eq:parauest}) implies that
$$
\Vert \chi\Pihigh^{\Psi}\chi[Q_\hsc - I ,\varphi] \phi u \Vert_{L^2(\mathbb T_{R_{\sharp}}^d)}\lesssim \hsc \Vert g \Vert_{\mathcal H}. 
$$
Combining this last estimate with (\ref{eq:high1})  and (\ref{eq:apprPiH}) we conclude that
$$
\norm{\Pihigh w}_{\Hilb^{\sharp}}+ \norm{P^{\sharp}_{\hsc} \Pihigh w}_{\Hilb^{\sharp}} \lesssim \Vert g \Vert_{\mathcal H};
$$
hence (\ref{eq:HF_eq}) holds.

\subsection{Decomposition (\ref{eq:LF_dec}) of $\uL$, and proof of the bounds (\ref{eq:decLF1})--(\ref{eq:decLF5}) and (\ref{eq:eps_eq})  (the low-frequency component)} \label{subsec:low}

By Assumption 2 in Theorem \ref{thm:mainbb}, there exists $E_\infty = \residual$ with
\begin{equation} \label{eq:Edec}
\mathcal E(P^\sharp_\hsc) = E  + E_\infty,
\end{equation}
and the low-frequency estimate (\ref{eq:lowenest}) holds. 
By (\ref{eq:propLambda}) (a consequence of the definition of the constant $\Lambda$ \eqref{eq:Lambda}), $\mathcal E$ is nowhere zero on the support of $\psi_\mu$; therefore  the function $\psi_\mu/\mathcal E$ is well-defined and in $C_0(\mathbb R)$. The definition
of $\Pilow$ (\ref{eq:PiL}) and Part \ref{it:fc3} of Theorem \ref{thm:fundfc} imply that
\begin{equation} \label{eq:divbyE}
 \Pilow =\psi_\mu(P^\sharp_\hsc)  = \mathcal E(P^\sharp_\hsc)\left (\frac{1}{\mathcal E} \psi_\mu\right)(P^\sharp_\hsc) = E \circ\bigg( \left [\frac{1}{\mathcal E} \psi_\mu\bigg](P^\sharp_\hsc)\right)+ E_\infty\circ\left(\left [\frac{1}{\mathcal E} \psi_\mu\right](P^\sharp_\hsc)\right).
\end{equation}
Then, by Part \ref{it:fc5} of Theorem \ref{thm:fundfc} and the fact that $E_\infty = \residual$,
\begin{equation} \label{eq:bigres}
E_\infty\circ\left(\left [\frac{1}{\mathcal E} \psi_\mu\right](P^\sharp_\hsc)\right)
  = \residual.
\end{equation}
\subsubsection{The decomposition (\ref{eq:LF_dec}) of $\uL$ when $\rho = 1$}
We first assume that $\rho = 1$ and we show the decomposition (\ref{eq:LF_dec}), together with the bound (\ref{eq:decLF5}) 
on $\ulow$ and the bound (\ref{eq:eps_eq}) on $u_\epsilon$.
In this case, we let
\beqs
u_{\mathcal A} :=  E  \circ\left(\left [\frac{1}{\mathcal E} \psi_\mu\right](P^\sharp_\hsc)\right)  w\quad\tand\quad u_\epsilon := E_\infty \circ\left(\left [\frac{1}{\mathcal E} \psi_\mu\right](P^\sharp_\hsc) \right) w,
\eeqs
so that  (\ref{eq:LF_dec}) holds by \eqref{eq:Edec} and \eqref{eq:PiL}.
Moreover, since both $u_{\mathcal A}$ and $u_\epsilon$ involve compactly-supported functions of $P^\sharp_\hsc$,
by the reasoning immediately above \eqref{eq:newlabel1}, both $u_{\cA}$ and $u_\epsilon$ are in $\mathcal{D}^{\sharp,\infty}_\hsc$.
Then, using (in this order) the low-frequency estimate (\ref{eq:lowenest}), Part \ref{it:fc5} of Theorem \ref{thm:fundfc}, and finally the resolvent estimate (\ref{eq:res}), we get
\begin{align*}
 \Vert D(\alpha) u_{\mathcal A}  \Vert_{\mathcal H^{\sharp}} &= \N{ D(\alpha) 
 E  \circ\left(\left [\frac{1}{\mathcal E} \psi_\mu\right](P^\sharp_\hsc)\right) 
 w}_{\mathcal H^{\sharp}}
\leq C_{\mathcal E}(\alpha, \hsc) \N{ \left [\frac{1}{\mathcal E} \psi_\mu\right](P^\sharp_\hsc) w }_{\mathcal H^{\sharp} } \\
&\leq C_{\mathcal E}(\alpha, \hsc) \,\sup_{\lambda \in \mathbb R} \left| \frac{1}{\mathcal E(\lambda)} \psi_{\mu}(\lambda) \right| \Vert w \Vert_{\mathcal H^{\sharp} }
= C_{\mathcal E}(\alpha, \hsc)\, \sup_{\lambda \in \mathbb R} \left| \frac{1}{\mathcal E(\lambda)} \psi_{\mu}(\lambda) \right| \Vert w \Vert_{\mathcal H } \\
&\hspace{5.5cm}\lesssim  C_{\mathcal E}(\alpha, \hsc) \,\sup_{\lambda \in \mathbb R} \left| \frac{1}{\mathcal E(\lambda)} \psi_{\mu}(\lambda) \right| \hsc^{-M-1}\Vert g \Vert_{\mathcal H };
\end{align*}
thus (\ref{eq:decLF5}) holds. In addition, the bound (\ref{eq:eps_eq}) on $u_\epsilon$ follows from (\ref{eq:bigres}) together with the resolvent estimate  (\ref{eq:res}).

\subsubsection{The decomposition (\ref{eq:LF_dec}) of $\uL$ when $\rho\neq 1$}\label{sec:332}

We now tackle the general case (i.e., $\rho\neq 1$).
Given $R_0$ and $\widetilde R$, let $\RfarA, \RfarB, \RlocB, \RlocA,$ be such that 
$R_0<\RfarA< \RfarB< \RlocB< \RlocA<\widetilde R$ and
  $\rho = 1$ near $B_{\RlocA}$. In addition, let 
$\rho_{1}\in C^\infty(\mathbb T^d_{R_\sharp})$ be equal to one near $B_{R_0}$ 
and such that $\operatorname{supp} (1 - \rho_{1}) \subset (B_{\RfarB})^c $
 and $\operatorname{supp} \rho_{1} \Subset B_{\RlocB} $ (see Figure \ref{fig:line_func}).

\begin{figure}
\begin{center}
    \includegraphics[scale=1]{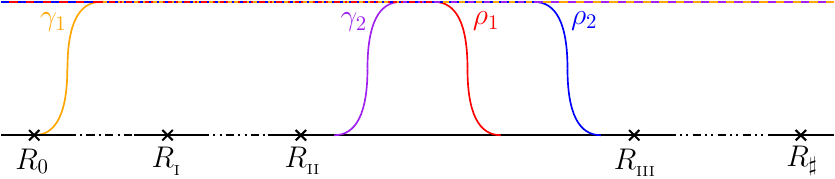}
  \end{center}
    \caption{The cut-off functions $\rho_1, \rho_2, \gamma_1, \gamma_2$. $\rho_1$ is used in \S\ref{sec:332}, $\rho_2$ in \S\ref{sec:333}, and $\gamma_1$ and $\gamma_2$ in \S\ref{sec:334}.}
  \label{fig:line_func}
\end{figure}

Using the decomposition (\ref{eq:divbyE}) of $\Pilow$, we decompose $\uL = \Pilow w$ as
\begin{align}\nonumber
\uL &= \Pilow \rho_1 w + \Pilow (1-\rho_1) w \\ 
 &= %E  \left (\frac{1}{\mathcal E} \psi_\mu\right)(P^\sharp_\hsc)
 E  \circ\left(\left [\frac{1}{\mathcal E} \psi_\mu\right](P^\sharp_\hsc)\right) 
 \rho_1w + 
E_\infty \circ\left(\left [\frac{1}{\mathcal E} \psi_\mu\right](P^\sharp_\hsc)\right) 
 %E_\infty\left (\frac{1}{\mathcal E} \psi_\mu\right)(P^\sharp_\hsc) 
 \rho_1w + \Pilow (1-\rho_1) w,\label{eq:boundary_condition1}
\end{align}
and we define
\beq\label{eq:uR0A}
u^{R_0}_{\mathcal A} :=  %E  \left (\frac{1}{\mathcal E} \psi_\mu\right)(P^\sharp_\hsc)
 E  \circ\left(\left [\frac{1}{\mathcal E} \psi_\mu\right](P^\sharp_\hsc)\right) 
\rho_1w\quad\tand\quad \uL^\infty :=  \Pilow (1-\rho_1) w.
\eeq
Since $u^{R_0}_{\mathcal A}$ involves a compactly-supported function of $P^\sharp_\hsc$, $u^{R_0}_{\mathcal A} \in \mathcal{D}^{\sharp,\infty}_\hsc$.
We decompose $u^{\infty}_{\rm Low}$ in \S\ref{sec:334} below as 
\begin{equation}\label{eq:A:1}
u^{\infty}_{\rm Low} = u^{\infty}_{\mathcal A} + \widetilde u_\epsilon
\end{equation}
with $u^{\infty}_{\mathcal A} \in \mathcal{D}^{\sharp,\infty}_\hsc$, (see \eqref{eq:uainf} below)
and then define
\begin{equation}\label{eq:A:defAe}
u_{\mathcal A} :=  u_{\cA}^{R_0} +  u^{\infty}_{\mathcal A} \in
\mathcal{D}^{\sharp,\infty}_\hsc
 %\mathcal{D}^\sharp
 \quad\tand\quad u_{\epsilon} := \widetilde u_\epsilon + 
E_\infty \circ\left(\left [\frac{1}{\mathcal E} \psi_\mu\right](P^\sharp_\hsc)\right) 
 \rho_1w
\end{equation}
(with the first definition implying (\ref{eq:decLF0})).
These definitions imply that $\uL = \ulow + u_\epsilon$, i.e., that (\ref{eq:LF_dec}) holds. To complete the proof, we now need to show that the bounds (\ref{eq:decLF1}) and \eqref{eq:decLF2} on $u^{R_0}_{\mathcal A}$, the bounds \eqref{eq:decLF3} and (\ref{eq:decLF4}) 
on $u^\infty_{\cA}$, and the bound (\ref{eq:eps_eq}) on $u_\epsilon$ all hold. This decomposition of $\uL$ and the ideas behind it
are summed-up in Figure \ref{fig:split_uL}.

\begin{figure}
\hspace{-18em}
%\begin{adjustwidth*}{}{-18em}
 \begin{tikzpicture}%
  [>=stealth,
   shorten >=1pt,
   align = center,
   node distance=4cm and 4.5cm,
   on grid
  ]
 \node (1)  {$\uL := \Pilow w$};
\node (21) [below=of 1] {\begin{minipage}{\textwidth}
            \begin{gather*} 
            \Pilow \rho_1 w \\ 
            = \mathcal E(P_\hbar^\sharp) \circ\left(\left [\frac{1}{\mathcal E} \psi_\mu\right](P^\sharp_\hsc)\right)  \rho_1 w
            \end{gather*}
        \end{minipage}};
\node (22) [right=of 21] {\begin{minipage}{\textwidth}
            \begin{gather*} 
            \uL^\infty := \Pilow (1-\rho_1) w
            \end{gather*}
        \end{minipage}};
\node (32) [below=of 21] {
\begin{minipage}{\textwidth}
            \begin{gather*} 
          E_\infty \circ\left(\left [\frac{1}{\mathcal E} \psi_\mu\right](P^\sharp_\hsc)\right) \rho_1 w \\ 
            \text{\small small part} 
            \end{gather*}
        \end{minipage}};
\node (31) [left=of 32] {\begin{minipage}{\textwidth}
            \begin{gather*} 
            u_{\mathcal A}^{R_0} :=E \circ\left(\left [\frac{1}{\mathcal E} \psi_\mu\right](P^\sharp_\hsc)\right)  \rho_1 w \\ 
            \text{\small regular near $B_{R_0}$ thanks} \\
            \text{\small to the low-frequency estimate,} \\
            \text{\small small away from $B_{R_0}$} \\
            \end{gather*}
        \end{minipage}};
\node (33) [below=of 22] {\begin{minipage}{\textwidth}
            \begin{gather*} 
            \tilde u_{\epsilon} \\ 
            \text{\small small part} 
            \end{gather*}
        \end{minipage}};
\node (34) [right=of 33] {\begin{minipage}{\textwidth}
            \begin{gather*} 
            u^\infty_{\mathcal A} \\ 
            \text{\small part given by} \\
            \text{\small a Fourier multiplier} \\
            \text{\small on the torus $\mathbb T^d$,} \\
            \text{\small entire away from $B_{R_0}$,} \\
             \text{\small small near $B_{R_0}$}           
            \end{gather*}
        \end{minipage}};
\node (4) [below=of 33] {$u_\epsilon$};
\path[->]
(1) edge node [left] {\small part near $B_{R_0}$ \;} (21)
    edge node [right]{\hspace{0.4cm} \small part away from $B_{R_0}$} (22)
(21) edge node {} (31)
    edge node {} (32)
(22) edge node {} (33)
    edge node {} (34)
(32) edge node {} (4)
(33) edge node {} (4)
;
\end{tikzpicture}
%\end{adjustwidth*}
\caption{The splitting of $\uL$}  \label{fig:split_uL}
\end{figure}

\subsubsection{Proof of (\ref{eq:decLF1}) and (\ref{eq:decLF2}) for the localised term $u^{R_0}_{\mathcal A}$.} \label{sec:333}

Using (in this order) the definition of $u^{R_0}_{\mathcal A}$ \eqref{eq:uR0A}, the fact that $\rho=1$ on $B_{\RlocA}$,
the low-frequency estimate (\ref{eq:lowenest}), Part \ref{it:fc5} of Theorem \ref{thm:fundfc}, and finally the resolvent estimate (\ref{eq:res}) we obtain 
\begin{align*}
 \Vert D(\alpha) u^{R_0}_{\mathcal A}  \Vert_{\mathcal H^{\sharp}(B_{\RlocA})} &= \N{ D(\alpha)
  E  \circ\left(\left [\frac{1}{\mathcal E} \psi_\mu\right](P^\sharp_\hsc)\right) 
 %E  \left (\frac{1}{\mathcal E} \psi_\mu\right)(P^\sharp_\hsc)
 \rho_1w}_{\mathcal H^{\sharp}(B_{\RlocA})}
\leq \N{ \rho D(\alpha)
  E  \circ\left(\left [\frac{1}{\mathcal E} \psi_\mu\right](P^\sharp_\hsc)\right) 
%E  \left (\frac{1}{\mathcal E} \psi_\mu\right)(P^\sharp_\hsc)
\rho_1w}_{\mathcal H^{\sharp} } \\
&\leq C_{\mathcal E}(\alpha, \hsc) \N{
\left(\left [\frac{1}{\mathcal E} \psi_\mu\right](P^\sharp_\hsc)\right)
%\left (\frac{1}{\mathcal E} \psi_\mu\right)(P^\sharp_\hsc) 
\rho_1w}_{\mathcal H^{\sharp} }
\leq C_{\mathcal E}(\alpha, \hsc) \sup_{\lambda \in \mathbb R} \left| \frac{1}{\mathcal E(\lambda)} \psi_{\mu}(\lambda) \right| \Vert w \Vert_{\mathcal H^{\sharp} }\\
&= C_{\mathcal E}(\alpha, \hsc) \sup_{\lambda \in \mathbb R} \left| \frac{1}{\mathcal E(\lambda)} \psi_{\mu}(\lambda) \right| \Vert w \Vert_{\mathcal H }
\lesssim  C_{\mathcal E}(\alpha, \hsc) \sup_{\lambda \in \mathbb R} \left| \frac{1}{\mathcal E(\lambda)} \psi_{\mu}(\lambda) \right| \hsc^{-M-1}\Vert g \Vert_{\mathcal H };
\end{align*}
thus (\ref{eq:decLF1}) holds, where the $\sup_{\lambda \in \mathbb R}$ becomes $\sup_{\lambda \in [-\Lambda, \Lambda]}$ because of the support property \eqref{eq:propLambda} of $\psi_\mu$.

Let $\rho_2 \in C^\infty(\mathbb T^d_{R_\sharp})$ be supported in $B_{\RlocB}$ and such that $\rho_2 = 1$ on $\operatorname{supp} \rho_1$ (see Figure \ref{fig:line_func}).
By (\ref{eq:Edec}), Part \ref{it:fc3} of Theorem \ref{thm:fundfc}, and the pseudo-locality of the functional calculus (Lemma \ref{thm:funcloc1}),
\begin{align}\nonumber
(1-\rho_2) 
 E  \circ\left(\left [\frac{1}{\mathcal E} \psi_\mu\right](P^\sharp_\hsc)\right) 
%E \left (\frac{1}{\mathcal E} \psi_\mu\right)(P^\sharp_\hsc) 
\rho_1 &= (1-\rho_2) \mathcal E(P^\sharp_\hsc)\left (\frac{1}{\mathcal E} \psi_\mu\right)(P^\sharp_\hsc) \rho_1 + O(\hsc^\infty)_{ {\mathcal D} _\hsc^{\sharp,-\infty} \rightarrow  {\mathcal D} _\hsc^{\sharp,\infty}} \\
&= (1-\rho_2) \Pilow \rho_1 + \residual  = O(\hsc^\infty)_{ {\mathcal D} _\hsc^{\sharp,-\infty} \rightarrow  {\mathcal D} _\hsc^{\sharp,\infty}}. 
 \label{eq:A:R0:1}
\end{align}
On the other hand, since $\rho_2 = 0$ on $B^c_{{\RlocB}}$,
\begin{align*}
\Vert  u^{R_0}_{\mathcal A}  \Vert_{\mathcal D^{m, \sharp}((B_{\RlocB})^c)}
&=  \N{(1-\rho_2) 
 E  \circ\left(\left [\frac{1}{\mathcal E} \psi_\mu\right](P^\sharp_\hsc)\right) 
%E \left (\frac{1}{\mathcal E} \psi_\mu\right)(P^\sharp_\hsc) 
\rho_1 w }_{\mathcal D^{m, \sharp}((B_{\RlocB})^c)} \\
&\leq \N{(1-\rho_2)
 E  \circ\left(\left [\frac{1}{\mathcal E} \psi_\mu\right](P^\sharp_\hsc)\right) 
%E \left (\frac{1}{\mathcal E} \psi_\mu\right)(P^\sharp_\hsc)
\rho_1 w }_{\mathcal D^{m, \sharp}}.
\end{align*}
Combining this with (\ref{eq:A:R0:1}) and then using the resolvent estimate (\ref{eq:res}), we obtain (\ref{eq:decLF2}).

\subsubsection{The term away from the black-box $u^{\infty}_{\rm Low}$.} \label{sec:334}
\paragraph{Step 1:~obtaining the decomposition  (\ref{eq:A:1}) and the bound (\ref{eq:eps_eq})  on $u_\epsilon$.}
Let $\gamma_1 \in C^\infty(\mathbb T^d_{R_\sharp})$ be equal to zero near $B_{R_0}$, and such that $\gamma_1 = 1$ near $(B_{\RfarA})^c$.
Since $\supp (1-\gamma_1)$ and $\supp(1-\rho_1)$ are disjoint (see Figure \ref{fig:line_func}), 
by the pseudo-locality of the functional calculus given by Lemma \ref{thm:funcloc1},
\begin{align*}
\Pilow (1-\rho_1) &= \gamma_1 \Pilow (1-\rho_1) +  O(\hsc^\infty)_{ {\mathcal D} _\hsc^{\sharp,-\infty} \rightarrow  {\mathcal D} _\hsc^{\sharp,\infty}}\\
&= \gamma_1 \Pilow \gamma_1 (1-\rho_1)+  O(\hsc^\infty)_{ {\mathcal D} _\hsc^{\sharp,-\infty} \rightarrow  {\mathcal D} _\hsc^{\sharp,\infty}}.
\end{align*}
Therefore, by Lemma \ref{thm:funcloc2},
\begin{equation} \label{eq:A:far:1}
\Pilow (1-\rho_1) = \gamma_1 \Pi^\Psi_L \gamma_1 (1-\rho_1) +  O(\hsc^\infty)_{ {\mathcal D} _\hsc^{\sharp,-\infty} \rightarrow  {\mathcal D} _\hsc^{\sharp,\infty}},
\end{equation}
where $\Pilow^\Psi \in\Psi^{-\infty}_\hsc(\mathbb T^d_{R_\sharp})$ and
\begin{equation} \label{eq:low:suppg}
\operatorname{WF}_\hsc \Pilow^\Psi \subset \operatorname{supp}\psi_\mu \circ q_h.
\end{equation}
By \eqref{eq:Qnew}, since $\psi_\mu$ is compactly supported, there exists $\lambda>1$ such that 
\begin{equation} \label{eq:low:a:A}
\operatorname{supp}\psi_\mu \circ q_\hsc \subset \mathbb T^d_{R_\sharp} \times B\left(0,\frac \lambda2\right).
\end{equation}
Now, let 
$\widetilde{\varphi} \in C^\infty_{\rm comp}$ be compactly supported in $B(0, \lambda^2)$ and equal to one on $B(0,\lambda^2/4)$. By (\ref{eq:low:a:A}) and (\ref{eq:low:suppg}) together with (\ref{eq:support}), $\operatorname{WF}_\hsc \big( 1 - \Optorus(\widetilde{\varphi}(|\xi|^2)) \big) \cap\operatorname{WF}_\hsc \big(\Pilow^\Psi \big)= \emptyset$. Therefore, by (\ref{eq:WFdis}), as operators on the torus,
\begin{equation}\label{eq:E1} 
\Pilow^\Psi = \Optorus(\widetilde{\varphi}(|\xi|^2)) \Pilow^\Psi + E_1,
\end{equation}
where $E_1=O(\hsc^\infty)_{\Psi^{-\infty}_\hsc}$.
Since $\gamma_1=0$ near $B_{R_0}$, 
by the definitions of $P^\sharp$ \eqref{eq:defref}, $\|\cdot\|_{\mathcal{D}_\hsc^{\sharp,m}}$ \eqref{eq:BB:norms}, and $\|\cdot\|_{H_\hsc^{2m}(\mathbb{T}^d_{R^\sharp})}$ \eqref{eq:Hhnorm},
\beq\label{eq:normequiv}
\N{\gamma_1 w}_{\mathcal{D}_\hsc^{\sharp,m}} \lesssim_m \N{\gamma_1 w}_{H^{2m}_\hsc(\mathbb{T}^d_{R^\sharp})} \lesssim_m 
\N{\gamma_1 w}_{\mathcal{D}_\hsc^{\sharp,m}} \quad \tfa w \in \mathcal{D}_\hsc^{\sharp,m},
\eeq
and thus $\gamma_1 E_1 \gamma_1 = O(\hsc^\infty)_{ {\mathcal D} _\hsc^{\sharp,-\infty} \rightarrow  {\mathcal D} _\hsc^{\sharp,\infty}}$. 
Therefore, combining this with \eqref{eq:E1} and (\ref{eq:A:far:1}), we obtain that
\beq\label{eq:PiLrho1}
\Pilow (1-\rho_1) = \gamma_1 \Optorus(\widetilde{\varphi}(|\xi|^2)) \Pi^\Psi_L \gamma_1 (1-\rho_1) + E_2,
\eeq
where $E_2=O(\hsc^\infty)_{ {\mathcal D} _\hsc^{\sharp,-\infty} \rightarrow  {\mathcal D} _\hsc^{\sharp,\infty}}$.
We let
\beq\label{eq:uainf}
u_{\mathcal A}^\infty := \gamma_1 \Optorus(\widetilde{\varphi}(|\xi|^2)) \Pi^\Psi_L \gamma_1 (1-\rho_1)w
\quad\tand\quad\widetilde u_\epsilon := E_2w;
\eeq
observe that $u_{\mathcal A}^\infty\in \cD^\sharp$ because of the presence of $\gamma_1$ at the start of the expression.
The decomposition (\ref{eq:A:1}) then holds by \eqref{eq:PiLrho1} and \eqref{eq:uR0A}. The bound (\ref{eq:eps_eq}) on $ u_\epsilon$ follows directly from the definition of $u_\epsilon$ (\ref{eq:A:defAe}), together with (\ref{eq:bigres}), the fact that $E_2=O(\hsc^\infty)_{ {\mathcal D} _\hsc^{\sharp,-\infty} \rightarrow  {\mathcal D} _\hsc^{\sharp,\infty}}$, and the resolvent estimate (\ref{eq:res}).

\paragraph{Step 2:~proving that $u_{\mathcal A}^\infty$ is regular in $(B_{\RfarA})^c$ (i.e., the bound (\ref{eq:decLF3})).}
By the definition  of $u_\cA^\infty$ \eqref{eq:uainf} and the fact that $\gamma_1 = 1$ on $(B_{\RfarA})^c$,
\begin{align} \label{eq:A:far:2}
\Vert \partial^\alpha u_{\mathcal A}^\infty  \Vert_{\mathcal H((B_{\RfarA})^c)} &= \Big\Vert \partial^\alpha\Optorus(\widetilde{\varphi}(|\xi|^2)) \Pi^\Psi_L \gamma_1 (1-\rho_1)w \Big\Vert_{\mathcal H((B_{\RfarA})^c)} \nonumber \\
&\leq  \Big\Vert \partial^\alpha\Optorus(\widetilde{\varphi}(|\xi|^2)) \Pi^\Psi_L \gamma_1 (1-\rho_1)w \Big\Vert_{L^2(\mathbb T^d_{R_\sharp})}.
\end{align}
We now bound the right-hand side of \eqref{eq:A:far:2}. By Lemma \ref{lem:toruscalculus}, $\Optorus(\widetilde{\varphi}(|\xi|^2))$ is given 
as a Fourier multiplier on the torus (defined by \eqref{eq:A:mult}), i.e.,
\begin{equation} \label{eq:A:far:3}
\Optorus(\widetilde{\varphi}(|\xi|^2)) = \widetilde{\varphi}(-\hsc^2 \Delta).
\end{equation}
Let $v\in L^2(\mathbb T^d_{R_\sharp})$ be arbitrary, and let $\widehat v(j)$ be the Fourier coefficients of $v$. By (\ref{eq:A:mult}),
$$
\widetilde{\varphi}(-\hsc^2 \Delta) v =\sum_{j \in \ZZ^d} \widehat v(j) \widetilde{\varphi}(\hsc^2 \lvert j\rvert ^2 \pi^2/R_\sharp^2) e_j,
$$
where the normalised eigenvectors $e_j$ are defined by (\ref{eq:A:eigen}). Hence, for any multi-index $\alpha$,
\begin{align*}
\partial^\alpha \widetilde{\varphi}(-\hsc^2 \Delta) v &=\sum_{j \in \ZZ^d} \widehat v(j) \widetilde{\varphi}(\hsc^2 \lvert j\rvert ^2 \pi^2/R_\sharp^2) \left( \frac{\ri \pi j}{R_{\sharp}}\right)^\alpha e_j \\
&= \sum_{j \in \ZZ^d, \; |j| \leq \frac{\lambda R_{\sharp}}{{\hsc\pi}}} \widehat v(j) \widetilde{\varphi}(\hsc^2 \lvert j\rvert ^2 \pi^2/R_\sharp^2) \left( \frac{\ri \pi j}{R_{\sharp}}\right)^\alpha e_j,
\end{align*}
since $\widetilde{\varphi}$ is supported in $B(0, \lambda^2)$.
Therefore
\begin{align} \label{eq:low:A:3}
\Vert \partial^\alpha \widetilde{\varphi}(-\hsc^2 \Delta) v \Vert_{L^2(\mathbb T^d_{R_{\sharp}})}^2 &=  \sum_{j \in \ZZ^d, \; |j| \leq \frac{\lambda R_{\sharp}}{\hsc\pi}} \left| \widehat v(j) \widetilde{\varphi}(\hsc^2 \lvert j\rvert ^2 \pi^2/R_\sharp^2) \left( \frac{\ri \pi j}{R_{\sharp}}\right)^\alpha \right|^2  \nonumber \\
&\leq \lambda^{2|\alpha|} \hsc^{-2|\alpha|}  \sum_{j \in \ZZ^d} | \widehat v(j) | ^2 \nonumber \\
&= \lambda^{2|\alpha|} \hsc^{-2|\alpha|} \Vert v \Vert^2_{L^2(\mathbb T^d_{R_{\sharp}})}.
\end{align}
We now use (\ref{eq:low:A:3}) with 
$$
v:= \Pi^\Psi_L \gamma_1 (1-\rho_1)w,
$$
and combine the resulting estimate with (\ref{eq:A:far:2}) and (\ref{eq:A:far:3}). Using the fact that $\Pi^\Psi_L \in \Psi^{\infty}(\mathbb T^d_{R_\sharp})$, $\gamma_1 = 0$ near $B_{R_0}$, and the resolvent estimate \eqref{eq:res}, we get
\begin{align*}
\Vert \partial^\alpha u_{\mathcal A}^\infty  \Vert_{\mathcal H((B_{\RfarA})^c)} 
&\leq\lambda^{|\alpha|} \hsc^{-|\alpha|} \Vert \Pi^\Psi_L \gamma_1 (1-\rho_1)w \Vert_{{L^2}(\mathbb T^d_{R_{\sharp}})} \lesssim \lambda^{|\alpha|} \hsc^{-|\alpha|} \Vert \gamma_1 (1-\rho_1)w \Vert_{{L^2}(\mathbb T^d_{R_{\sharp}})} \\
&= \lambda^{|\alpha|} \hsc^{-|\alpha|} \Vert \gamma_1 (1-\rho_1)w \Vert_{\mathcal H} \leq \lambda^{|\alpha|} \hsc^{-|\alpha|} \hsc^{-M-1} \Vert g \Vert_{\mathcal H};
\end{align*}
hence (\ref{eq:decLF3}) holds.

\paragraph{Step 3:~proving that $u_{\mathcal A}^\infty$ is negligible in $B_{\RfarB}$ (i.e., the bound (\ref{eq:decLF4})).}

It therefore remains  to show (\ref{eq:decLF4}). Let $\gamma_2 \in C^\infty(\mathbb T^d_{R_\sharp})$ be equal to zero on $B_{\RfarB}$ and such that $\gamma_2 = 1$ on $\operatorname{supp} (1-\rho_1)$; see Figure \ref{fig:line_func}. Since
$\supp(1-\gamma_2)$ and $\supp(1-\rho_1)$ are disjoint, using (\ref{eq:WFprod}) and (\ref{eq:support})
\beqs
\operatorname{WF}_\hsc \Big( (1-\gamma_2) \Optorus(\widetilde{\varphi}(|\xi|^2)) \Pi^\Psi_L \Big) \cap {\operatorname{WF}_\hsc (1-\rho_1)} = \emptyset.
\eeqs
Then, by (\ref{eq:WFdis}),
$$
(1-\gamma_2) \Optorus(\widetilde{\varphi}(|\xi|^2)) \Pi^\Psi_L (1-\rho_1) = O(\hsc^\infty)_{\Psi^{-\infty}_\hsc}
$$
as a pseudo-differential operator on the torus.
Multiplying by $\gamma_1$ on the right and on the left, and then using
the fact that $\gamma_1=0$ on $B_{R_0}$ and the norm equivalence
\eqref{eq:normequiv}, we find
\begin{equation} \label{eq:A:far:fin}
(1-\gamma_2) \gamma_1 \Optorus(\widetilde{\varphi}(|\xi|^2)) \Pi^\Psi_L \gamma_1 (1-\rho_1) =   O(\hsc^\infty)_{ {\mathcal D} _\hsc^{\sharp,-\infty} \rightarrow  {\mathcal D} _\hsc^{\sharp,\infty}}
\end{equation}
as an element of $\mathcal L(\mathcal H^\sharp)$. 
On the other hand, since $\gamma_2 = 0$ near $B_{\RfarB}$,
$$
\Vert u_{\mathcal A}^\infty \Vert_{\mathcal D^{\sharp, m}_\hsc (B_{\RfarB})} = \Vert (1-\gamma_2)u_{\mathcal A}^\infty \Vert_{\mathcal D^{\sharp, m}_\hsc (B_{\RfarB})}.
$$
Then (\ref{eq:decLF4}) follows from combining this last equation with the definition of $\ulow^\infty$ \eqref{eq:uainf},  (\ref{eq:A:far:fin}),  and the resolvent estimate (\ref{eq:res}). 

\subsubsection{Showing that the decomposition is independent of $\mathcal{E}$ when $E_\infty=0$.}

When $E_\infty=0$,
 $u_{\mathcal A}^{R_0}
=\Pilow \rho_1 w$ (by \eqref{eq:boundary_condition1}), and 
$u_\epsilon= \widetilde{u}_\epsilon$ (by \eqref{eq:A:defAe}); see Figure \ref{fig:split_uL}.
The decomposition and associated bounds are therefore independent of $\mathcal{E}$.

The proof of Theorem \ref{thm:mainbb} is now complete.

\section{Proofs of Theorems \ref{thm:Dirichlet}, \ref{thm:transmission}, and \ref{thm:LSW3} (i.e., the application of Theorem \ref{thm:mainbb} to the Dirichlet, transmission, and full-space problems)}\label{sec:obstacles}

Theorem \ref{thm:LSW3} is proved by directly verifying the assumptions of Theorem \ref{thm:mainbb}.
Theorems \ref{thm:Dirichlet} and \ref{thm:transmission} are proved using the following two corollaries of Theorem \ref{thm:mainbb}. In the first corollary (Corollary \ref{cor:heat}), the low-frequency estimate (\ref{eq:lowenest}) comes from a heat-flow estimate, and 
in the second (Corollary \ref{cor:reg}) from an elliptic-regularity estimate.

\begin{corollary} \label{cor:heat}
Let $P_{\hsc}$ be a semiclassical black-box operator on $\mathcal H$ satisfying  the polynomial resolvent estimate 
 (\ref{eq:res}) in $\subsetH \subset (0, \hsc_0]$. 
Assume further that (i)  $P^{\sharp}_{\hsc}\geq a(\hsc) >0$ for some $a(\hsc) >0$, and (ii) 
for some $\alpha$-family of black-box differentiation operators $(D(\alpha))_{\alpha \in \frak A}$ (Definition \ref{def:BBdiff}), there exists $\rho\in C^\infty(\mathbb T^d_{R_\sharp})$ equal to one near $B_{R_0}$ such that, for some family of subsets $I(\hbar, \alpha) \subset [0, +\infty)$, the following localised heat-flow estimate holds, 
\begin{equation} \label{eq:heatflow}
\N{\rho D(\alpha) \re^{-t P^\sharp_{\hsc}}}_{\mathcal H^{\sharp} \rightarrow \mathcal H^{\sharp} } \leq C(\alpha, t, \hbar ) \quad \tfa  \alpha \in \frak A , \;t\in I(\hbar, \alpha), \; \hsc \in \subsetH.
\end{equation}

Then, if $R>0$ is such that $R_0<R<R_{\sharp}$, $g\in\mathcal H$ is compactly supported in $B_{R}$, and $u \in \mathcal D_{\rm out}$ satisfies (\ref{eq:pde}), there exist $u_{\mathcal A} \in \mathcal D_\hsc^{\sharp, \infty}$ and $u_{ H^2} \in \mathcal D^{\sharp}$ such that $u$ decomposes as (\ref{eq:maindec}). Furthermore, $u_{ H^2}$ satisfies (\ref{eq:decHF})
and there exists $\RfarA, \RfarB, \RlocB, \RlocA$, and $R_\sharp$,
with $R_0<\RfarA<\RfarB<\RlocB<\RlocA<R_\sharp$,
such that $u_{\mathcal A}$ decomposes as $u_{\mathcal A} = u^{R_0}_{\mathcal A} + u^{\infty}_{\mathcal A}$
with, for some $\Lambda > 0$ and $\lambda>1$,
\begin{equation} \label{eq:decLFheat1}
\Vert D(\alpha) u^{R_0}_{\mathcal A} \Vert_{\mathcal H^{\sharp}(B_{\RlocA}) } \lesssim  \inf_{t \in  I(\hbar, \alpha)} C(\alpha, \hbar, t) \re^{\Lambda t}\; \hsc^{-M-1} \Vert g \Vert_{\mathcal H} \quad \tfa  \hsc \in \subsetH \tand \alpha \in \mathfrak A,
\end{equation}
\begin{equation} \label{eq:decLFheat2}
\Vert \partial^\alpha u^{\infty}_{\mathcal A} \Vert_{\mathcal H^{\sharp}((B_{\RfarA})^c) } \lesssim \lambda^{|\alpha|} \hsc^{-|\alpha|-M-1} \Vert g \Vert_{\mathcal H} \quad \tfa  \hsc \in \subsetH \tand \alpha \in \mathfrak A,
\end{equation}
and, for any $N,m>0$ there exists $C_{N,m}>0$ such that
\begin{equation} \label{eq:decLFheat3}
\Vert u^{\infty}_{\mathcal A} \Vert_{\mathcal D_\hsc^{\sharp,m}(B_{\RfarB}) } + \Vert u^{R_0}_{\mathcal A} \Vert_{\mathcal D_\hsc^{\sharp,m}((B_{\RlocB})^c) }  \leq C_{N,m}\hsc^N \Vert g \Vert_{\mathcal H}   \quad \tfa  \hsc \in \subsetH \tand \alpha \in \mathfrak A.
\end{equation}
In addition, if $\rho = 1$, the decomposition (\ref{eq:maindec}) can be constructed in such a way that instead of (\ref{eq:decLFheat1})--(\ref{eq:decLFheat3}), $u_{\mathcal A}$ satisfies
the global regularity estimate
\begin{equation} \label{eq:decLFheat5}
\Vert D(\alpha) u_{\mathcal A} \Vert_{\mathcal H^{\sharp} } \lesssim  \inf_{t \in  I(\hbar, \alpha)} C(\alpha, \hbar, t) \re^{\Lambda t} \; \hsc^{-M-1} \Vert g \Vert_{\mathcal H} \quad \tfa  \hsc \in \subsetH \tand \alpha \in \mathfrak A.
\end{equation}
Finally, the omitted constants in \eqref{eq:decLFheat1}, \eqref{eq:decLFheat2}, and \eqref{eq:decLFheat5} are independent of $\hsc$ and $\alpha$.
\end{corollary}
\begin{proof}
For $\alpha \in \mathfrak A$ and $\hbar \in \subsetH$, let $t\in I(\hbar, \alpha)$, and $\mathcal E_t(\lambda) := \re^{-t|\lambda|}$. Since $P^{\sharp}_{\hsc}\geq a(\hsc) >0$, $\operatorname{Sp}P^{\sharp}_{\hsc} \subset [a(\hsc), \infty)$. Therefore, by 
Parts  \ref{it:fc6} and  \ref{it:fc5} of Theorem \ref{thm:fundfc}, $\re^{-t P^\sharp_{\hsc}} = \mathcal E_t(P^\sharp_{\hsc})$. Such an $\mathcal E_t$ is in $C_0(\mathbb R)$, never vanishes, and satisfies (\ref{eq:lowenest}) with $E_t := \mathcal E_t(P^\sharp_{\hsc})$ and $C_{\mathcal E_t}(\alpha, \hsc) := C(\alpha, \hbar, t)$
by   (\ref{eq:heatflow}). From Theorem \ref{thm:mainbb}, we therefore obtain the above decomposition $u_\mathcal A, u_\mathcal A^{R_0},  u_\mathcal A^{\infty},  u_{H^2}$.
Since $\mathcal E_t(P^\sharp_{\hsc}) = E_t$, by the final part of Theorem \ref{thm:mainbb}, the decomposition is constructed independently of $\mathcal E_t$, and hence independently of $t$. The result then follows, with the infimum in $t$ in (\ref{eq:decLFheat1}) coming from (\ref{eq:decLF1}) and the fact that this estimate in valid for any $t\in I(\hbar, \alpha)$.
\end{proof}

\begin{corollary} \label{cor:reg}
Let $P_{\hsc}$ be a semiclassical black-box operator on $\mathcal H$ satisfying the polynomial resolvent estimate \eqref{eq:res} in $\subsetH \subset (0, \hsc_0]$. Assume further that, 
for some $\alpha$-family of black-box differentiation operators $(D(\alpha))_{\alpha \in \frak A}$ (in the sense of Definition \ref{def:BBdiff}), there exists $L>0$ and $0<L(\alpha)\leq L$ such that the following elliptic-regularity estimate holds,
\begin{equation} \label{eq:reg}
\N{ D(\alpha) w}_{\mathcal H^{\sharp}} \leq 
\sum_{\ell=0}^{L(\alpha)} C_\ell(\alpha,\hsc) \big\Vert (P^{\sharp}_\hsc)^\ell w \big\Vert_{\mathcal H^{\sharp}}
\quad \tfa  \alpha \in \frak A , \; w\in\mathcal D_\hsc^{\sharp, \infty}, \tand \hsc \in \subsetH,
\end{equation}
for some $C_\ell(\alpha, \hsc)>0, \,\ell=0,\ldots,L(\alpha)$.

Then, if $R_0<R<R_{\sharp}$, $g\in\mathcal H$ is compactly supported in $B_{R}$ and $u \in \mathcal D_{\rm out}$ satisfies (\ref{eq:pde}),  there exists $u_{\mathcal A} \in \mathcal D_\hsc^{\sharp, \infty}$, $u_{ H^2} \in \mathcal D^{\sharp}$ such that $u$ can be written as (\ref{eq:maindec}), $u_{H^2}$ satisfies (\ref{eq:decHF}), and $u_{\mathcal A}$ satisfies
\begin{equation} \label{eq:decLFreg}
\Vert D(\alpha) u_{\mathcal A} \Vert_{\mathcal H^{\sharp} } \lesssim  \bigg(\sum_{\ell=0}^{L(\alpha)} C_\ell(\alpha, \hsc)\bigg) \hsc^{-M-1} \Vert g \Vert_{\mathcal H} \quad \tfa  \alpha \in \frak A \tand \hsc \in \subsetH,
\end{equation}
where the omitted constant is independent of $\hsc$ and $\alpha$.
\end{corollary}
\begin{proof}
Let $\rho:=1$, $\mathcal E (\lambda) := \langle \lambda \rangle^{-L}$ and $C_{\mathcal E}(\alpha, \hsc) :=\sum_{\ell=0}^{L(\alpha)} C_\ell(\alpha, \hsc)$.
We now need to show that the bound \eqref{eq:reg} implies that the bound \eqref{eq:lowenest} holds with these choices of $\mathcal E$ and $C_{\mathcal E}$.
Given $v\in D_\hsc^{\sharp, \infty}$, let $w:= \langle
P^\sharp_\hsc\rangle^{-L} v \in D_\hsc^{\sharp, \infty}$.  The bound \eqref{eq:reg} implies that 
\begin{equation} \label{eq:reg2}
\N{\rho D(\alpha)  \langle P^\sharp_\hsc\rangle^{-L} v}_{\mathcal H^{\sharp}} \leq 
\sum_{\ell=0}^{L(\alpha)} C_\ell(\alpha, \hsc) 
 \big\Vert (P^{\sharp}_\hsc)^\ell  \langle P^\sharp_\hsc\rangle^{-L} v\big\Vert_{\mathcal H^{\sharp}}  \quad \tfa  \alpha \in \frak A \tand \hsc \in \subsetH.
\end{equation}
Since $\langle \lambda \rangle^{-L} \lambda^\ell \leq 1$, by Part \ref{it:fc5} of Theorem \ref{thm:fundfc}, the term in brackets on the right-hand side of \eqref{eq:reg2} is bounded by $C_{\mathcal{E}}(\alpha,\hsc) \|v\|_{H^\sharp}$, and then \eqref{eq:lowenest} follows.
The result \eqref{eq:decLFreg} then follows from the bound \eqref{eq:decLF5} in Theorem \ref{thm:mainbb}.
\end{proof}

\

\subsection{Proof of Theorem \ref{thm:Dirichlet}}

Let $\hsc := k^{-1}$, $g := \hsc^2 f$, and define $\mathcal H$ and $P_{\hsc}$ as in Lemma \ref{lem:obstacle}, so that $P_{\hsc}$ is a semiclassical black-box operator on $\mathcal H$. The assumption that $\Csol(k)$ is polynomially bounded means that \eqref{eq:res} holds with 
\beq\label{eq:H}
\subsetH:= \big\{\hsc : \hsc=k^{-1} \text{ with } k\in K\big\}.
\eeq
The plan is to apply Corollary \ref{cor:heat}, showing that the heat-flow estimate \eqref{eq:heatflow} is satisfied using the following theorem.

\begin{theorem}\mythmname{Heat-equation estimate from \cite{EsMoZh:17}}\label{thm:heat}
Suppose that $\obstacle_-, A, c,R_0$, and $R_1$ are as in Definition \ref{def:EDP}.
In addition, assume that $\obstacle_-$ is analytic, and that $A$ and $c$ are $C^\infty$ everywhere and analytic in $B_{R_*}$ for some $R_0<R_*<R_1$.
Let $P^\sharp_\hsc$ denote the associated black-box reference operator on the torus (as described in \S\ref{subsec:bb}).

Given $\rho \in C^\infty_{\rm comp}$ with $\operatorname{supp}\rho \subset B_{R_*}$,
there exists $C>0$ such that % for all $f \in \mathcal{H}^\sharp=L^2(\mathbb{T}^d\cap \obstacle_+),$ and 
for all $t \in (0, 1]$ and for all $\EMZ\in[0,1]$
\beq\label{eq:EMZheat}
\norm{\rho\pa^\alpha \re^{t \hsc^{-2}P_\hbar^\sharp} }_{L^2\to L^2}
 \leq 
\exp(t^{-\EMZ})
  \abs{\alpha}!\,  C^{\abs{\alpha}}t^{(\EMZ-1)|\alpha|/2}.
\eeq
  \end{theorem}
   Note that the operator $\re^{t\hsc^{-2} P_\hbar^\sharp}$
is just the variable coefficient heat
operator for time $t.$ 

\

\bpf[References for the proof of Theorem \ref{thm:heat}]
When $\EMZ=1$, \eqref{eq:EMZheat} is essentially \cite[Theorem 1.1]{EsMoZh:17}, and when $\EMZ=0$, \eqref{eq:EMZheat} is a more-standard heat-equation estimate \cite[Equation 1.5]{EsMoZh:17}, attributed there to \cite[Part 3, \S3]{Fr:69}.

Indeed, the bound with $\EMZ=1$ follows from \cite[Lemma 2.7]{EsMoZh:17} with the choice of their parameter $\theta$ equal to $1$ (via an argument using Sobolev embedding in time, as discussed immediately before \cite[Lemma 2.7]{EsMoZh:17}).
The bound with $\EMZ=0$ follows from \cite[Lemma 2.7]{EsMoZh:17} with $\theta= t$ (since $\sigma=1$ for the heat equation in the notation of \cite[\S2]{EsMoZh:17}), as highlighted in \cite[Remark 2.8]{EsMoZh:17}.
The bound for general $\EMZ\in[0,1]$ then follows from \cite[Lemma 2.7]{EsMoZh:17} with $\theta= t^{1-\EMZ}$.

The main difference between the set up of \cite{EsMoZh:17} and the hypotheses of Theorem \ref{thm:heat} is that \cite{EsMoZh:17} works on a bounded domain with Dirichlet boundary conditions, whereas Theorem \ref{thm:heat} works on the torus with a Dirichlet obstacle inside. 
However, these global considerations only enter the arguments in \cite{EsMoZh:17} in deriving time-analyticity estimates of the heat semi group in \cite[Lemma 2.1]{EsMoZh:17}, and these estimates hold equally well on the torus with a Dirichlet obstacle.
\epf

\

As in Corollary \ref{cor:heat}, we choose $\rho$ to be equal to one near $B_{R_0}$, and further assume that $\rho$ is supported in $B_{(R_0+R_*)/2}$ (i.e., in a region where $A$ and $c$ are known to be analytic).
Given $\hsc \in \subsetH$ and a multi-index $\alpha$, let $\tau = \tau(\hsc, \alpha) \in [0,1]$, depending only on $\hsc$ and $\alpha$, to be fixed later. By letting $t\mapsto t\hsc^2$ in Theorem \ref{thm:heat}, we see that the heat-flow estimate (\ref{eq:heatflow}) is satisfied with
$D(\alpha) := \partial^\alpha$,
$$
C(\alpha, \hbar, t) := 
 \exp\big((\hsc^2t)^{-\EMZ}\big)
  \abs{\alpha}!\,  C^{\abs{\alpha}}  \big(\hsc^2 t)^{(\EMZ-1)|\alpha|/2}
 %M C^{|\alpha|} |\alpha|! (t\hbar^2)^{-|\alpha|/2},
 \quad\tand\quad I(\hbar, \alpha) := (0, \hbar ^{-2}];
$$
note that the heat-flow given by the functional calculus, appearing in (\ref{eq:heatflow}), is indeed the solution of the heat equation; see, e.g., \cite[Theorem  VIII.7]{ReeSim72}. 

  We can therefore apply Corollary \ref{cor:heat} with an arbitrary $R_{\sharp}>R$,
 and we obtain $u_{H^2}\in \mathcal{D}^\sharp$ and $\ulow \in \mathcal{D}^{\sharp,\infty}_\hsc$
 with $u_\mathcal A = u^{R_0}_{\mathcal A}+ u^{\infty}_{\mathcal A}$
 satisfying  (\ref{eq:maindec}),  (\ref{eq:decHF}), 
 \eqref{eq:decLF0}, and the bounds (\ref{eq:decLFheat1})--(\ref{eq:decLFheat3}).  Observe that $u_{H^2}$ and $u_\mathcal A $ satisfy the Dirichlet boundary condition (\ref{eq:Dirtr0}) since they are in $\mathcal D^\sharp$ \eqref{eq:defdomsharp}. 
 
The low-frequency bounds (\ref{eq:decLFheat2})--(\ref{eq:decLFheat3}) give directly the low-frequency bound away from the obstacle (\ref{eq:decLFDir2}) and the error bound (\ref{eq:decLFDirRes}). 
The rest of the proof therefore consists in obtaining 
the low-frequency bound near the obstacle  (\ref{eq:decLFDir1}) from (\ref{eq:decLFheat1}) and the high-frequency bound (\ref{eq:decHFDir}) from  (\ref{eq:decHF}).

To obtain (\ref{eq:decLFDir1}), 
%We first show the low-frequency bound near the obstacle (\ref{eq:decLFDir1}). 
by (\ref{eq:decLFheat1}), we only have to show that, for some $\EMZ\in[0,1]$ and $\mathcal{C}>0$,
\begin{equation}\label{eq:heat_imp0}
\inf_{t\in (0, \hbar^{-2}]} \left(
 \exp\big[(\hsc^2t)^{-\EMZ} +\Lambda t \big]
  \abs{\alpha}!\,  C^{\abs{\alpha} \big(\hsc^2 t)^{(\EMZ-1)|\alpha|/2}}\right)
\leq \mathcal{C}^{|\alpha|} \max \big\{ |\alpha|^{|\alpha|}, \hbar^{-|\alpha| }\big\}.
\end{equation}
We first prove \eqref{eq:heat_imp0} when $|\alpha|\geq \hsc^{-1}$, i.e., when the max on the right equals $\mathcal{C}^{|\alpha|} \alpha^{|\alpha|}$.
If $\EMZ=1$ and $t=\hsc^{-1}$, then the quantity in the infimum on the left-hand side of \eqref{eq:heat_imp0} equals 
\beqs
\exp\big[(1+\Lambda)\hsc^{-1} \big]
  \abs{\alpha}!\,  C^{\abs{\alpha}} \leq (\widetilde{C})^{|\alpha|}|\alpha|^{|\alpha|}
\eeqs
(by Stirling's formula) as required.

To prove \eqref{eq:heat_imp0} when $|\alpha|\leq \hsc^{-1}$, we seek to choose $t$ and $\EMZ$ such that 
\beq\label{eq:MT1}
(\hsc^2 t)^{(\EMZ-1)|\alpha|/2} = \hsc^{-|\alpha|} |\alpha|^{-|\alpha|} \quad
\tand\quad t = (\hsc^2 t)^{-\EMZ}.
\eeq
Under the second equality in \eqref{eq:MT1}, the left-hand side of the first equality becomes $\hsc^{-|\alpha|} t^{-|\alpha|}$; we therefore let $t= |\alpha|$, which is allowed since $|\alpha| \leq \hsc^{-1}\leq \hsc^{-2}$.
We now choose $\EMZ$ such that the  second equality in \eqref{eq:MT1} holds; i.e.,
\beqs
\EMZ= \frac{\log|\alpha|}{\log( \hsc^{-2} |\alpha|^{-1})}.
\eeqs
When $1\leq |\alpha|\leq \hsc^{-1}$, $0\leq \EMZ\leq 1$, and so this choice of $\EMZ$ is allowed. Under the equalities in \eqref{eq:MT1}, the quantity in the infimum on the left-hand side of \eqref{eq:heat_imp0} equals 
\beqs
\exp\big[(1+\Lambda)|\alpha| \big]
|\alpha|! 
  C^{\abs{\alpha}} 
    \hsc^{-|\alpha|}   |\alpha|^{-|\alpha|}
  \leq (\widetilde{C})^{|\alpha|}\hsc^{-|\alpha|},
\eeqs
which is the right-hand side of \eqref{eq:heat_imp0} when $|\alpha|\leq \hsc^{-1}$.
We have therefore proved \eqref{eq:heat_imp0}, and thus the low-frequency bound near the obstacle (\ref{eq:decLFDir1}).

We now complete the proof by proving the high-frequency bound (\ref{eq:decHFDir}). The bound (\ref{eq:decHF}) implies that 
\begin{equation*}
\Vert u_{H^2} \Vert_{L^2(\mathbb T^d_{R_{\sharp}}\backslash \obstacle_-)} +
k^{-2} \Vert \nabla\cdot(A\nabla u_{H^2}) \Vert_{L^2(\mathbb T^d_{R_{\sharp}}\backslash \obstacle_-)} \lesssim k^{-2} \Vert f \Vert_{L^2(B_R \cap \obstacle_+)},
\end{equation*}
and then Green's first identity (see, e.g., \cite[Lemma 4.3]{Mc:00}) and the fact that $A$ satisfies \eqref{eq:Aelliptic} imply that 
\begin{equation}\label{eq:dirH12}
\Vert u_{H^2} \Vert_{L^2(\mathbb T^d_{R_{\sharp}}\backslash \obstacle_-)} + k^{-1}\Vert \nabla u_{H^2} \Vert_{L^2(\mathbb T^d_{R_{\sharp}}\backslash \obstacle_-)}  +  k^{-2} \Vert \nabla\cdot(A\nabla u_{H^2})\Vert_{L^2(\mathbb T^d_{R_{\sharp}}\backslash \obstacle_-)} \lesssim k^{-2} \Vert f \Vert_{L^2(B_R \cap \obstacle_+)};
\end{equation}
see, e.g., \cite[Lemma 3.10]{GrPeSp:19}.
That is, (\ref{eq:decHFDir}) holds for $|\alpha| = 0$ and $1$. To obtain (\ref{eq:decHFDir}) for $|\alpha| = 2$, we combine (\ref{eq:dirH12}) with the $H^2$ regularity result of, e.g., \cite[Part (i) of Theorem 4.18, pages 137-138]{Mc:00}, applied with $\Omega_1 = B_R \cap \obstacle_+$ and $\Omega_2 = B_{(R+R_{\sharp})/2} \cap \obstacle_+$.
Finally, the fact that $ u^{R_0}_{\mathcal A} $ is analytic in  $B_{\RlocA}$ and $u^{\infty}_{\mathcal A}$ is analytic in $(B_{\RfarA})^c$ 
follows from Lemma \ref{lem:analytic} and the bounds \eqref{eq:decLFDir1} and \eqref{eq:decLFDir2}, respectively.

\subsection{Proof of Theorem \ref{thm:transmission}}
The plan is to apply Corollary \ref{cor:reg}. Let $\hsc := k^{-1}$, $g := \hsc^2 f$, and define $\mathcal H$ and $P_{\hsc}$ as in Lemma \ref{lem:obstacle}. By Lemma \ref{lem:obstacle}, $P_{\hsc}$ is a semiclassical black-box operator on $\mathcal H$. 

The assumption that $\Csol(k)$ is polynomially bounded means that \eqref{eq:res} holds with $\subsetH$ given by \eqref{eq:H}
and thus we only need to show that the regularity estimate \eqref{eq:reg} is satisfied for appropriate $D(\alpha), C_\ell(\alpha,\hsc)$, and $L(\alpha)$. 

We claim that for $n$ even with $n\leq 2m$
\beq\label{eq:induction1}
\Vert w \Vert_{H^{n}(\obstacle_-) \oplus
  H^{n}(\mathbb T ^d_{R_{\sharp}}\cap \obstacle_+)} \leq
\sum_{\ell=0}^{n/2}
  \widetilde{C}_\ell(n)
\N{ ( \nabla \cdot (A\nabla ) )^\ell w }_{L^2(\obstacle_-)\oplus
  L^2(\mathbb T ^d_{R_{\sharp}}\cap \obstacle_+)}  \quad\tfa w\in \mathcal{D}^{\sharp, \infty}_\hsc,
\eeq
where $\widetilde{C}_\ell(n)$ also depends on $\obstacle_-, A$,
and $c$. If \eqref{eq:induction1} holds, then the regularity estimate
(\ref{eq:reg}) is satisfied with (i) $D(\alpha) :=
(\partial^\alpha|_{\obstacle_-}, \partial^\alpha|_{\obstacle_+})$,
(ii) $\frak A$ consisting of multi-indices $\alpha$ such that
$|\alpha|$ is even and $|\alpha|\leq 2m$, (iii)
$L(\alpha):=|\alpha|/2,$ and (iv) 
\beq\label{eq:Cell}
C_\ell(\alpha,\hsc):=\hsc^{-2\ell}\widetilde{C}_\ell(|\alpha|).
\eeq
We assume that \eqref{eq:induction1} holds, and show how the result of the theorem follows from Corollary \ref{cor:reg}.
Applying this corollary, we obtain $u_{H^2}, u_\mathcal A$ satisfying  (\ref{eq:maindec}),  (\ref{eq:decHF}), and (\ref{eq:decLFreg}). Observe that $u_{H^2}$ and $u_{\mathcal A}$ satisfy the transmission conditions (\ref{eq:trcont}) since they are in $\mathcal D^\sharp$.
By \eqref{eq:Cell}, there exists $C_2=C_2(m)>0$ such that, for $|\alpha|\leq 2m$,
\beqs
\sum_{\ell=0}^{L(\alpha)} C_\ell(\alpha, \hsc) \leq C_2(m) \hsc^{-|\alpha|}.
\eeqs
The low-frequency bound (\ref{eq:decLFreg}) therefore gives (\ref{eq:decLFTR})
for all $\alpha\in \frak A$, i.e., for all $\alpha$ with $|\alpha|$ even and $\leq 2m$. The bound (\ref{eq:decLFTR}) then holds for all $\alpha$ with $|\alpha|\leq 2m$ by interpolation (see, e.g., \cite[Theorem B.8]{Mc:00}, \cite[\S4.2]{ChHeMo:15}).
Finally, (\ref{eq:decHFTR})
follows from the high-frequency estimate (\ref{eq:decHF}), together with Green's identity and (\ref{eq:induction1}) applied with $n = 2$ (similar to the end of the proof of Theorem \ref{thm:Dirichlet}).

We therefore only need to prove \eqref{eq:induction1}. The two ingredients to do this are the regularity result 
\beq\label{eq:ingredient1}
\N{v}_{H^{n+2}(\obstacle_-) \oplus H^{n+2}(\mathbb T ^d_{R_{\sharp}}\cap \obstacle_+) }
\lesssim \N{ \nabla \cdot (A\nabla  v) }_{H^n(\obstacle_-)\oplus
  H^n(\mathbb T ^d_{R_{\sharp}}\cap \obstacle_+)} 
  + \N{ v }_{H^1(\obstacle_-)\oplus
  H^1(\mathbb T ^d_{R_{\sharp}}\cap \obstacle_+)}
\eeq
for all integers $n\leq 2m-2$,
and the bound 
\beq\label{eq:ingredient2}
\N{ v }_{H^1(\obstacle_-)\oplus
  H^1(\mathbb T ^d_{R_{\sharp}}\cap \obstacle_+)}
\lesssim   \N{ \nabla \cdot (A\nabla v) }_{L^2(\obstacle_-)\oplus
  L^2(\mathbb T ^d_{R_{\sharp}}\cap \obstacle_+)} + \N{ v }_{L^2(\obstacle_-)\oplus
  L^2(\mathbb T ^d_{R_{\sharp}}\cap \obstacle_+)},
  \eeq
where both bounds are valid for all $v\in \mathcal{D}^\sharp$, and the omitted constants in both depend on $A, c,$ and $\beta$.

The bound \eqref{eq:ingredient2} is proved using Green's first identity (see, e.g., \cite[Lemma 4.3]{Mc:00}), the fact that $v$ satisfies the transmission conditions in \eqref{eq:domain_transmission}, 
and the fact that $A$ satisfies \eqref{eq:Aelliptic}; see, e.g., \cite[Lemma 3.10]{GrPeSp:19} for an analogous bound in $\Rea^d$ for the case $\beta=1$.

Regarding \eqref{eq:ingredient1}: standard elliptic regularity results imply that, given $\Omega_1,\Omega_2$ with $\obstacle_-\Subset \Omega_1\Subset\Omega_2\Subset B_{R^\sharp}$, 
\begin{align}\nonumber
\Vert v \Vert_{H^{n+2}(\obstacle_-) \oplus H^{n+2}(\Omega_1\cap \obstacle_+)} &\lesssim
\Vert \nabla \cdot (A\nabla v) \Vert_{H^n(\obstacle_-)\oplus H^n(\Omega_2\cap \obstacle_+)} + \Vert v  \Vert_{H^1(\obstacle_-)\oplus H^1(\Omega_2\cap \obstacle_+)} \\
&\leq  \Vert \nabla \cdot (A\nabla v) \Vert_{H^n(\obstacle_-)\oplus H^n(\mathbb T ^d_{R_{\sharp}}\cap \obstacle_+)} + \Vert v  \Vert_{H^1(\obstacle_-)\oplus H^1(\mathbb T ^d_{R_{\sharp}}\cap \obstacle_+)},\label{eq:Tmainevent}
\end{align}
for all $v\in \mathcal{D}^\sharp$ and integers $n\leq 2m-2$, where the omitted constant depends on $A, c, \beta$;
see, e.g., \cite[Theorem 4.20]{Mc:00}, \cite[Theorem 5.2.1, Part (i)]{CoDaNi:10}.
Since the torus is compact (and is thus covered by a finite number of $\Omega_1$s),
 \eqref{eq:Tmainevent} holds with the left-hand side replaced by
$\Vert v \Vert_{H^{n+2}(\obstacle_-) \oplus
  H^{n+2}(\obstacle_+\cap \mathbb T^d_{R_\sharp})}$ and \eqref{eq:ingredient1} follows.

We now use \eqref{eq:ingredient1} and \eqref{eq:ingredient2} to prove \eqref{eq:induction1} by induction. 
The bound \eqref{eq:induction1} with $n=2$ follows from 
combining \eqref{eq:ingredient1} with $n=0$ and $v=w$ and \eqref{eq:ingredient2} with $v=w$ (observe that choosing $v=w$ in both is allowed since $w \in \mathcal{D}^\sharp$).
We now assume that we have proved \eqref{eq:induction1} for $n$ even and $n\leq 2q$ for some $0\leq q\leq m-1$; i.e.,
\beq\label{eq:induction2}
\Vert w \Vert_{H^{2q}}\lesssim
\sum_{\ell=0}^{q} 
\N{ ( \nabla \cdot (A\nabla ) )^\ell w }_{L^2}  \quad\tfa w\in \mathcal{D}^{\sharp, \infty}_\hsc,
\eeq
where we have omitted the $q$-dependent constants and 
the domains of the norms for brevity.

Applying \eqref{eq:ingredient1} with $n=2q$ and $v=w$, we have
\beq\label{eq:induction3}
\N{w}_{H^{2q+2}}
\lesssim \N{ \nabla \cdot (A\nabla  w) }_{H^{2q}}
  + \N{ w }_{H^1}
\eeq
(again omitting the domains of the norms for brevity).
The desired bound \eqref{eq:induction1} with $n=2q+2$ then follows by using in \eqref{eq:induction3} the inequality \eqref{eq:induction2} with $w$ replaced by $\nabla \cdot (A\nabla  w)$ (which is allowed since 
$w \in \mathcal{D}^{\sharp,\infty}_\hsc$ implies that $P^\sharp_\hsc w \in \mathcal{D}^{\sharp,\infty}_\hsc$ by \eqref{eq:Dinfty}),
and then using \eqref{eq:ingredient2} with $v=w$.

\subsection{Proof of Theorem \ref{thm:LSW3}}\label{sec:LSW3proof}

Let $\hsc := k^{-1}$, $g := \hsc^2 f$, and define $\mathcal H$ and $P_{\hsc}$ as in Lemma \ref{lem:obstacle} with $\obstacle_- = \emptyset$. By Lemma \ref{lem:obstacle}, $P_{\hsc}$ is a semiclassical black-box operator on $\mathcal H$. The reference operator is given by $P^\sharp_\hsc = -\hsc^{2} c^2 \nabla \cdot (A\nabla )$, acting on the torus $\mathbb T^d_{R_\sharp}$.

The assumption that $\Csol(k)$ is polynomially bounded means that 
the bound \eqref{eq:res} holds with $\subsetH$ given by \eqref{eq:H}; i.e., the assumption in Point 1 of Theorem \ref{thm:mainbb} is satisfied.

We now construct $\mathcal E$ and $E $ satisfying the assumptions in Point 2 of Theorem \ref{thm:mainbb}.
Let $\Lambda>0$ be as in Theorem \ref{thm:mainbb}, and let $\mathcal E \in C^\infty_{\rm comp}(\mathbb R)$ be such that $\mathcal E = 1$ in $[-\Lambda, \Lambda]$, and $\mathcal E = 0$ outside
$[-2\Lambda, 2\Lambda]$. The results of Helffer-Robert  \cite{HeRo:83} (see the account in \cite{Rob87}) imply that
$\mathcal E (P^\sharp_\hsc) = \mathcal E(-\hsc^{2} c^2 \nabla \cdot \big(A\nabla))$ is a pseudo-differential operator on the torus $\mathbb T^d_{R_\sharp}$. Then, the same argument as in the proof of Lemma \ref{thm:funcloc2} shows that
$$
\operatorname{WF}_\hsc  \mathcal E\big(-\hsc^{2} c^2 \nabla \cdot (A\nabla)\big) \subset  \operatorname{supp} \mathcal E \circ q,
$$
where $q(x, \xi) = c(x)^2  \langle A(x) \xi ,\xi \rangle$
is the semi-classical principal symbol of $-\hsc^{2} c^2 \nabla \cdot (A\nabla)$. Hence, since $\mathcal E$ is compactly supported and $A$ satisfies \eqref{eq:Aelliptic}, there exists $\Lambda_0>0$ such that
\begin{equation} \label{eq:cor:1}
\operatorname{WF}_\hsc  \mathcal E\big(-\hsc^{2} c^2 \nabla \cdot (A\nabla)\big) \subset \mathbb T^d_{R_\sharp} \times B\left(0, \frac{\Lambda_0}{2}\right).
\end{equation}
Let $\widetilde{\varphi} \in C^\infty_{\rm comp}$ be compactly supported in $B(0, \Lambda_0^2)$ and equal to one on $B(0,\Lambda_0^2/4)$. By (\ref{eq:cor:1}) and (\ref{eq:support}), $\operatorname{WF}_\hsc \big( 1 - \Optorus(\widetilde{\varphi}(|\xi|^2)) \big) \cap\operatorname{WF}_\hsc  \mathcal E(-\hsc^{2} c^2 \nabla \cdot \big(A\nabla))= \emptyset$, therefore, by \eqref{eq:WFdis},
\begin{equation*} 
 \mathcal E\big(-\hsc^{2} c^2 \nabla \cdot (A\nabla)\big)  = \Optorus(\widetilde{\varphi}(|\xi|^2))   \mathcal E\big(-\hsc^{2} c^2 \nabla \cdot (A\nabla )\big) + O(\hsc^\infty)_{\Psi^{-\infty}_\hsc}.
\end{equation*}
Then, by Lemma \ref{lem:toruscalculus},
\begin{equation} \label{eq:cor:2}
 \mathcal E(-\hsc^{2} c^2 \nabla \cdot \big(A\nabla ))  = \widetilde{\varphi}(-\hsc^2\Delta)  \mathcal E(-\hsc^{2} c^2 \nabla \cdot \big(A\nabla)) + O(\hsc^\infty)_{\Psi^{-\infty}_\hsc}.
\end{equation}
We now define
\beq\label{eq:LSW3E0}
E  :=  \widetilde{\varphi}\big(-\hsc^2\Delta)  \mathcal E(-\hsc^{2} c^2 \nabla \cdot (A\nabla )\big),
\eeq
and thus (\ref{eq:cor:2}) implies that
$$
 \mathcal E(P^\sharp_\hsc) = E  + \residual.
$$
We now need to show that a low-frequency estimate of the form \eqref{eq:lowenest} is satisfied. Since $\widetilde{\varphi}$ is compactly supported in $B(0, \Lambda_0^2)$, the definition of $E $ \eqref{eq:LSW3E0} and the same argument used to show the bound \eqref{eq:low:A:3} imply that 
$$
\Vert \partial^\alpha E  v \Vert_{L^2(\mathbb T^d_{R_\sharp})} \leq  \Lambda_0^{|\alpha|} \hsc^{-|\alpha|} \Vert  \mathcal E(-\hsc^{2} c^2 \nabla \cdot \big(A\nabla v)) v\Vert_{L^2(\mathbb T^d_{R_\sharp})} \quad \tfa v \in L^2(\mathbb T^d_{R_\sharp}) \text{ and multi-indices } \alpha.
$$
Then, since $\mathcal E(-\hsc^{2} c^2 \nabla \cdot \big(A\nabla)) \in \Psi^{-\infty}_\hsc(\mathbb T^d_{R_\sharp})$, there exists $C>0$ such that 
$$
\Vert \partial^\alpha E  v \Vert_{L^2(\mathbb T^d_{R_\sharp})} \leq C \Lambda_0^{|\alpha|} \hsc^{-|\alpha|} \Vert v \Vert_{L^2(\mathbb T^d_{R_\sharp})}  \quad \tfa v \in L^2(\mathbb T^d_{R_\sharp})\text{ and multi-indices } \alpha.
$$ 
Therefore, the assumption in Point 2 of Theorem \ref{thm:mainbb} is satisfied with $D(\alpha):=\partial^\alpha$, $C_{\mathcal E}(\alpha, \hsc) := C\Lambda_0^{|\alpha|} \hsc^{-|\alpha|} $ and $\rho = 1$. The result then follows from Theorem \ref{thm:mainbb}; indeed, the bound \eqref{eq:decLFLSW3} follows immediately from \eqref{eq:decLF5}, and 
 \eqref{eq:decHFLSW3} follows from \eqref{eq:decHF} after using Green's identity and elliptic regularity in the same way as at the end of the proof of Theorem \ref{thm:Dirichlet} -- see \eqref{eq:dirH12} and the surrounding text.

\section{Proofs of Theorems \ref{thm:FEM1} and \ref{thm:FEM2} and Corollary \ref{cor:1} (the frequency-explicit results about the convergence of the FEM)}
\label{sec:FEM}

\subsection{Recap of FEM convergence theory}

The two ingredients for the proof of Theorems \ref{thm:FEM1} and \ref{thm:FEM2} are 
\bit
\item Lemma \ref{lem:Schatz}, which is the standard duality argument giving a condition for quasioptimality to hold in terms of how well the solution of the adjoint problem 
is approximated by the finite-element space (measured by the quantity  $\eta(V_N)$ defined by \eqref{eq:etadef}), and  
\item Lemma \ref{lem:MS} that bounds $\eta(V_N)$ using the decomposition from Theorems \ref{thm:Dirichlet} and \ref{thm:transmission}.
\eit 
Regarding Lemma \ref{lem:Schatz}:  this argument came out of ideas introduced in \cite{Sc:74}, was then formalised in \cite{Sa:06}, and has been used extensively in the analysis of the Helmholtz FEM; see, e.g., \cite{AzKeSt:88, IhBa:95a, Me:95, Sa:06, MeSa:10, MeSa:11, ZhWu:13, Wu:14, DuWu:15, ChNi:18, LiWu:19, ChNi:20, GaChNiTo:18, GrSa:20, GaSpWu:20, LaSpWu:22}.

Before stating Lemma \ref{lem:Schatz} we need to introduce some notation.
Let $\Ccont = \Ccont(A, c^{-2}, R, k_0)$ be the \emph{continuity constant} of the sesquilinear form $a(\cdot,\cdot)$ (defined in \eqref{eq:sesqui}) in the norm $\|\cdot\|_{H^1_k(B_R\cap \obstacle_+)}$; i.e.
\beqs
|a(u,v)|\leq \Ccont \N{u}_{H^1_k(B_R\cap \obstacle_+)} \N{v}_{H^1_k(B_R\cap \obstacle_+)} \quad\tfa u, v \in H^1(B_R\cap \obstacle_+).
\eeqs
By the Cauchy-Schwarz inequality and \eqref{eq:CDTN1},
\beq\label{eq:Cc}
\Ccont \leq \max\{ A_{\max}, c_{\min}^{-2} \} + \CDTN.
\eeq
The following definitions are stated for the sesquilinear form of the Dirichlet problem \eqref{eq:sesqui}. For the sesquilinear form of the transmission problem with the transmission parameter $\beta=1$, one only needs to replace $B_R\cap \obstacle_+$ by $B_R$ and define $c$ to be 
equal to one in $B_R\cap \obstacle_+$.

\begin{definition}[The adjoint sesquilinear form $a^*(\cdot,\cdot)$]\label{def:adjoint}
The adjoint sesquilinear form, $a^*(u,v)$, to the sesquilinear form $a(\cdot,\cdot)$ defined in \eqref{eq:sesqui} is given by
\beqs
a^*(u,v) := \overline{a(v,u)}= \int_{B_R\cap \obstacle_+} 
\left((A \gu)\cdot\gvb
 - \frac{k^2}{c^2} u\vb\right) - \big\langle u,\DtN(v)\big\rangle_{\partial B_R}.
\eeqs
\end{definition}

\begin{definition}[Adjoint solution operator $\cS^*$]
Given $f\in L^2(B_R\cap \obstacle_+)$, let $\cS^*f$ be defined as the solution of the variational problem
\beq\label{eq:S*vp}
\tfind \cS^*f \in H^1(B_R\cap \obstacle_+) \quad\tst\quad a^*(\cS^*f,v) = \int_{B_R\cap \obstacle_+} f\, \overline{v} \quad\tfa v\in H^1(B_R\cap \obstacle_+).
\eeq
\end{definition}

Green's second identity applied to solutions of the Helmholtz equation satisfying the Sommerfeld radiation condition \eqref{eq:src} implies that
$\big\langle \DtN \psi, \overline{\phi}\big\rangle_{\GR} =\big\langle \DtN \phi, \overline{\psi}\big\rangle_{\GR} $ for all $\phi,\psi\in H^{1/2}(\GR)$ 
(see, e.g., \cite[Lemma 6.13]{Sp:15}); thus $a(\overline{v},u) = a(\overline{u},v)$ and so the definition \eqref{eq:S*vp} implies that
\beq\label{eq:S*fkey}
a(\overline{\cS^*f}, v)= (\overline{f},v)_{L^2(\OR)}\quad\tfa v\in H^1(B_R\cap \obstacle_+).
\eeq

\begin{definition}[$\eta(V_N)$]
Given a sequence $(V_N)_{N=0}^\infty$ of finite-dimensional subspaces of 
let
\beq\label{eq:etadef}
\eta(V_N):= \sup_{0\neq f\in L^2(B_R\cap \obstacle_+)}
\min_{v_N\in V_N} \frac{\N{S^*f-v_N}_{H^1_k(B_R\cap \obstacle_+)}}{\big\|
f\big\|_{L^2(B_R\cap \obstacle_+)}}.
\eeq
\end{definition} 

\ble[Conditions for quasioptimality]\label{lem:Schatz}
If 
\beqs
k\,\eta(V_N) \leq \frac{1}{\Ccont} \sqrt{ \frac{A_{\min}}{2\big( A_{\min}+c^{-2}_{\min}} \big)},
\eeqs
then the Galerkin equations \eqref{eq:FEM} have a unique solution which satisfies
\beqs
\N{u-u_N}_{H^1_k(B_R\cap \obstacle_+)}
 \leq \frac{2\Ccont}{A_{\min}}\left(\min_{v_N\in V_N} \N{u-v_N}_{H^1_k(B_R\cap \obstacle_+)}\right).
\eeqs
\ele

\bpf[References for the proof] See, e.g., \cite[Lemma 6.4]{LaSpWu:22}.
\epf

\

The following two lemmas are proved in the next subsections.

\ble[Bound on $\eta(V_N)$ for the exterior Dirichlet problem]\label{lem:MS}
Let $d=2$ or $3$. Suppose that $\obstacle_-, A, c,R, \RfarA$, and $\RlocA$ are as in Theorem \ref{thm:Dirichlet}, and that $\Csol(k)$ is polynomially bounded for $k\in K$.

Let $(V_N)_{N=0}^\infty$ be the piecewise-polynomial approximation spaces described in \cite[\S5]{MeSa:10}, \cite[\S5.1.1]{MeSa:11}. 
%assume further that the triangulations fit $B_{\RfarA}$ and $B_{\RlocA}$ exactly.

Given $k_0>0$ and $N>0$ there exist 
\bit
\item
$\mathcal{C}_1,\mathcal{C}_2, \sigma>0$, depending on $A, c, R$, $d$, and $k_0$, but independent of $k$, $h$, $p$, and $N$, and
\item $C_N$ depending on $A, c, R$, $d$, $k_0$, and $N$, but independent of $k$, $h$, $p$,
\eit
such that
\beq\label{eq:etabound}
k\,\eta(V_N) \leq \mathcal{C}_1 \frac{hk}{p}\left(1+ \frac{hk}{p}\right) + \mathcal{C}_2
k^{M}
 \left(
\left(\frac{h}{h+\sigma}\right)^p  
+k\left(
 \frac{hk}{\sigma p}\right)^p\right)
+C_N k^{1-N}
  \quad\,\, k \in K\cap[k_0,\infty).
\eeq
\ele

\ble[Bound on $\eta(V_N)$ for the transmission problem]\label{lem:etaT}
Let $d=2$ or $3$ and let $\beta=1$. Suppose that $A, c,$ and $\obstacle_-$ are as in Definition \ref{def:transmission} and, given an integer $p$, satisfy the regularity assumptions in Theorem \ref{thm:FEM2}.
Suppose that $\Csol(k)$ is polynomially bounded for $k\in K$.

Let $(V_N)_{N=0}^\infty$ be a sequence of  piecewise-polynomial approximation spaces of degree $p$ satisfying Assumption \ref{ass:VN}.

Given $k_0>0$, there exist $\widetilde{\mathcal{C}}_1,\widetilde{\mathcal{C}}_2$, depending on $A, c, R$, $d$, $k_0$, and $p$, but independent of $k$ and $h$, such that 
\beq\label{eq:etaboundT}
k\,\eta(V_N) \leq \big(1 +hk\big)\Big(\widetilde{\mathcal{C}}_1 hk + \widetilde{\mathcal{C}}_2 \, 
k^{M+1} (hk)^{p}\Big)
\quad\tfa k \in K\cap [k_0,\infty).
\eeq
\ele

\bpf[Proof of Theorems \ref{thm:FEM1}/\ref{thm:FEM2} assuming Lemmas \ref{lem:MS}/\ref{lem:etaT}]
Theorem \ref{thm:FEM2} follows immediately by combining Lemmas \ref{lem:Schatz} and \ref{lem:etaT}
and the inequality \eqref{eq:Cc}.

Theorem \ref{thm:FEM1} follows in a similar way
(and is essentially the same as the proof of \cite[Theorem 5.8]{MeSa:11}), except that we first choose $N>1$, and then let $k_1>0$ be such that 
\beqs
C_N k^{1-N} \leq 
 \frac{1}{2\Ccont} \sqrt{ \frac{A_{\min}}{2\big(A_{\min}+c^{-2}_{\min} \big)}} \quad\tfa k \geq k_1.
\eeqs
Theorem \ref{thm:FEM1} then follows by 
using this bound in \eqref{eq:etabound} and then combining the resulting inequality with Lemma \ref{lem:Schatz} and the inequality \eqref{eq:Cc}.
\epf

\subsection{Proof of Lemma \ref{lem:MS}}

Given $f\in L^2(B_R\cap \obstacle_+)$, let $v= \mathcal{S}^*f$. By \eqref{eq:S*fkey} and Theorem \ref{thm:Dirichlet}, $v = v_{H^2} + v_{\cA}$, where $v_{H^2}$ and $v_{\cA}$ satisfy the bounds \eqref{eq:decHFDir}--\eqref{eq:decLFDirRes} with $u$ replaced by $v$. 

The proof of Lemma \ref{lem:MS} is very similar to the proofs of \cite[Theorem 5.5]{MeSa:10} and \cite[Proposition 5.3]{MeSa:11} (covering the constant-coefficient Helmholtz equation in, respectively, $\Rea^d$ and the exterior of an analytic Dirichlet obstacle). 

The only difference is that in \cite{MeSa:10}, \cite{MeSa:11} the function $v_\cA$ is analytic on the whole of $B_R\cap\obstacle_+$, whereas here $v_\cA= v_\cA^{R_0}+ v_\cA^\infty$ with $v_\cA^{R_0}$ and $v_\cA^\infty$ analytic in subsets of the domain and $O(k^{-\infty})$ in the complements of these subsets; see \eqref{eq:decLFDir1}-\eqref{eq:decLFDirRes} and Figure \ref{fig:line2}.
The consequence is that $C_N k^{1-N}$ appears on the right-hand side of \eqref{eq:etabound}, but this term is not present on the right-hand sides of the analogous bounds in  \cite[Theorem 5.5]{MeSa:10} and \cite[Proposition 5.3 and Equation 5.11]{MeSa:11}. Since this term can be made arbitrarily-small for $k$ sufficiently large, the only consequence is that Lemma \ref{lem:MS} and Theorem \ref{thm:FEM1} are valid for $k$ sufficiently large (as opposed to for all $k\geq k_0$ with $k_0$ arbitrary).

Exactly as in the proof of \cite[Theorem 5.5]{MeSa:10}, there exists $\mathcal{C}_3>0$ (dependent only on the constants in \cite[Assumption 5.2]{MeSa:10} defining the element maps from the reference element) such that 
\beq\label{eq:bae1}
\min_{w_N\in V_N}\N{v - w_N}_{H^1_k(B_R\cap \obstacle_+)}
\leq  \mathcal{C}_3
\frac{h}{p}\left(1 + \frac{hk}{p}\right)  |v|_{H^2(B_R\cap \obstacle_+)},
\eeq
for all $v\in H^2(B_R\cap\obstacle_+)$; recall that this result follows from the polynomial-approximation result of \cite[Theorem B.4]{MeSa:10} and the definition \eqref{eq:1knorm} of the norm $\|\cdot\|_{H^1_k}$. 
Applying the bound \eqref{eq:bae1} to $\vhigh$ and using \eqref{eq:decHFDir} with $|\alpha|=2$, we obtain 
\beqs
\min_{w_N\in V_N}\N{\vhigh - w_N}_{H^1_k(B_R\cap \obstacle_+)}
 \leq
 \mathcal{C}_3 \,C_1\,
 \frac{h}{p}\left(1 + \frac{hk}{p}\right) \N{f}_{L^2(B_R\cap \obstacle_+)};
\eeqs
we then let $\mathcal{C}_1:= C_1\, \mathcal{C}_3$.

To prove \eqref{eq:etabound}, therefore, we only need to show that   
\beq\label{eq:intermediate1}
\min_{w_N\in V_N}\N{\vlow - w_N}_{H^1_k(B_R\cap \obstacle_+)}
\leq\left(
\mathcal{C}_2 \,
k^{M}
 \left(
\left(\frac{h}{h+\sigma}\right)^p  
+
k\left( \frac{hk}{\sigma p}\right)^p\right)
  +C_N k^{-N}
  \right)
\N{f}_{L^2(B_R\cap \obstacle_+)},
\eeq
for some $\mathcal{C}_2>0$ independent of $k,h,p,$ and $N$ and some $C_N>0$ independent of $k,h,$ and $p$.
Recall the regions where $\vlow^{R_0}$ and $\vlow^\infty$ are analytic (see Figure \ref{fig:line2}).
Given $V_N$, choose $D_1$ such that (i) $D_1$ is a union of elements of the triangulation associated with $V_N$ and (ii) $B_{\RlocB}\Subset D_1 \Subset B_{\RlocA}$. Thus, by 
\eqref{eq:decLFDirRes}, 
\begin{align*}
\min_{w_N\in V_N}\big\|\vlow^{R_0} - w_N\big\|_{H^1_k(B_R\cap \obstacle_+)} 
&\leq \min_{w_N\in V_N}\big\|\vlow^{R_0} - w_N\big\|_{H^1_k(D_1\cap \obstacle_+)} + \big\|\vlow^{R_0}\big\|_{H^1_k(B_R \cap (D_1)^c)}\\
&\leq \min_{w_N\in V_N}\big\|\vlow^{R_0} - w_N\big\|_{H^1_k(D_1\cap \obstacle_+)} + C_N' k^{-N}\N{f}_{L^2(B_R\cap \obstacle_+)}
\end{align*}
for some $C_N'>0$ independent of $k,h,$ and $p$. Similarly, with $D_2$ a union of elements of the triangulation and such that 
$B_{\RfarA}\Subset D_2 \Subset B_{\RfarB}$,
\beqs
\min_{w_N\in V_N}\big\|\vlow^{\infty} - w_N\big\|_{H^1_k(D_2\cap \obstacle_+)} 
\leq \min_{w_N\in V_N}\big\|\vlow^{\infty} - w_N\big\|_{H^1_k(B_R\cap (D_2)^c)} + C_N'' k^{-N}\N{f}_{L^2(B_R\cap \obstacle_+)}
\eeqs
for some $C_N''>0$, independent of $k,h,$ and $p$. 
To prove \eqref{eq:intermediate1}, therefore, we only need to show that 
\beq
\min_{w_N\in V_N}\big\|\vlow^{R_0} - w_N\big\|_{H^1_k(
D_1
\cap \obstacle_+)}
\leq
\frac{\mathcal{C}_2}{2} \,
k^{M}
 \left(
\left(\frac{h}{h+\sigma}\right)^p  
+
k \left(\frac{hk}{\sigma p}\right)^p\right)
\N{f}_{L^2(B_R\cap \obstacle_+)}
\label{eq:intermediate2}
\eeq
and
\beq
\min_{w_N\in V_N}\big\|\vlow^{\infty} - w_N\big\|_{H^1_k(B_R\cap (D_2
)^c)}
\leq
\frac{\mathcal{C}_2}{2} \,
k^{M}
 \left(
\left(\frac{h}{h+\sigma}\right)^p  
+
k \left(\frac{hk}{\sigma p}\right)^p\right)
\N{f}_{L^2(B_R\cap \obstacle_+)},
\label{eq:intermediate3}
\eeq
for some $\mathcal{C}_2>0$, independent of $k,h,p,$ and $N$.
Note that (i) we introduced $D_1$ and $D_2$ so that the domains on which $\vlow^{R_0}$ and $\vlow^\infty$ are approximated in \eqref{eq:intermediate2} and \eqref{eq:intermediate3} are exactly triangulated by the mesh, and (ii)  for the approximation \eqref{eq:intermediate2}, it is important that $\vlow^{R_0}= 0$ on $\partial \obstacle^+$, since the space $V_N$ has this zero Dirichlet boundary condition imposed.
 
The bounds \eqref{eq:intermediate2} and \eqref{eq:intermediate3} then follow from \cite[Proposition 5.3]{MeSa:11} (which uses \cite[Theorem 5.5]{MeSa:10}); the key point is that $\vlow^{\infty}$ and $\vlow^{R_0}$ satisfy the same type of bound -- namely that  in Part (iii) of Lemma \ref{lem:analytic} -- as $u_{\mathcal A}$ in \cite{MeSa:11} (see the second displayed equation in \cite[Theorem 4.20]{MeSa:11}, and note that $\alpha$ in \cite{MeSa:11} equals our $M$).
 
\subsection{Proof of Lemma \ref{lem:etaT}}

Given $f\in L^2(B_R)$, let $v= \mathcal{S}^*f$. By \eqref{eq:S*fkey} and Theorem \ref{thm:transmission}, $v = v_{H^2} + v_{\cA}$, where $v_{H^2}$ and $v_{\cA}$ satisfy the bounds \eqref{eq:decHFTR} and \eqref{eq:decLFTR} with $u$ replaced by $v$. 

By the definition of the $H^1_k$ norm \eqref{eq:1knorm} and the bound \eqref{eq:interp}, there exists $\Cint=\Cint(\ell,d)>0$ such that 
\beq\label{eq:interp2}
\min_{w_N \in V_N}\N{w- w_N}_{H^1_k(B_R)} \leq \Cint(\ell,d) \big(1+hk\big) h^\ell \Big( \N{w_+}_{H^{\ell+1}(B_R\cap\obstacle_+)} +
\N{w_-}_{H^{\ell+1}(\obstacle_-)} \Big)
\eeq
for all $w=(w_+, w_-) \in H^{\ell+1}(B_R\cap\obstacle_+)\times H^{\ell+1}(\obstacle_-)$.
Applying \eqref{eq:interp2} with $\ell=1$ to $v_{H^2}$ and using \eqref{eq:decHFTR} with $|\alpha|=2$, we obtain that 
\beq\label{eq:interp3}
\min_{w_N \in V_N}\N{v_{H_2}- w_N}_{H^1_k(B_R)} \leq \Cint(1,d) \big(1+hk\big)\, h\, C_1 \N{f}_{L^2(B_R)}.
\eeq

Let $\Csob(p,d)$ be such that
\beqs
\tif \quad\N{\partial^\alpha v}_{L^2}\leq C \quad \tfa \alpha \text{ with } |\alpha|\leq p, \quad\tthen 
\N{v}_{H^{p+1}}\leq \Csob(p,d)\, C;
\eeqs
i.e., $\Csob$ depends only on the normalisations in the definition of $\|\cdot\|_{H^{p+1}}$.

The regularity assumptions on $\obstacle_-,A,$ and $c$ and the regularity results of, e.g., \cite[Theorem 4.20]{Mc:00}, \cite[Theorem 5.2.1, Part (i)]{CoDaNi:10} imply that $u_{\pm, \cA} \in H^{p+1}$ for $p$ odd and $H^{p+2}$ for $p$ even.
For $p$ odd we apply Theorem \ref{thm:transmission} with $m= (p+1)/2$ and for $p$ even with $m=(p+2)/2$. In both cases, 
we apply \eqref{eq:interp2} with $\ell=p$ to $v_{\cA}=(v_{\cA,+},v_{\cA,-})$ and use \eqref{eq:decLFTR} with $|\alpha|=p+1$ to obtain that
\beq\label{eq:interp4}
\min_{w_N \in V_N}\N{v_{\cA}- w_N}_{H^1_k(B_R)} \leq \Cint(p) \big(1+hk\big) \,h^p\, \Csob(p,d)\,C_2(p)\, k^{p+M}\N{f}_{L^2(B_R)}.
\eeq
The bound on $\eta(V_N)$ in \eqref{eq:etaboundT} then follows from combining \eqref{eq:interp3} and \eqref{eq:interp4}, with $\widetilde{\mathcal{C}}_1:= \Cint(1,d) C_1$ and $\widetilde{\mathcal{C}}_2:= \Cint(p,d) \Csob(p,d) C_2$.
 
\subsection{Proof of Corollary \ref{cor:1}}

If $u$ is the solution of the plane-wave scattering problem, then
\beq\label{eq:Cosc}
|u|_{H^2(B_R)} \leq \Cosc k \N{u}_{H^1_k(B_R)}
\eeq
by \cite[Theorem 9.1 and Remark 9.10]{LaSpWu:19}, where $\Cosc$ depends on $\MA,c,d,$ and $R$, but is independent of $k$. 
The combination of \eqref{eq:Cosc} and \eqref{eq:bae1} then imply that 
\beq\label{eq:MSbae}
\min_{v_N\in V_N} \N{u-v_N}_{H^1_k(B_R)}\leq \mathcal{C}_3 \Cosc \frac{hk}{p} \left( 1 + \frac{hk}{p}\right)\N{u}_{H^1_k(B_R)}.
\eeq
Combining \eqref{eq:qo}, \eqref{eq:MSbae}, and \eqref{eq:thresholdD}, we obtain the result \eqref{eq:rel_error} with $C_6:= \mathcal{C}_3 \Cosc$.
 
\appendix

\section{Semiclassical pseudodifferential operators on the torus} \label{app:sct}

Recall that for $R_\sharp >0$ we defined the torus
$$
 \mathbb T^d_{R_{\sharp}} := {\mathbb{R}^d}/{(2R_\sharp\mathbb{Z})^d}.
 $$
This appendix reviews the material about semiclassical pseudodifferential operators on $\mathbb T^d_{R_{\sharp}}$
used in \S \ref{subsec:high}, and appearing in Lemma  \ref{thm:funcloc2}, with our default references being  \cite{Zw:12} and \cite[Appendix E]{DyZw:19}. 

\paragraph{Semiclassical Sobolev spaces.} 
We consider functions or distributions on the torus as periodic
functions or distributions on $\mathbb R^d.$  To eliminate
confusion between Fourier series and integrals, for $f \in
L^2(\torus_{R_\sharp}^d)$  we define the Fourier coefficients
$$
\widehat f(j) := \int_{\torus_{R_\sharp}^d} f(x) \overline{e_j}(x) \, dx,
$$
where $j \in \ZZ^d$ and the integral
is over the cube of side $2 R_\sharp$, and where the Fourier
basis given by the $L^2$-normalized functions
\begin{equation} \label{eq:A:eigen}
e_j(x) = (2R_\sharp)^{-d/2} \exp{\big(\ri\pi j\cdot  x/R_\sharp\big)}
\end{equation}
for $j \in \ZZ^d$.
The Fourier inversion formula is then
$$
f = \sum_{j \in \ZZ^d} \widehat f(j) e_j.
$$
The action of the operator $(\hsc D)^\alpha$ on the torus is therefore
$$
(\hsc D)^\alpha f  =\sum_{j \in \ZZ^d} (\hsc j)^\alpha \widehat f(j) e_j.
$$
We work on the spaces defined by the boundedness of these operators, namely
\beqs
H_\hsc ^ m (\mathbb T_{R_\sharp}^d):= \Big\{ u\in L^2(\mathbb T_{R_\sharp}^d),
\; \langle j \rangle^m \widehat f(j) \in \ell^2(\ZZ^d) \Big\},
\eeqs
and use the norm
\beq\label{eq:Hhnorm}
\Vert u \Vert_{H_\hsc^m(\mathbb T_{R_\sharp}^d)} ^2 :=\sum \lvert \widehat f(j)\rvert^2
\langle \hsc j\rangle^{2m};
\eeq
see \cite[\S8.3]{Zw:12}, \cite[\S E.1.8]{DyZw:19}. 
In this appendix, we abbreviate $H_\hsc ^ m (\mathbb T_{R_\sharp}^d)$ to $H_\hsc ^ m$ and $L^2(\mathbb T_{R_\sharp}^d)$ to $L^2$.

Since  these spaces are defined for positive integer $m$ by boundedness of $(hD)^\alpha$ with
$\lvert \alpha \rvert =m$ (and can be extended to $m \in \RR$ by
interpolation and duality), they agree with localized versions of the
corresponding spaces on $\mathbb R^d$ defined by semiclassical Fourier transform %:
%let the semiclassical Fourier transform $\mathcal F_{\hsc}$ (see \cite[\S3.3]{Zw:12}) on the torus be defined for $\hsc>0$ by
$$
\mathcal F_{\hsc}u(\xi) := \int_{\mathbb R^d} \exp\big( -\ri x \cdot \xi/\hsc\big)
u(x) \, \rd x,
$$
and %for a function on $\mathbb R^d,$ we set
$$
\Vert u \Vert_{H_\hsc^m(\mathbb R^d)} ^2 := (2\pi \hsc)^{-d} \int_{\mathbb R^d} \langle \xi \rangle^m  |\mathcal F_\hsc u(\xi)|^2 \, \rd \xi.
$$
We note for later use that the inverse semiclassical Fourier transform has a pre-factor of $(2 \pi \hsc)^{-d}$ in this normalisation. 

\paragraph{Phase space.}
The set of all possible positions $x$ and momenta (i.e.~Fourier variables) $\xi$ is denoted by $T^*\mathbb T_{R_\sharp}^d$; this is known informally as ``phase space". Strictly, $T^*\mathbb T_{R_\sharp}^d :=\mathbb T_{R_\sharp}^d\times (\mathbb R^d)^*$, but 
for our purposes, we can consider $T^*\mathbb T_{R_\sharp}^d$ as $\{(x,\xi)
: \bx\in \mathbb T^d_{R_\sharp}, \xi\in\Rea^d\}$.  We also use the analogous
notation for $T^* \mathbb R^d$ where appropriate.

To deal uniformly near fiber-infinity with the behavior of
functions on phase space, we also consider the \emph{radial
  compactification} in the fibers of this space,
$$
\Tbar^* \torus^d_{R_\sharp}:= \mathbb R^d \times B^d,
$$
where $B^d$ denotes the closed unit ball, considered as the closure of the
image of $\mathbb R^d$ under the radial compactification map 
$$\RC: \xi \mapsto \xi/(1+\langle \xi
\rangle);$$
see \cite[\S E.1.3]{DyZw:19}.
Near the boundary of the
ball, $\lvert \xi\rvert^{-1}\circ \RC^{-1}$ is a smooth function, vanishing to
first order at the boundary, with $(\lvert \xi\rvert^{-1}\circ \RC^{-1}, \widehat\xi\circ\RC^{-1})$
thus furnishing local coordinates on the ball near its boundary.  The boundary of the
ball should be considered as a sphere at infinity consisting of all
possible \emph{directions} of the momentum variable.  Where
appropriate (e.g., in dealing with finite values of $\xi$ only), we abuse notation by dropping the composition with $\RC$ from our
notation and simply identifying $\mathbb R^d$ with the interior of $B^d$.

\paragraph{Symbols, quantisation, and semiclassical pseudodifferential operators.} 

A symbol on $\mathbb R^d$ is a function on $T^*\mathbb{R}^d$ that is also allowed to depend on $\hsc$, and thus can be considered as an $\hsc$-dependent family of functions.
Such a family $a=(a_\hsc)_{0<\hsc\leq\hsc_0}$, with $a_\hsc \in C^\infty({\mathbb R^d})$, 
is a \emph{symbol
of order $m$} on the $\mathbb R^d$, written as $a\in S^m(\mathbb R^d)$, if for any multi-indices $\alpha, \beta$
\beqs
| \partial_x^\alpha \partial^\beta_\xi a(x,\xi) | \leq C_{\alpha, \beta} \langle \xi \rangle^{m -|\beta|} \quad\tfa (x,\xi) \in T^*\mathbb R^d \text{ and for all } 0<\hsc\leq \hsc_0,
\eeqs
where $C_{\alpha, \beta}$ does not depend on $\hsc$; see \cite[p.~207]{Zw:12}, \cite[\S E.1.2]{DyZw:19}. 

For $a \in S^m(\mathbb R^d)$, we define the \emph{semiclassical quantisation} of $a$ on $\mathbb R^d$, denoted by $\operatorname{Op}_{\hsc}(a)$
\beq \label{Oph}
\big(\operatorname{Op}_{\hsc}(a) v\big)(x) := (2\pi \hsc)^{-d} \int_{\xi \in\mathbb R^d} \int_{y \in \Rea^d} 
\exp\big(\ri (x-y)\cdot\xi/\hsc\big)\,
a(x,\xi) v(y) \,\rd y  \rd \xi;
\eeq
\cite[\S4.1]{Zw:12} \cite[Page 543]{DyZw:19}. The integral in
\eqref{Oph} need not converge, and can be understood \emph{either} as
an oscillatory integral in the sense of \cite[\S3.6]{Zw:12},
\cite[\S7.8]{Ho:83}, \emph{or} as an iterated integral, with the $y$
integration performed first; see \cite[Page 543]{DyZw:19}.  It can be
shown that for any symbol $a,$ $\Op_\hsc(a)$ preserves Schwartz
functions, and extends by duality to act on tempered distributions
\cite[\S4.4]{Zw:12}

We use below that if $a=a(\xi)$
depends only on $\xi$, then
$$
\Op_{\hsc}(a) = \F_\hsc^{-1} M_a \F_{\hsc},
$$
where $M_a$ denotes multiplication by $a$; i.e., in this case $\Op_{\hsc}(a)$
is simply a Fourier multiplier on $\mathbb R^d.$

We now return to considering the torus:~if $a(x,\xi) \in S^m(\mathbb R^d)$ and
is periodic, and if $v$ is a distribution on the torus, we can view
$v$ as a periodic (hence, tempered) distribution on $\mathbb R^d,$ and
define
$$
\big(\operatorname{Op}^{\torus^d_{R_\sharp}}_{\hsc}(a) v\big)=\big(\operatorname{Op}_{\hsc}(a) v\big),
$$
since the right side is again periodic; for details see, e.g., \cite[\S 5.3.1]{Zw:12}.  

If $A$ can be written in the form above, i.\,e.\ $A = \operatorname{Op}^{\torus^d_{R_\sharp}}_{\hsc}(a)$ with $a\in S^m$, we say that $A$ is a \emph{semiclassical pseudodifferential operator of order $m$} on the torus and
we write $A \in \Psi_{\hsc}^m(\mathbb T_{R_\sharp}^d)$; furthermore that we often abbreviate $ \Psi_{\hsc}^m(\mathbb T_{R_\sharp}^d)$ to $\Psi_{\hsc}^m$ in this Appendix. We use the notation $a \in \hsc^l S^m$  if $\hsc^{-l} a \in S^m$; similarly 
$A \in \hsc^l \Psi_\hsc^m$ if 
$\hsc^{-l}A \in \Psi_\hsc^m$.

\begin{theorem}\mythmname{Composition and mapping properties of
semiclassical pseudodifferential operators \cite[Theorem 8.10]{Zw:12}, \cite[Proposition E.17 and Proposition E.19]{DyZw:19}}\label{thm:basicP} If $A\in \Psi_{\hsc}^{m_1}$ and $B  \in \Psi_{\hsc}^{m_2}$, then
\begin{itemize}
\item[(i)]  $AB \in \Psi_{\hsc}^{m_1+m_2}$,
\item[(ii)]  $[A,B] \in \hsc\Psi_{\hsc}^{m_1+m_2-1}$,
\item[(iii)]  For any $s \in \mathbb R$, $A$ is bounded uniformly in $\hsc$ as an operator from $H_\hsc^s$ to $H_\hsc^{s-m_1}$.
\end {itemize}
\end{theorem}

\paragraph{Residual class.} 
We say that $A =O(\hsc^\infty)_{\Psi^{-\infty}_\hsc}$ if, for any $s>0$ and $N\geq 1$, there exists $C_{s,N}>0$ such that
\beq \label{eq:residual}
\Vert A \Vert_{H_\hsc^{-s} \rightarrow H_\hsc^{s}} \leq C_{N,s} \hsc^N;
\eeq
i.e.~$A\in \Psi_\hsc^{-\infty}$ and furthermore all of its operator norms are bounded by any algebraic power of $\hsc$.

\paragraph{Principal symbol $\sigma_{\hsc}$.}
Let the quotient space $ S^m/\hsc S^{m-1}$ be defined by identifying elements 
of  $S^m$ that differ only by an element of $\hsc S^{m-1}$. 
For any $m$, there is a linear, surjective map
$$
\sigma^m_{\hsc}:\Psi_\hsc ^m \to S^m/\hsc S^{m-1},
$$
called the \emph{principal symbol map}, 
such that, for $a\in S^m$,
\beq\label{eq:symbolone}
\sigma_\hsc^m\big(\Op^{\torus^d_{R_\sharp}}_\hsc(a)\big) = a \quad\text{ mod } \hsc S^{m-1};
\eeq
see \cite[Page 213]{Zw:12}, \cite[Proposition E.14]{DyZw:19} (observe that \eqref{eq:symbolone} implies that 
$\operatorname{ker}(\sigma^m_{\hsc}) = \hsc\Psi_\hsc ^{m-1}$).

When applying the map $\sigma^m_{\hsc}$ to 
elements of $\Psi^m_\hsc$, we denote it by $\sigma_{\hsc}$ (i.e.~we omit the $m$ dependence) and we use $\sigma_{\hsc}(A)$ to denote one of the representatives
in $S^m$ (with the results we use then independent of the choice of representative).
%Two key properties of the principal symbol that we use in \S \ref{subsec:high} are
%\beqs
%\sigma_{\hsc}(AB)=\sigma_{\hsc}(A)\sigma_{\hsc}(B),
%\eeqs
%\beqs
%\sigma_{\hsc}(-\hsc^2\Delta)=  |\xi|^2.
%\eeqs

\paragraph{Operator wavefront set $\operatorname{WF}_{\hsc}$.}  
We say that $(x_0,\zeta_0) \in {\Tbar}^*\mathbb T_{R_\sharp}^d$ is \emph{not}
in the \emph{semiclassical operator wavefront set} of
$A = \Optorus(a) \in \Psi_{\hsc}^m$, denoted by
$\operatorname{WF}_{\hsc} A$, if there exists a neighbourhood $U$ of
$(x_0,\zeta_0)$ such that for all multi-indices $\alpha, \beta$ and all
$N\geq 1$ there exists $C_{\alpha,\beta,U,N}>0$ (independent of
$\hsc$) such that, for all $0<\hsc\leq \hsc_0$,
\beq\label{eq:microsupport} |\partial_x^\alpha \partial_\xi^\beta
a(x,\xi)| \leq C_{\alpha, \beta, U, N} \hsc^N \langle \xi \rangle^{-N}\quad\tfa (x,\RC(\xi))\in U.
\eeq
For $\zeta_0=\RC(\xi_0)$ in the interior of $B^d,$ the factor $\ang{\xi}^{-N}$ is moot, and the definition
merely says that
outside its semiclassical operator wavefront set an operator
is the quantization of a symbol that vanishes faster than any algebraic power of $\hsc$; see \cite[Page
194]{Zw:12}, \cite[Definition E.27]{DyZw:19}.  For $\zeta_0\in \partial
B^d=S^{d-1},$ by contrast, the definition says that the symbol decays rapidly
in a conic neighborhood of the direction $\zeta_0,$ in addition to decaying in $\hsc.$

Three properties of the
semiclassical operator wavefront set that we use in \S
\ref{subsec:high} are 
\begin{equation} \label{eq:WFempt}
\operatorname{WF}_\hsc A = \emptyset \quad\text{ if and only if }\quad A = O(\hsc^\infty)_{\Psi^{-\infty}_\hsc},
\end{equation}
(see  \cite[E.2.3]{DyZw:19}),
\begin{equation} \label{eq:WFsum}
\operatorname{WF}_\hsc (A+B) \subset \operatorname{WF}_\hsc A \cup \operatorname{WF}_\hsc B,
\end{equation}
(see  \cite[E.2.4]{DyZw:19}),
\beq \label{eq:WFprod} \operatorname{WF}_{\hsc}
(AB) \subset \operatorname{WF}_{\hsc} A \cap
\operatorname{WF}_{\hsc}B, \eeq (see \cite[\S 8.4]{Zw:12},
\cite[E.2.5]{DyZw:19}), \beq \label{eq:WFdis} \operatorname{WF}_{\hsc}
(A) \cap \operatorname{WF}_{\hsc} (B) = \emptyset \quad\text{ implies that } \quad AB =
O(\hsc^\infty)_{\Psi^{-\infty}_\hsc}, \eeq 
 (as a consequence of \eqref{eq:WFempt} and \eqref{eq:WFprod}), and 
 %(by, e.g., (\ref{eq:WFprod})
%together with \cite[E.2.3]{DyZw:19}), 
and 
\beq\label{eq:support}
\operatorname{WF}_{\hsc}\big( \Op_\hsc(a)\big) \subset \supp \,a \eeq
(since
$(\supp \, a)^c \subset (\operatorname{WF}_{\hsc}( \Op_\hsc(a)))^c$ by
\eqref{eq:microsupport}).

\paragraph{Ellipticity.} 
We  say that $B\in \Psi_\hsc^m$ is \emph{elliptic} at $(x_0,
  \zeta_0) \in  \Tbar^*\mathbb T_{R_\sharp}^d$ if there exists a
  neighborhood $U$ of $(x_0, \zeta_0)$ and $c>0$, independent of $\hsc$, such that 
\beqs
\langle \xi \rangle^{-m} \big|\sigma_\hsc(B)(x,\xi)\big| \geq c \quad \tfa (x,\RC(\xi))\in U  \text{ and for all } 0<\hsc\leq \hsc_0.
\eeqs

A key feature of elliptic operators is that they are microlocally
invertible; this is reflected in the following result.
\newcounter{fnnumber}
\begin{theorem}\mythmname{Elliptic parametrix \cite[Proposition E.32]{DyZw:19}} \footnote{We highlight that working in a compact manifold
allows us to dispense with the proper-support assumption appearing in \cite[\S4]{LaSpWu:22}, \cite[Proposition E.32, Theorem E.33]{DyZw:19}.\label{fn:comp}} \setcounter{fnnumber}{\thefootnote}
 \label{thm:para} 
Let $A \in \Psi_\hsc^{\ell}(\mathbb T_{R_\sharp}^d)$ and $B \in \Psi_\hsc^{m}(\mathbb T_{R_\sharp}^d)$ be such 
that $B$ is elliptic on $\operatorname{WF_\hsc}(A)$.
Then there exist $S, S' \in \Psi_\hsc^{\ell-m}(\mathbb T_{R_\sharp}^d)$ such that
$$
A = BS + O(\hsc^\infty)_{\Psi^{-\infty}_\hsc} = S'B + O(\hsc^\infty)_{\Psi^{-\infty}_\hsc},
$$ 
with
$$
\operatorname{WF}_\hsc S \subset \operatorname{WF}_\hsc A , \qquad \operatorname{WF}_\hsc S' \subset \operatorname{WF}_\hsc A.
$$
\end{theorem}

\paragraph{Functional Calculus.}

The main properties of the functional calculus in the black-box
  context are recalled in
  \S\ref{BBFC}; here we record a simple result that we need about
  functions of the flat Laplacian.
  
For $f$ a Borel function, the operator $f(-\hsc^2 \Lap)$ is defined on
smooth functions on the torus (and indeed on distributions if $f$ has
polynomial growth) by the functional calculus for the flat Laplacian,
i.e., by the Fourier multiplier
\begin{equation} \label{eq:A:mult}
f(-\hsc^2 \Lap) v=\sum_{j \in \ZZ^d} \widehat v(j) f(\hsc^2 \lvert j\rvert ^2 \pi^2/R_{\sharp}^2) e_j.
\end{equation}
It is reassuring to discover that indeed it is precisely
the quantization of $f(\lvert\xi\rvert^2).$  Since our quantization
procedure was defined in terms of Fourier transform rather than
Fourier series, this is not obvious a priori.

\begin{lemma}\label{lem:toruscalculus}
  For $f \in S^m(\RR^1)$ (i.e., $f$ is a function of only one variable),
  $$
f(-\hsc^2 \Lap) = \Op_{\hsc} f(\lvert\xi\rvert^2).
  $$
  \end{lemma}
  \begin{proof}
    First note that for $v \in \CI(\torus^d_{R_\sharp}),$
    \begin{align}\nonumber
      v = \sum \widehat v(j) e_j
      &= (2 R_{\sharp})^{-d/2}\int_{\mathbb R^d} \sum_{j \in \ZZ^d} \widehat v(j)
      \delta(\xi-\hsc \pi j/R_{\sharp})
      \exp(\ri\xi x/\hsc) \, \rd\xi\\
      &=(2\pi\hsc)^d (2R_{\sharp})^{-d/2} \F_\hsc ^{-1}  \sum_{j \in \ZZ^d} \widehat v(j)
      \delta(\xi-\hsc \pi j/R_{\sharp}).\label{eq:vFourier}
\end{align}
Thus, if we take the semiclassical Fourier transform of $v,$ regarded
as a periodic function, 
$$
\F_\hsc v (\xi) = (2\pi\hsc)^d (2R_{\sharp})^{-d/2}  \sum_{j \in \ZZ^d} \widehat v(j)
      \delta(\xi-\hsc \pi j/R_{\sharp}).
      $$
      Consequently,
      $$
      \begin{aligned}
\F_\hsc \big[ f(-\hsc^2 \Lap) v\big] (\xi) &= (2\pi\hsc)^d
(2R_{\sharp})^{-d/2}  \sum_{j \in \ZZ^d}  f(\hsc^2 \pi^2 \lvert j \rvert^2 /R_{\sharp}^2) \widehat v(j)
\delta(\xi-\hsc \pi j/R_{\sharp})\\
&= (2\pi\hsc)^d
(2R_{\sharp})^{-d/2}  \sum_{j \in \ZZ^d}  f(\lvert \xi \rvert^2) \widehat v(j)
\delta(\xi-\hsc \pi j/R_{\sharp})\\
&=f(\lvert \xi \rvert^2) \F_\hsc[v](\xi),
      \end{aligned}
      $$
      by \eqref{eq:vFourier},
      from which 
      $$
f(-\hsc^2 \Lap) v = \Op_\hsc f(\lvert \xi \rvert^2) (v).
      $$
\end{proof}

\section{Proof of \eqref{eq:gro} for the transmission problem}\label{app:gro}

By the min-max principle for self-adjoint operators with
compact resolvent (see, e.g., \cite[Page 76, Theorem 13.1]{ReSi:78})
\beq\label{eq:app0}
\lambda_{n}  = \inf_{X\in\Phi_{n}(\mathcal{D}^\sharp)}\sup_{u\in X}
\frac{\langle P^\#u, u \rangle_{\beta, c}}
{
\N{u_+}^2_{L^2(\mathbb{T}^d_{R_\sharp}\setminus\obstacle_-)} + \beta^{-1} \N{u_-/c}^2_{L^2(\obstacle_-)}
},
\eeq
where $(\lambda_{n})_{n\geq1}$ denotes the ordered eigenvalues of
$P^\#$, $\mathcal{D}^\sharp$ is the domain of $P^{\#}$ defined by \eqref{eq:defdomsharp} (with $\cD$ given by \eqref{eq:domain_transmission}),  $\Phi_{n}(\mathcal{D}^\sharp)$ the set of all $n$-dimensional
subspaces of $\mathcal{D}^\sharp$, and $\langle\cdot,\cdot \rangle_{\beta, c}$
 is the scalar product defined implicitly by the norm in the denominator (which is the norm in Lemma \ref{lem:transmission}).
 
By Green's identity and the definition of $\cD^\sharp$, 
\beq\label{eq:app1}
\langle P^\#u, u \rangle_{\beta, c} = \hsc^2\langle A_+ \nabla u_+, \nabla u_+\rangle_{L^2(\mathbb{T}^d_{R_\sharp}\setminus\obstacle_-)} + \beta^{-1}\hsc^2 \langle A_- \nabla u_-, \nabla u_-\rangle_{L^2(\obstacle_-)}.
\eeq
Furthermore, 
\beq\label{eq:app2}
\frac{
\langle A_+ \nabla u_+, \nabla u_+\rangle_{L^2(\mathbb{T}^d_{R_\sharp}\setminus\obstacle_-)} + \beta^{-1} \langle A_- \nabla u_-, \nabla u_-\rangle_{L^2(\obstacle_-)}
}{
\N{u_+}^2_{L^2(\mathbb{T}^d_{R_\sharp}\setminus\obstacle_-)} + \beta^{-1} \N{u_-/c}^2_{L^2(\obstacle_-)}
}
\geq 
\frac{\min \big( (A_+)_{\min}, \, \beta^{-1} (A_-)_{\min}\big)
}{
\max \big( 1, \, \beta^{-1} (c_{\min})^{-2}\big)
}
\frac{
\N{\nabla u}^2_{L^2(\mathbb{T}^d_{R_\sharp})}
}{
\N{u}^2_{L^2(\mathbb{T}^d_{R_\sharp})}
}.
\eeq
The definition of $\cD^\sharp$ implies that
\beq\label{eq:app3}
\cD^\sharp \subset  \big\{(u_{1}, u_{2}) \in H^{1}(\mathbb{T}^d_{R_\sharp}\backslash\obstacle_-)\oplus H^{1}(\obstacle_-) \tst u_{1}=u_{2}\text{ on }\partial\obstacle_-\big\} = H^1(\mathbb{T}^d_{R_\sharp}).
\eeq
Using \eqref{eq:app1}, \eqref{eq:app2}, and \eqref{eq:app3} in \eqref{eq:app0}, we have 
\beqs
\lambda_{n}  \geq 
\frac{\min \big( (A_+)_{\min}, \, \beta^{-1} (A_-)_{\min}\big)
}{
\max \big( 1, \, \beta^{-1} (c_{\min})^{-2}\big)
}
\left(\inf_{X\in\Phi_{n}(H^1(\mathbb{T}^d_{R_\sharp}))}\sup_{u\in X}
\frac{
\hsc^2\N{\nabla u}^2_{L^2(\mathbb{T}^d_{R_\sharp})}
}{
\N{u}^2_{L^2(\mathbb{T}^d_{R_\sharp})}
}\right).
\eeqs
The result then follows from the min-max principle for the eigenvalues of the Laplacian on the torus.

\section*{Acknowledgements}

The authors thank Luis Escauriaza (Universidad del Pa\'{i}s Vasco/Euskal Herriko Unibertsitatea) for useful discussions related to the paper \cite{EsMoZh:17}. We also thank the referees and the associate editor for their constructive comments that improved the organisation of the paper. JG acknowledges support from EPSRC grant EP/V001760/1. 
DL and EAS acknowledge support from EPSRC grant EP/1025995/1. JW was
partly supported by Simons Foundation grant 631302 and by NSF grant DMS--2054424.

\footnotesize{
\bibliographystyle{plain}
\bibliography{../biblio_combined_sncwadditions}

\begin{thebibliography}{10}

\bibitem{AzKeSt:88}
A.~K. Aziz, R.~B. Kellogg, and A.~B. Stephens.
\newblock A two point boundary value problem with a rapidly oscillating
  solution.
\newblock {\em Numer. Math.}, 53(1-2):107--121, 1988.

\bibitem{BaSa:00}
I.~M. Babu\v{s}ka and S.~A. Sauter.
\newblock Is the pollution effect of the {FEM} avoidable for the {H}elmholtz
  equation considering high wave numbers?
\newblock {\em SIAM Review}, pages 451--484, 2000.

\bibitem{BaChGo:17}
H.~Barucq, T.~Chaumont-Frelet, and C.~Gout.
\newblock {Stability analysis of heterogeneous Helmholtz problems and finite
  element solution based on propagation media approximation}.
\newblock {\em Math. Comp.}, 86(307):2129--2157, 2017.

\bibitem{Be:89}
C.~Bernardi.
\newblock Optimal finite-element interpolation on curved domains.
\newblock {\em SIAM J. Numer. Anal.}, 26(5):1212--1240, 1989.

\bibitem{BeMe:18}
M.~Bernkopf and J.~M. Melenk.
\newblock {Analysis of the $ hp $-version of a first order system least squares
  method for the Helmholtz equation}.
\newblock In {\em Advanced Finite Element Methods with Applications: Selected
  Papers from the 30th Chemnitz Finite Element Symposium 2017}, pages 57--84.
  Springer International Publishing, 2019.

\bibitem{Bu:98}
N.~Burq.
\newblock D\'ecroissance des ondes absence de de l'\'energie locale de
  l'\'equation pour le probl\`{e}me ext\'erieur et absence de resonance au
  voisinage du r\'eel.
\newblock {\em Acta Math.}, 180:1--29, 1998.

\bibitem{Bu:02}
N.~Burq.
\newblock Semi-classical estimates for the resolvent in nontrapping geometries.
\newblock {\em International Mathematics Research Notices}, 2002(5):221--241,
  2002.

\bibitem{CaPoVo:99}
F.~Cardoso, G.~Popov, and G.~Vodev.
\newblock Distribution of resonances and local energy decay in the transmission
  problem {II}.
\newblock {\em Mathematical Research Letters}, 6:377--396, 1999.

\bibitem{ChGrLaSp:12}
S.~N. Chandler-Wilde, I.~G. Graham, S.~Langdon, and E.~A. Spence.
\newblock Numerical-asymptotic boundary integral methods in high-frequency
  acoustic scattering.
\newblock {\em Acta Numerica}, 21(1):89--305, 2012.

\bibitem{ChHeMo:15}
S.~N. Chandler-Wilde, D.~P. Hewett, and A.~Moiola.
\newblock Interpolation of {H}ilbert and {S}obolev spaces: quantitative
  estimates and counterexamples.
\newblock {\em Mathematika}, 61:414--443, 2015.

\bibitem{Ch:16}
T.~Chaumont-Frelet.
\newblock {On high order methods for the heterogeneous Helmholtz equation}.
\newblock {\em Computers \& Mathematics with Applications}, 72(9):2203--2225,
  2016.

\bibitem{ChNi:18}
T.~Chaumont-Frelet and S.~Nicaise.
\newblock {High-frequency behaviour of corner singularities in Helmholtz
  problems}.
\newblock {\em ESAIM: Math. Model. Numer. Anal.}, 52(5):1803--1845, 2018.

\bibitem{ChNi:20}
T.~Chaumont-Frelet and S.~Nicaise.
\newblock {Wavenumber explicit convergence analysis for finite element
  discretizations of general wave propagation problem}.
\newblock {\em IMA J. Numer. Anal.}, 40(2):1503--1543, 2020.

\bibitem{CoDaNi:10}
M.~Costabel, M.~Dauge, and S.~Nicaise.
\newblock {Corner Singularities and Analytic Regularity for Linear Elliptic
  Systems. Part I: Smooth domains.}
\newblock 2010.
\newblock
  \url{https://hal.archives-ouvertes.fr/file/index/docid/453934/filename/CoDaNi_Analytic_Part_I.pdf}.

\bibitem{Da:95}
E.~B. Davies.
\newblock The functional calculus.
\newblock {\em J. London Math. Soc. (2)}, 52(1):166--176, 1995.

\bibitem{Da:96}
E.~B. Davies.
\newblock {\em Spectral theory and differential operators}.
\newblock Cambridge University Press, 1995.

\bibitem{DiSj:99}
M.~Dimassi and J.~Sj{\"o}strand.
\newblock {\em Spectral asymptotics in the semi-classical limit}, volume 268 of
  {\em London Mathematical Society Lecture Note Series}.
\newblock Cambridge University Press, Cambridge, 1999.

\bibitem{DuWu:15}
Y.~Du and H.~Wu.
\newblock {Preasymptotic error analysis of higher order FEM and CIP-FEM for
  Helmholtz equation with high wave number}.
\newblock {\em SIAM J. Numer. Anal.}, 53(2):782--804, 2015.

\bibitem{DuZh:16}
Y.~Du and L.~Zhu.
\newblock {Preasymptotic error analysis of high order interior penalty
  discontinuous Galerkin methods for the Helmholtz equation with high wave
  number}.
\newblock {\em J. Sci. Comp.}, 67(1):130--152, 2016.

\bibitem{DyZw:19}
S.~Dyatlov and M.~Zworski.
\newblock {\em Mathematical theory of scattering resonances}, volume 200 of
  {\em Graduate Studies in Mathematics}.
\newblock American Mathematical Society, 2019.

\bibitem{EsMoZh:17}
L.~Escauriaza, S.~Montaner, and C.~Zhang.
\newblock Analyticity of solutions to parabolic evolutions and applications.
\newblock {\em SIAM J. Math. Anal.}, 49(5):4064--4092, 2017.

\bibitem{EsMe:12}
S.~Esterhazy and J.~M. Melenk.
\newblock On stability of discretizations of the {H}elmholtz equation.
\newblock In I.~G. Graham, T.~Y. Hou, O.~Lakkis, and R.~Scheichl, editors, {\em
  Numerical Analysis of Multiscale Problems}, pages 285--324. Springer, 2012.

\bibitem{EsMe:14}
S.~Esterhazy and J.~M. Melenk.
\newblock {An analysis of discretizations of the Helmholtz equation in L2 and
  in negative norms}.
\newblock {\em Computers \& Mathematics with Applications}, 67(4):830--853,
  2014.

\bibitem{Fr:69}
A.~Friedman.
\newblock {\em Partial differential equations}.
\newblock Holt, Rinehart and Winston, Inc., New York-Montreal, Que.-London,
  1969.

\bibitem{GaLaSp:21}
J.~Galkowski, D.~Lafontaine, and E.~A. Spence.
\newblock {Local absorbing boundary conditions on fixed domains give order-one
  errors for high-frequency waves}.
\newblock {\em arXiv preprint}, 2021.
\newblock \url{https://arxiv.org/abs/2101.02154}.

\bibitem{GaSpWu:20}
J.~Galkowski, E.~A. Spence, and J.~Wunsch.
\newblock {Optimal constants in nontrapping resolvent estimates}.
\newblock {\em Pure and Applied Analysis}, 2(1):157--202, 2020.

\bibitem{GaChNiTo:18}
D.~Gallistl, T.~Chaumont-Frelet, S.~Nicaise, and J.~Tomezyk.
\newblock Wavenumber explicit convergence analysis for finite element
  discretizations of time-harmonic wave propagation problems with perfectly
  matched layers.
\newblock {\em hal preprint 01887267}, 2018.

\bibitem{GaMo:19}
M.~Ganesh and C.~Morgenstern.
\newblock {A coercive heterogeneous media Helmholtz model: formulation,
  wavenumber-explicit analysis, and preconditioned high-order FEM}.
\newblock {\em Numerical Algorithms}, pages 1--47, 2019.

\bibitem{GoGrSp:20}
S.~Gong, I.~G. Graham, and E.~A. Spence.
\newblock {Domain decomposition preconditioners for high-order discretisations
  of the heterogeneous Helmholtz equation}.
\newblock {\em IMA J. Num. Anal.}, 41(3):2139--2185, 2021.

\bibitem{GrPeSp:19}
I.~G. Graham, O.~R. Pembery, and E.~A. Spence.
\newblock {The Helmholtz equation in heterogeneous media: a priori bounds,
  well-posedness, and resonances}.
\newblock {\em Journal of Differential Equations}, 266(6):2869--2923, 2019.

\bibitem{GrSa:20}
I.~G. Graham and S.~A. Sauter.
\newblock {Stability and finite element error analysis for the Helmholtz
  equation with variable coefficients}.
\newblock {\em Math. Comp.}, 89(321):105--138, 2020.

\bibitem{HeRo:83}
B.~Helffer and D.~Robert.
\newblock Calcul fonctionnel par la transformation de {M}ellin et
  op\'{e}rateurs admissibles.
\newblock {\em J. Funct. Anal.}, 53(3):246--268, 1983.

\bibitem{HeSj:89}
B.~Helffer and J.~Sj\"{o}strand.
\newblock \'{E}quation de {S}chr\"{o}dinger avec champ magn\'{e}tique et
  \'{e}quation de {H}arper.
\newblock In {\em Schr\"{o}dinger operators ({S}\o nderborg, 1988)}, volume 345
  of {\em Lecture Notes in Phys.}, pages 118--197. Springer, Berlin, 1989.

\bibitem{Ho:83}
L~H\"{o}rmander.
\newblock {\em The Analysis of Linear Differential Operators. I, Distribution
  Theory and Fourier Analysis}.
\newblock Springer-Verlag, Berlin, 1983.

\bibitem{Ih:98}
F.~Ihlenburg.
\newblock {\em Finite element analysis of acoustic scattering}.
\newblock Springer Verlag, 1998.

\bibitem{IhBa:95a}
F.~Ihlenburg and I.~Babu{\v{s}}ka.
\newblock {Finite element solution of the Helmholtz equation with high wave
  number Part I: The h-version of the FEM}.
\newblock {\em Comput. Math. Appl.}, 30(9):9--37, 1995.

\bibitem{IhBa:97}
F.~Ihlenburg and I.~Babuska.
\newblock {Finite element solution of the Helmholtz equation with high wave
  number part II: the $hp$ version of the FEM}.
\newblock {\em SIAM J. Numer. Anal.}, 34(1):315--358, 1997.

\bibitem{LaSpWu:20}
D.~Lafontaine, E.~A. Spence, and J.~Wunsch.
\newblock For most frequencies, strong trapping has a weak effect in
  frequency-domain scattering.
\newblock {\em Communications on Pure and Applied Mathematics},
  74(10):2025--2063, 2021.

\bibitem{LaSpWu:19}
D.~Lafontaine, E.~A. Spence, and J.~Wunsch.
\newblock {A sharp relative-error bound for the Helmholtz $h$-FEM at high
  frequency}.
\newblock {\em Numerische Mathematik}, 150:137--178, 2022.

\bibitem{LaSpWu:22}
D.~Lafontaine, E.~A. Spence, and J.~Wunsch.
\newblock {Wavenumber-explicit convergence of the hp-FEM for the full-space
  heterogeneous Helmholtz equation with smooth coefficients}.
\newblock {\em Comp. Math. Appl.}, 113:59--69, 2022.

\bibitem{LaxPhi}
P.~D. Lax and R.~S. Phillips.
\newblock {\em {Scattering Theory}}.
\newblock Academic Press, revised edition, 1989.

\bibitem{LiWu:19}
Y.~Li and H.~Wu.
\newblock {FEM and CIP-FEM for Helmholtz Equation with High Wave Number and
  Perfectly Matched Layer Truncation}.
\newblock {\em SIAM J. Numer. Anal.}, 57(1):96--126, 2019.

\bibitem{LoMe:11}
M.~L\"{o}hndorf and J.~M. Melenk.
\newblock {Wavenumber-Explicit $hp$-{BEM} for High Frequency Scattering}.
\newblock {\em {SIAM J. Numer. Anal.}}, 49(6):2340--2363, 2011.

\bibitem{Mc:00}
W.~McLean.
\newblock {\em {Strongly elliptic systems and boundary integral equations}}.
\newblock Cambridge University Press, 2000.

\bibitem{Me:95}
J.~M. Melenk.
\newblock {\em {On generalized finite element methods}}.
\newblock PhD thesis, The University of Maryland, 1995.

\bibitem{Me:12}
J.~M. Melenk.
\newblock {Mapping properties of combined field Helmholtz boundary integral
  operators}.
\newblock {\em {SIAM Journal on Mathematical Analysis}}, 44(4):2599--2636,
  2012.

\bibitem{MePaSa:13}
J.~M. Melenk, A.~Parsania, and S.~Sauter.
\newblock {General DG-methods for highly indefinite Helmholtz problems}.
\newblock {\em Journal of Scientific Computing}, 57(3):536--581, 2013.

\bibitem{MeSa:10}
J.~M. Melenk and S.~Sauter.
\newblock Convergence analysis for finite element discretizations of the
  {H}elmholtz equation with {D}irichlet-to-{N}eumann boundary conditions.
\newblock {\em Math. Comp}, 79(272):1871--1914, 2010.

\bibitem{MeSa:11}
J.~M. Melenk and S.~Sauter.
\newblock Wavenumber explicit convergence analysis for {G}alerkin
  discretizations of the {H}elmholtz equation.
\newblock {\em SIAM J. Numer. Anal.}, 49:1210--1243, 2011.

\bibitem{MeSa:20}
J.~M. Melenk and S.~A. Sauter.
\newblock {Wavenumber-explicit hp-FEM analysis for Maxwell's equations with
  transparent boundary conditions}.
\newblock {\em Foundations of Computational Mathematics}, pages 1--117, 2020.

\bibitem{MeSaTo:20}
J.~M. Melenk, S.~A. Sauter, and C.~Torres.
\newblock {Wavenumber Explicit Analysis for Galerkin Discretizations of Lossy
  Helmholtz Problems}.
\newblock {\em SIAM J. Numer. Anal.}, 58(4):2119--2143, 2020.

\bibitem{MeSj:82}
R.~B. Melrose and J.~Sj{\"o}strand.
\newblock {Singularities of boundary value problems. II}.
\newblock {\em Communications on Pure and Applied Mathematics}, 35(2):129--168,
  1982.

\bibitem{MoSp:19}
A.~Moiola and E.~A. Spence.
\newblock Acoustic transmission problems: wavenumber-explicit bounds and
  resonance-free regions.
\newblock {\em Mathematical Models \& Methods in Applied Sciences},
  29(2):317--354, 2019.

\bibitem{NiTo:20}
S.~Nicaise and J.~Tomezyk.
\newblock {Convergence analysis of a hp-finite element approximation of the
  time-harmonic Maxwell equations with impedance boundary conditions in domains
  with an analytic boundary}.
\newblock {\em Numerical Methods for Partial Differential Equations},
  36(6):1868--1903, 2020.

\bibitem{Pe:20}
O.~R. Pembery.
\newblock {\em {The Helmholtz Equation in Heterogeneous and Random Media:
  Analysis and Numerics}}.
\newblock PhD thesis, University of Bath, 2020.

\bibitem{Ra:71}
J.~V. Ralston.
\newblock Trapped rays in spherically symmetric media and poles of the
  scattering matrix.
\newblock {\em Communications on Pure and Applied Mathematics}, 24(4):571--582,
  1971.

\bibitem{ReeSim72}
M.~Reed and B.~Simon.
\newblock {\em {Methods of Modern Mathematical Physics Volume 1: Functional
  analysis}}.
\newblock New York-London: Academic Press, Inc, 1972.

\bibitem{ReSi:78}
M.~Reed and B.~Simon.
\newblock {\em {Methods of Modern Mathematical Physics Volume 4: Analysis of
  Operators}}.
\newblock Academic Press, 1978.

\bibitem{Rob87}
D.~Robert.
\newblock {\em Autour de l'approximation semi-classique}, volume~68 of {\em
  Progress in Mathematics}.
\newblock Birkh\"{a}user Boston, Inc., Boston, MA, 1987.

\bibitem{Sa:06}
S.~A. Sauter.
\newblock {A refined finite element convergence theory for highly indefinite
  Helmholtz problems}.
\newblock {\em Computing}, 78(2):101--115, 2006.

\bibitem{Sc:74}
A.~H. Schatz.
\newblock {An observation concerning Ritz-Galerkin methods with indefinite
  bilinear forms}.
\newblock {\em Math. Comp.}, 28(128):959--962, 1974.

\bibitem{Sj:97}
J.~Sj\"{o}strand.
\newblock A trace formula and review of some estimates for resonances.
\newblock In {\em Microlocal analysis and spectral theory ({L}ucca, 1996)},
  volume 490 of {\em NATO Adv. Sci. Inst. Ser. C Math. Phys. Sci.}, pages
  377--437. Kluwer Acad. Publ., Dordrecht, 1997.

\bibitem{SjZw:91}
J.~Sj\"{o}strand and M.~Zworski.
\newblock Complex scaling and the distribution of scattering poles.
\newblock {\em J. Amer. Math. Soc.}, 4(4):729--769, 1991.

\bibitem{Sp:15}
E.~A. Spence.
\newblock {Overview of Variational Formulations for Linear Elliptic PDEs}.
\newblock In A.~S. Fokas and B.~Pelloni, editors, {\em Unified transform method
  for boundary value problems: applications and advances}, pages 93--159. SIAM,
  2015.

\bibitem{Sp:22}
E.~A. Spence.
\newblock {A simple proof that the $hp$-FEM does not suffer from the pollution
  effect for the constant-coefficient full-space Helmholtz equation}.
\newblock {\em arXiv preprint arXiv:2202.06939}, 2022.

\bibitem{St:01}
P.~Stefanov.
\newblock {Resonance expansions and Rayleigh waves}.
\newblock {\em Mathematical Research Letters}, 8(2):107--124, 2001.

\bibitem{St:72}
R.~S. Strichartz.
\newblock A functional calculus for elliptic pseudo-differential operators.
\newblock {\em Amer. J. Math.}, 94:711--722, 1972.

\bibitem{Va:75}
B.~R. Vainberg.
\newblock {On the short wave asymptotic behaviour of solutions of stationary
  problems and the asymptotic behaviour as $t\rightarrow \infty$ of solutions
  of non-stationary problems}.
\newblock {\em Russian Mathematical Surveys}, 30(2):1--58, 1975.

\bibitem{Va:89}
B.~R. Vainberg.
\newblock {\em Asymptotic methods in equations of mathematical physics}.
\newblock Gordon \& Breach Science Publishers, New York, 1989.
\newblock Translated from the Russian by E. Primrose.

\bibitem{Wu:14}
H.~Wu.
\newblock {Pre-asymptotic error analysis of CIP-FEM and FEM for the Helmholtz
  equation with high wave number. Part I: linear version}.
\newblock {\em IMA J. Numer. Anal.}, 34(3):1266--1288, 2014.

\bibitem{ZhDu:15}
L.~Zhu and Y.~Du.
\newblock {Pre-asymptotic error analysis of hp-interior penalty discontinuous
  Galerkin methods for the Helmholtz equation with large wave number}.
\newblock {\em Computers \& Mathematics with Applications}, 70(5):917--933,
  2015.

\bibitem{ZhWu:13}
L.~Zhu and H.~Wu.
\newblock {Preasymptotic error analysis of CIP-FEM and FEM for Helmholtz
  equation with high wave number. Part II: $hp$ version}.
\newblock {\em SIAM J. Numer. Anal.}, 51(3):1828--1852, 2013.

\bibitem{Zw:12}
M.~Zworski.
\newblock {\em Semiclassical analysis}, volume 138 of {\em Graduate Studies in
  Mathematics}.
\newblock American Mathematical Society, Providence, RI, 2012.

\end{thebibliography}
}

\end{document}